\documentclass[10pt, reqno]{amsart}
\usepackage{amsmath,amssymb,amscd,mathrsfs,amscd, mdwlist}
\usepackage{graphics,verbatim, tikz}
\allowdisplaybreaks

\usetikzlibrary{matrix,arrows,patterns, snakes}
\usepackage{amsmath,amstext,amsbsy,amssymb,amscd}
\usepackage{amsmath}
\usepackage{amsxtra}
\usepackage{amscd}
\usepackage{amsthm}
\usepackage{amsfonts}
\usepackage{graphicx}%
\usepackage{amssymb}
\usepackage{eucal}
\usepackage{color}
\usepackage{multirow}
\usepackage[enableskew,vcentermath]{youngtab}

\newcommand{\red}[1]{{\color{red}#1}}

\newtheorem{thm}{Theorem}[section]
\theoremstyle{plain}
\newtheorem*{thm*}{{\bf Theorem}}
\newtheorem{lem}[thm]{Lemma}
\newtheorem{prop}[thm]{Proposition}

\newtheorem{cor}[thm]{Corollary}

\theoremstyle{definition}
\newtheorem{dfn}[thm]{Definition}
\newtheorem{example}[thm]{Example}

\newtheorem{rmk}[thm]{Remark}

\theoremstyle{remark}
\newtheorem{rem}[thm]{Remark}
\definecolor{A}{rgb}{.75,1,.75}

\numberwithin{equation}{section}

\newcommand{\cB}{{\mathcal{B}}}

\newcommand{\BB}{{\mathbf{B}}}

\newcommand{\catO}{{\mathcal O}}

\newcommand{\cR}{{\mathcal{R}}}

\newcommand{\Q}{{\mathbb{Q}}}
\newcommand{\N}{{\mathbb{N}}}
\newcommand{\Z}{{\mathbb{Z}}}
\newcommand{\Qq}{{\Q(q)}}

\newcommand{\Qqp}{{\Qq^{\pi}}}

\newcommand{\Qqtp}{{\Q(q,\bt)^{\pi}}}

\newcommand{\ff}{\ensuremath{\mathbf{f}}}

\newcommand{\UU}{{\mathbf U}}

\newcommand{\hatU}{{\widehat \UU}}
\newcommand{\Up}{\UU^+}
\newcommand{\Um}{\UU^-}
\newcommand{\Uz}{\UU^0}

\newcommand{\osp}{{\mathfrak{osp}}}
\newcommand{\fsl}{{\mathfrak{sl}}}
\newcommand{\so}{{\mathfrak{so}}}

\newcommand{\ir}{{{}_i r}}

\newcommand{\tK}{{\tilde{K}}}

\newcommand{\tJ}{{\tilde{J}}}

\newcommand{\bt}{{\mathbf t}}
\newcommand{\tT}{{\tilde{T}}}
\newcommand{\var}{{\varepsilon}}

\newcommand{\bp}{{\mathbf p}}

\newcommand{\height}{{\mathrm{ht}}}
\newcommand{\Tw}{{\mathfrak X}}

\newcommand{\Hom}{{\rm Hom}}

\newcommand{\barmap}{{\bar{\phantom{x}}}}

\newcommand{\set}[1]{\left\{#1\right\}}
\newcommand{\parens}[1]{\left(#1\right)}
\newcommand{\ang}[1]{\left\langle#1\right\rangle}
\newcommand{\bra}[1]{\left[#1\right]}

\newcommand{\bbinom}[2]{{\begin{bmatrix}#1 \\ #2\end{bmatrix}}}
\newcommand{\abs}[1]{\left|#1\right|}
\newcommand{\norm}[1]{\abs{\abs{#1}}}

\renewcommand{\bar}[1]{\overline{#1}}

\DeclareMathOperator{\ev}{{\rm ev}}
\DeclareMathOperator{\qtr}{{\rm qtr}}
\DeclareMathOperator{\coev}{{\rm coev}}
\DeclareMathOperator{\coqtr}{{\rm coqtr}}

\newcommand{\frf}{\mathfrak{f}}
\newcommand{\frr}{\mathfrak{r}}
\newcommand{\frl}{\mathfrak{l}}
\newcommand{\frF}{\mathfrak{F}}
\newcommand{\frs}{\mathfrak{s}}

\usepackage[all,knot]{xy}
\xyoption{arc}

\newcommand{\hctikz}[1]{\raisebox{-.5\height}{\begin{tikzpicture}#1\end{tikzpicture}}}

\usetikzlibrary{decorations.markings}
\tikzset{->->-/.style={decoration={
  markings,
  mark=at position .25 with {\arrow[scale=1.5,black]{>}}, mark=at position .85 with {\arrow[scale=1.5,black]{>}}},postaction={decorate}}}
\tikzset{->-/.style={decoration={
  markings,
  mark=at position .55 with {\arrow[scale=1.5,black]{>}}},postaction={decorate}}}
\tikzset{->/.style={decoration={
  markings,
  mark=at position .99 with {\arrow[scale=1.5,black]{>}}},postaction={decorate}}}
\tikzset{<-/.style={decoration={
  markings,
  mark=at position .10 with {\arrow[scale=1.5,black]{<}}},postaction={decorate}}}
\tikzset{-<-/.style={decoration={
  markings,
  mark=at position .55 with {\arrow[scale=1.5,black]{<}}},postaction={decorate}}}
\tikzset{-<-<-/.style={decoration={
  markings,
  mark=at position .25 with {\arrow[scale=1.5,black]{<}}, mark=at position .85 with {\arrow[scale=1.5,black]{<}}},postaction={decorate}}}

\newcommand{\nesc}[4]{
\pgfmathsetmacro{\xscale}{{#3}}
\pgfmathsetmacro{\yscale}{{#4}}
\pgfmathsetmacro{\xcenter}{{#1}}
\pgfmathsetmacro{\ycenter}{{#2}}
	\pgfmathparse{-.5*\xscale+\xcenter}		\let\Xa\pgfmathresult
    \pgfmathparse{-1*\yscale+\ycenter}		\let\Ya\pgfmathresult
    \coordinate (A) at (\Xa,\Ya);
	\pgfmathparse{-.5*\xscale+\xcenter}		\let\Xb\pgfmathresult
    \pgfmathparse{0+\ycenter}		\let\Yb\pgfmathresult
    \coordinate (B) at (\Xb,\Yb);
	\pgfmathparse{-.25*\xscale+\xcenter}		\let\Xc\pgfmathresult
    \pgfmathparse{.5*\yscale+\ycenter}		\let\Yc\pgfmathresult
    \coordinate (C) at (\Xc,\Yc);
	\pgfmathparse{0+\xcenter}		\let\Xd\pgfmathresult
    \pgfmathparse{0+\ycenter}		\let\Yd\pgfmathresult
    \coordinate (D) at (\Xd,\Yd);
	\pgfmathparse{.25*\xscale+\xcenter}		\let\Xe\pgfmathresult
    \pgfmathparse{-.5*\yscale+\ycenter}		\let\Ye\pgfmathresult
    \coordinate (E) at (\Xe,\Ye);
	\pgfmathparse{.5*\xscale+\xcenter}		\let\Xf\pgfmathresult
    \pgfmathparse{0*\yscale+\ycenter}		\let\Yf\pgfmathresult
    \coordinate (F) at (\Xf,\Yf);
	\pgfmathparse{.5*\xscale+\xcenter}		\let\Xg\pgfmathresult
    \pgfmathparse{1*\yscale+\ycenter}		\let\Yg\pgfmathresult
    \coordinate (G) at (\Xg,\Yg);
\draw[line width=1pt] (A) .. controls (B) and (C) .. (D) .. controls (E) and (F)  .. (G);
}
\newcommand{\dcap}[4]{
\pgfmathsetmacro{\xscale}{{#3}}
\pgfmathsetmacro{\yscale}{{#4}}
\pgfmathsetmacro{\xleft}{{#1}}
\pgfmathsetmacro{\ybottom}{{#2}}
	\pgfmathparse{\xleft}		\let\Xa\pgfmathresult
    \pgfmathparse{\ybottom}		\let\Ya\pgfmathresult
    \coordinate (A) at (\Xa,\Ya);
	\pgfmathparse{\xleft}		\let\Xb\pgfmathresult
    \pgfmathparse{\ybottom+\yscale}		\let\Yb\pgfmathresult
    \coordinate (B) at (\Xb,\Yb);
	\pgfmathparse{\xleft+\xscale}		\let\Xc\pgfmathresult
    \pgfmathparse{\ybottom+\yscale}		\let\Yc\pgfmathresult
    \coordinate (C) at (\Xc,\Yc);
	\pgfmathparse{\xleft+\xscale}		\let\Xd\pgfmathresult
    \pgfmathparse{\ybottom}		\let\Yd\pgfmathresult
    \coordinate (D) at (\Xd,\Yd);
\draw[line width=1pt] (A) .. controls (B) and (C) .. (D);
}
\newcommand{\cwcap}[4]{
\pgfmathsetmacro{\xscale}{{#3}}
\pgfmathsetmacro{\yscale}{{#4}}
\pgfmathsetmacro{\xleft}{{#1}}
\pgfmathsetmacro{\ybottom}{{#2}}
	\pgfmathparse{\xleft}		\let\Xa\pgfmathresult
    \pgfmathparse{\ybottom}		\let\Ya\pgfmathresult
    \coordinate (A) at (\Xa,\Ya);
	\pgfmathparse{\xleft}		\let\Xb\pgfmathresult
    \pgfmathparse{\ybottom+\yscale}		\let\Yb\pgfmathresult
    \coordinate (B) at (\Xb,\Yb);
	\pgfmathparse{\xleft+\xscale}		\let\Xc\pgfmathresult
    \pgfmathparse{\ybottom+\yscale}		\let\Yc\pgfmathresult
    \coordinate (C) at (\Xc,\Yc);
	\pgfmathparse{\xleft+\xscale}		\let\Xd\pgfmathresult
    \pgfmathparse{\ybottom}		\let\Yd\pgfmathresult
    \coordinate (D) at (\Xd,\Yd);
\draw[line width=1pt, ->-] (A) .. controls (B) and (C) .. (D);
}
\newcommand{\ccwcap}[4]{
\pgfmathsetmacro{\xscale}{{#3}}
\pgfmathsetmacro{\yscale}{{#4}}
\pgfmathsetmacro{\xleft}{{#1}}
\pgfmathsetmacro{\ybottom}{{#2}}
	\pgfmathparse{\xleft}		\let\Xa\pgfmathresult
    \pgfmathparse{\ybottom}		\let\Ya\pgfmathresult
    \coordinate (A) at (\Xa,\Ya);
	\pgfmathparse{\xleft}		\let\Xb\pgfmathresult
    \pgfmathparse{\ybottom+\yscale}		\let\Yb\pgfmathresult
    \coordinate (B) at (\Xb,\Yb);
	\pgfmathparse{\xleft+\xscale}		\let\Xc\pgfmathresult
    \pgfmathparse{\ybottom+\yscale}		\let\Yc\pgfmathresult
    \coordinate (C) at (\Xc,\Yc);
	\pgfmathparse{\xleft+\xscale}		\let\Xd\pgfmathresult
    \pgfmathparse{\ybottom}		\let\Yd\pgfmathresult
    \coordinate (D) at (\Xd,\Yd);
\draw[line width=1pt, -<-] (A) .. controls (B) and (C) .. (D);
}
\newcommand{\dcup}[4]{
\pgfmathsetmacro{\xscale}{{#3}}
\pgfmathsetmacro{\yscale}{{#4}}
\pgfmathsetmacro{\xleft}{{#1}}
\pgfmathsetmacro{\ybottom}{{#2}}
	\pgfmathparse{\xleft}		\let\Xa\pgfmathresult
    \pgfmathparse{\ybottom+\yscale}		\let\Ya\pgfmathresult
    \coordinate (A) at (\Xa,\Ya);
	\pgfmathparse{\xleft}		\let\Xb\pgfmathresult
    \pgfmathparse{\ybottom}		\let\Yb\pgfmathresult
    \coordinate (B) at (\Xb,\Yb);
	\pgfmathparse{\xleft+\xscale}		\let\Xc\pgfmathresult
    \pgfmathparse{\ybottom}		\let\Yc\pgfmathresult
    \coordinate (C) at (\Xc,\Yc);
	\pgfmathparse{\xleft+\xscale}		\let\Xd\pgfmathresult
    \pgfmathparse{\ybottom+\yscale}		\let\Yd\pgfmathresult
    \coordinate (D) at (\Xd,\Yd);
\draw[line width=1pt] (A) .. controls (B) and (C) .. (D);
}
\newcommand{\cwcup}[4]{
\pgfmathsetmacro{\xscale}{{#3}}
\pgfmathsetmacro{\yscale}{{#4}}
\pgfmathsetmacro{\xleft}{{#1}}
\pgfmathsetmacro{\ybottom}{{#2}}
	\pgfmathparse{\xleft}		\let\Xa\pgfmathresult
    \pgfmathparse{\ybottom+\yscale}		\let\Ya\pgfmathresult
    \coordinate (A) at (\Xa,\Ya);
	\pgfmathparse{\xleft}		\let\Xb\pgfmathresult
    \pgfmathparse{\ybottom}		\let\Yb\pgfmathresult
    \coordinate (B) at (\Xb,\Yb);
	\pgfmathparse{\xleft+\xscale}		\let\Xc\pgfmathresult
    \pgfmathparse{\ybottom}		\let\Yc\pgfmathresult
    \coordinate (C) at (\Xc,\Yc);
	\pgfmathparse{\xleft+\xscale}		\let\Xd\pgfmathresult
    \pgfmathparse{\ybottom+\yscale}		\let\Yd\pgfmathresult
    \coordinate (D) at (\Xd,\Yd);
\draw[line width=1pt, -<-] (A) .. controls (B) and (C) .. (D);
}
\newcommand{\ccwcup}[4]{
\pgfmathsetmacro{\xscale}{{#3}}
\pgfmathsetmacro{\yscale}{{#4}}
\pgfmathsetmacro{\xleft}{{#1}}
\pgfmathsetmacro{\ybottom}{{#2}}
	\pgfmathparse{\xleft}		\let\Xa\pgfmathresult
    \pgfmathparse{\ybottom+\yscale}		\let\Ya\pgfmathresult
    \coordinate (A) at (\Xa,\Ya);
	\pgfmathparse{\xleft}		\let\Xb\pgfmathresult
    \pgfmathparse{\ybottom}		\let\Yb\pgfmathresult
    \coordinate (B) at (\Xb,\Yb);
	\pgfmathparse{\xleft+\xscale}		\let\Xc\pgfmathresult
    \pgfmathparse{\ybottom}		\let\Yc\pgfmathresult
    \coordinate (C) at (\Xc,\Yc);
	\pgfmathparse{\xleft+\xscale}		\let\Xd\pgfmathresult
    \pgfmathparse{\ybottom+\yscale}		\let\Yd\pgfmathresult
    \coordinate (D) at (\Xd,\Yd);
\draw[line width=1pt, ->-] (A) .. controls (B) and (C) .. (D);
}

\newcommand{\ids}[5]{
\pgfmathsetmacro{\xl}{{#1}}
\pgfmathsetmacro{\yb}{{#2}}
\pgfmathsetmacro{\num}{{#5}}
\pgfmathsetmacro{\yheight}{{#4}}
\pgfmathsetmacro{\yt}{\yb+\yheight};
\pgfmathsetmacro{\xscale}{{#3}};
	\foreach \x in {1,...,\num}{
		\pgfmathparse{\xl+(\x-1)*\xscale}\let\xc\pgfmathresult;
		\draw[line width=1pt] (\xc,\yb) -- (\xc, \yt);
	}
}
\newcommand{\idsup}[5]{
\pgfmathsetmacro{\xl}{{#1}}
\pgfmathsetmacro{\yb}{{#2}}
\pgfmathsetmacro{\num}{{#5}}
\pgfmathsetmacro{\yheight}{{#4}}
\pgfmathsetmacro{\yt}{\yb+\yheight};
\pgfmathsetmacro{\xscale}{{#3}};
	\foreach \x in {1,...,\num}{
		\pgfmathparse{\xl+(\x-1)*\xscale}\let\xc\pgfmathresult
		\draw[line width=1pt,->-] (\xc,\yb) -- (\xc, \yt);
	}
}
\newcommand{\idsdown}[5]{
\pgfmathsetmacro{\xl}{{#1}}
\pgfmathsetmacro{\yb}{{#2}}
\pgfmathsetmacro{\num}{{#5}}
\pgfmathsetmacro{\yheight}{{#4}}
\pgfmathsetmacro{\yt}{\yb+\yheight};
\pgfmathsetmacro{\xscale}{{#3}};
	\foreach \x in {1,...,\num}{
		\pgfmathparse{\xl+(\x-1)*\xscale}\let\xc\pgfmathresult
		\draw[line width=1pt,-<-] (\xc,\yb) -- (\xc, \yt);
	}
}

\newcommand{\rcross}[4]{
\tikzstyle{cross line}=[preaction={draw=white, -,line width=10pt}];
\pgfmathsetmacro{\xl}{{#1}}
\pgfmathsetmacro{\yb}{{#2}}
\pgfmathsetmacro{\yheight}{{#3}}
\pgfmathsetmacro{\xscale}{{#4}}
\pgfmathsetmacro{\yt}{\yb+\yheight};
\pgfmathsetmacro{\xr}{\xscale+\xl};
\pgfmathsetmacro{\yc}{\yb+\yheight/2};
\pgfmathsetmacro{\xc}{\xscale/2+\xl};

	\pgfmathparse{\xl}		\let\Xa\pgfmathresult
    \pgfmathparse{\yt}		\let\Ya\pgfmathresult
    \coordinate (A) at (\Xa,\Ya);
	\pgfmathparse{\xl+(0.3*\xscale)}		\let\Xaa\pgfmathresult
    \pgfmathparse{\yb+(0.9*\yheight)}		\let\Yaa\pgfmathresult
    \coordinate (A') at (\Xaa,\Yaa);
	\pgfmathparse{\xl}		\let\Xb\pgfmathresult
    \pgfmathparse{\yb}		\let\Yb\pgfmathresult
    \coordinate (B) at (\Xb,\Yb);
	\pgfmathparse{\xl+(0.3*\xscale)}		\let\Xbb\pgfmathresult
    \pgfmathparse{\yb+(0.1*\yheight)}		\let\Ybb\pgfmathresult
    \coordinate (B') at (\Xbb,\Ybb);
	\pgfmathparse{\xr}		\let\Xc\pgfmathresult
    \pgfmathparse{\yb}		\let\Yc\pgfmathresult
    \coordinate (C) at (\Xc,\Yc);
	\pgfmathparse{\xl+(0.7*\xscale)}		\let\Xcc\pgfmathresult
    \pgfmathparse{\yb+(0.1*\yheight)}		\let\Ycc\pgfmathresult
    \coordinate (C') at (\Xcc,\Ycc);
	\pgfmathparse{\xr}		\let\Xd\pgfmathresult
    \pgfmathparse{\yt}		\let\Yd\pgfmathresult
    \coordinate (D) at (\Xd,\Yd);
	\pgfmathparse{\xl+(0.7*\xscale)}		\let\Xdd\pgfmathresult
    \pgfmathparse{\yb+(0.9*\yheight)}		\let\Ydd\pgfmathresult
    \coordinate (D') at (\Xdd,\Ydd);
	\pgfmathparse{\xc}		\let\Xe\pgfmathresult
    \pgfmathparse{\yc}		\let\Ye\pgfmathresult
    \coordinate (E) at (\Xe,\Ye);

\draw[line width=1pt] (A) .. controls (B') and (D') .. (C);
\draw (\xc, \yc) node[shape=circle, fill=white] {};
\draw[line width=1pt] (B) .. controls (A') and (C') .. (D);
}

\newcommand{\rcrossup}[4]{
\tikzstyle{cross line}=[preaction={draw=white, -,line width=10pt}];
\pgfmathsetmacro{\xl}{{#1}}
\pgfmathsetmacro{\yb}{{#2}}
\pgfmathsetmacro{\yheight}{{#3}}
\pgfmathsetmacro{\xscale}{{#4}}
\pgfmathsetmacro{\yt}{\yb+\yheight};
\pgfmathsetmacro{\xr}{\xscale+\xl};
\pgfmathsetmacro{\yc}{\yb+\yheight/2};
\pgfmathsetmacro{\xc}{\xscale/2+\xl};

	\pgfmathparse{\xl}		\let\Xa\pgfmathresult
    \pgfmathparse{\yt}		\let\Ya\pgfmathresult
    \coordinate (A) at (\Xa,\Ya);
	\pgfmathparse{\xl+(0.3*\xscale)}		\let\Xaa\pgfmathresult
    \pgfmathparse{\yb+(0.9*\yheight)}		\let\Yaa\pgfmathresult
    \coordinate (A') at (\Xaa,\Yaa);
	\pgfmathparse{\xl}		\let\Xb\pgfmathresult
    \pgfmathparse{\yb}		\let\Yb\pgfmathresult
    \coordinate (B) at (\Xb,\Yb);
	\pgfmathparse{\xl+(0.3*\xscale)}		\let\Xbb\pgfmathresult
    \pgfmathparse{\yb+(0.1*\yheight)}		\let\Ybb\pgfmathresult
    \coordinate (B') at (\Xbb,\Ybb);
	\pgfmathparse{\xr}		\let\Xc\pgfmathresult
    \pgfmathparse{\yb}		\let\Yc\pgfmathresult
    \coordinate (C) at (\Xc,\Yc);
	\pgfmathparse{\xl+(0.7*\xscale)}		\let\Xcc\pgfmathresult
    \pgfmathparse{\yb+(0.1*\yheight)}		\let\Ycc\pgfmathresult
    \coordinate (C') at (\Xcc,\Ycc);
	\pgfmathparse{\xr}		\let\Xd\pgfmathresult
    \pgfmathparse{\yt}		\let\Yd\pgfmathresult
    \coordinate (D) at (\Xd,\Yd);
	\pgfmathparse{\xl+(0.7*\xscale)}		\let\Xdd\pgfmathresult
    \pgfmathparse{\yb+(0.9*\yheight)}		\let\Ydd\pgfmathresult
    \coordinate (D') at (\Xdd,\Ydd);
	\pgfmathparse{\xc}		\let\Xe\pgfmathresult
    \pgfmathparse{\yc}		\let\Ye\pgfmathresult
    \coordinate (E) at (\Xe,\Ye);

\draw[line width=1pt, <-] (A) .. controls (B') and (D') .. (C);
\draw (\xc, \yc) node[shape=circle, fill=white] {};
\draw[line width=1pt, ->] (B) .. controls (A') and (C') .. (D);
}
\newcommand{\NESE}[4]{
\tikzstyle{cross line}=[preaction={draw=white, -,line width=10pt}];
\pgfmathsetmacro{\xl}{{#1}}
\pgfmathsetmacro{\yb}{{#2}}
\pgfmathsetmacro{\yheight}{{#3}}
\pgfmathsetmacro{\xscale}{{#4}}
\pgfmathsetmacro{\yt}{\yb+\yheight};
\pgfmathsetmacro{\xr}{\xscale+\xl};
\pgfmathsetmacro{\yc}{\yb+\yheight/2};
\pgfmathsetmacro{\xc}{\xscale/2+\xl};

	\pgfmathparse{\xl}		\let\Xa\pgfmathresult
    \pgfmathparse{\yt}		\let\Ya\pgfmathresult
    \coordinate (A) at (\Xa,\Ya);
	\pgfmathparse{\xl+(0.3*\xscale)}		\let\Xaa\pgfmathresult
    \pgfmathparse{\yb+(0.9*\yheight)}		\let\Yaa\pgfmathresult
    \coordinate (A') at (\Xaa,\Yaa);
	\pgfmathparse{\xl}		\let\Xb\pgfmathresult
    \pgfmathparse{\yb}		\let\Yb\pgfmathresult
    \coordinate (B) at (\Xb,\Yb);
	\pgfmathparse{\xl+(0.3*\xscale)}		\let\Xbb\pgfmathresult
    \pgfmathparse{\yb+(0.1*\yheight)}		\let\Ybb\pgfmathresult
    \coordinate (B') at (\Xbb,\Ybb);
	\pgfmathparse{\xr}		\let\Xc\pgfmathresult
    \pgfmathparse{\yb}		\let\Yc\pgfmathresult
    \coordinate (C) at (\Xc,\Yc);
	\pgfmathparse{\xl+(0.7*\xscale)}		\let\Xcc\pgfmathresult
    \pgfmathparse{\yb+(0.1*\yheight)}		\let\Ycc\pgfmathresult
    \coordinate (C') at (\Xcc,\Ycc);
	\pgfmathparse{\xr}		\let\Xd\pgfmathresult
    \pgfmathparse{\yt}		\let\Yd\pgfmathresult
    \coordinate (D) at (\Xd,\Yd);
	\pgfmathparse{\xl+(0.7*\xscale)}		\let\Xdd\pgfmathresult
    \pgfmathparse{\yb+(0.9*\yheight)}		\let\Ydd\pgfmathresult
    \coordinate (D') at (\Xdd,\Ydd);
	\pgfmathparse{\xc}		\let\Xe\pgfmathresult
    \pgfmathparse{\yc}		\let\Ye\pgfmathresult
    \coordinate (E) at (\Xe,\Ye);

\draw[line width=1pt, ->] (A) .. controls (B') and (D') .. (C);
\draw (\xc, \yc) node[shape=circle, fill=white] {};
\draw[line width=1pt, ->] (B) .. controls (A') and (C') .. (D);
}
\newcommand{\SWNW}[4]{
\tikzstyle{cross line}=[preaction={draw=white, -,line width=10pt}];
\pgfmathsetmacro{\xl}{{#1}}
\pgfmathsetmacro{\yb}{{#2}}
\pgfmathsetmacro{\yheight}{{#3}}
\pgfmathsetmacro{\xscale}{{#4}}
\pgfmathsetmacro{\yt}{\yb+\yheight};
\pgfmathsetmacro{\xr}{\xscale+\xl};
\pgfmathsetmacro{\yc}{\yb+\yheight/2};
\pgfmathsetmacro{\xc}{\xscale/2+\xl};

	\pgfmathparse{\xl}		\let\Xa\pgfmathresult
    \pgfmathparse{\yt}		\let\Ya\pgfmathresult
    \coordinate (A) at (\Xa,\Ya);
	\pgfmathparse{\xl+(0.3*\xscale)}		\let\Xaa\pgfmathresult
    \pgfmathparse{\yb+(0.9*\yheight)}		\let\Yaa\pgfmathresult
    \coordinate (A') at (\Xaa,\Yaa);
	\pgfmathparse{\xl}		\let\Xb\pgfmathresult
    \pgfmathparse{\yb}		\let\Yb\pgfmathresult
    \coordinate (B) at (\Xb,\Yb);
	\pgfmathparse{\xl+(0.3*\xscale)}		\let\Xbb\pgfmathresult
    \pgfmathparse{\yb+(0.1*\yheight)}		\let\Ybb\pgfmathresult
    \coordinate (B') at (\Xbb,\Ybb);
	\pgfmathparse{\xr}		\let\Xc\pgfmathresult
    \pgfmathparse{\yb}		\let\Yc\pgfmathresult
    \coordinate (C) at (\Xc,\Yc);
	\pgfmathparse{\xl+(0.7*\xscale)}		\let\Xcc\pgfmathresult
    \pgfmathparse{\yb+(0.1*\yheight)}		\let\Ycc\pgfmathresult
    \coordinate (C') at (\Xcc,\Ycc);
	\pgfmathparse{\xr}		\let\Xd\pgfmathresult
    \pgfmathparse{\yt}		\let\Yd\pgfmathresult
    \coordinate (D) at (\Xd,\Yd);
	\pgfmathparse{\xl+(0.7*\xscale)}		\let\Xdd\pgfmathresult
    \pgfmathparse{\yb+(0.9*\yheight)}		\let\Ydd\pgfmathresult
    \coordinate (D') at (\Xdd,\Ydd);
	\pgfmathparse{\xc}		\let\Xe\pgfmathresult
    \pgfmathparse{\yc}		\let\Ye\pgfmathresult
    \coordinate (E) at (\Xe,\Ye);

\draw[line width=1pt, <-] (A) .. controls (B') and (D') .. (C);
\draw (\xc, \yc) node[shape=circle, fill=white] {};
\draw[line width=1pt, <-] (B) .. controls (A') and (C') .. (D);
}

\newcommand{\SENE}[4]{
\tikzstyle{cross line}=[preaction={draw=white, -,line width=10pt}];
\pgfmathsetmacro{\xl}{{#1}}
\pgfmathsetmacro{\yb}{{#2}}
\pgfmathsetmacro{\yheight}{{#3}}
\pgfmathsetmacro{\xscale}{{#4}}
\pgfmathsetmacro{\yt}{\yb+\yheight};
\pgfmathsetmacro{\xr}{\xscale+\xl};
\pgfmathsetmacro{\yc}{\yb+\yheight/2};
\pgfmathsetmacro{\xc}{\xscale/2+\xl};

	\pgfmathparse{\xl}		\let\Xa\pgfmathresult
    \pgfmathparse{\yt}		\let\Ya\pgfmathresult
    \coordinate (A) at (\Xa,\Ya);
	\pgfmathparse{\xl+(0.3*\xscale)}		\let\Xaa\pgfmathresult
    \pgfmathparse{\yb+(0.9*\yheight)}		\let\Yaa\pgfmathresult
    \coordinate (A') at (\Xaa,\Yaa);
	\pgfmathparse{\xl}		\let\Xb\pgfmathresult
    \pgfmathparse{\yb}		\let\Yb\pgfmathresult
    \coordinate (B) at (\Xb,\Yb);
	\pgfmathparse{\xl+(0.3*\xscale)}		\let\Xbb\pgfmathresult
    \pgfmathparse{\yb+(0.1*\yheight)}		\let\Ybb\pgfmathresult
    \coordinate (B') at (\Xbb,\Ybb);
	\pgfmathparse{\xr}		\let\Xc\pgfmathresult
    \pgfmathparse{\yb}		\let\Yc\pgfmathresult
    \coordinate (C) at (\Xc,\Yc);
	\pgfmathparse{\xl+(0.7*\xscale)}		\let\Xcc\pgfmathresult
    \pgfmathparse{\yb+(0.1*\yheight)}		\let\Ycc\pgfmathresult
    \coordinate (C') at (\Xcc,\Ycc);
	\pgfmathparse{\xr}		\let\Xd\pgfmathresult
    \pgfmathparse{\yt}		\let\Yd\pgfmathresult
    \coordinate (D) at (\Xd,\Yd);
	\pgfmathparse{\xl+(0.7*\xscale)}		\let\Xdd\pgfmathresult
    \pgfmathparse{\yb+(0.9*\yheight)}		\let\Ydd\pgfmathresult
    \coordinate (D') at (\Xdd,\Ydd);
	\pgfmathparse{\xc}		\let\Xe\pgfmathresult
    \pgfmathparse{\yc}		\let\Ye\pgfmathresult
    \coordinate (E) at (\Xe,\Ye);

\draw[line width=1pt, ->] (B) .. controls (A') and (C') .. (D);
\draw (\xc, \yc) node[shape=circle, fill=white] {};
\draw[line width=1pt, ->] (A) .. controls (B') and (D') .. (C);
}
\newcommand{\NWSW}[4]{
\tikzstyle{cross line}=[preaction={draw=white, -,line width=10pt}];
\pgfmathsetmacro{\xl}{{#1}}
\pgfmathsetmacro{\yb}{{#2}}
\pgfmathsetmacro{\yheight}{{#3}}
\pgfmathsetmacro{\xscale}{{#4}}
\pgfmathsetmacro{\yt}{\yb+\yheight};
\pgfmathsetmacro{\xr}{\xscale+\xl};
\pgfmathsetmacro{\yc}{\yb+\yheight/2};
\pgfmathsetmacro{\xc}{\xscale/2+\xl};

	\pgfmathparse{\xl}		\let\Xa\pgfmathresult
    \pgfmathparse{\yt}		\let\Ya\pgfmathresult
    \coordinate (A) at (\Xa,\Ya);
	\pgfmathparse{\xl+(0.3*\xscale)}		\let\Xaa\pgfmathresult
    \pgfmathparse{\yb+(0.9*\yheight)}		\let\Yaa\pgfmathresult
    \coordinate (A') at (\Xaa,\Yaa);
	\pgfmathparse{\xl}		\let\Xb\pgfmathresult
    \pgfmathparse{\yb}		\let\Yb\pgfmathresult
    \coordinate (B) at (\Xb,\Yb);
	\pgfmathparse{\xl+(0.3*\xscale)}		\let\Xbb\pgfmathresult
    \pgfmathparse{\yb+(0.1*\yheight)}		\let\Ybb\pgfmathresult
    \coordinate (B') at (\Xbb,\Ybb);
	\pgfmathparse{\xr}		\let\Xc\pgfmathresult
    \pgfmathparse{\yb}		\let\Yc\pgfmathresult
    \coordinate (C) at (\Xc,\Yc);
	\pgfmathparse{\xl+(0.7*\xscale)}		\let\Xcc\pgfmathresult
    \pgfmathparse{\yb+(0.1*\yheight)}		\let\Ycc\pgfmathresult
    \coordinate (C') at (\Xcc,\Ycc);
	\pgfmathparse{\xr}		\let\Xd\pgfmathresult
    \pgfmathparse{\yt}		\let\Yd\pgfmathresult
    \coordinate (D) at (\Xd,\Yd);
	\pgfmathparse{\xl+(0.7*\xscale)}		\let\Xdd\pgfmathresult
    \pgfmathparse{\yb+(0.9*\yheight)}		\let\Ydd\pgfmathresult
    \coordinate (D') at (\Xdd,\Ydd);
	\pgfmathparse{\xc}		\let\Xe\pgfmathresult
    \pgfmathparse{\yc}		\let\Ye\pgfmathresult
    \coordinate (E) at (\Xe,\Ye);

\draw[line width=1pt, <-] (B) .. controls (A') and (C') .. (D);
\draw (\xc, \yc) node[shape=circle, fill=white] {};
\draw[line width=1pt, <-] (A) .. controls (B') and (D') .. (C);
}

\newcommand{\rcrossdown}[4]{
\tikzstyle{cross line}=[preaction={draw=white, -,line width=10pt}];
\pgfmathsetmacro{\xl}{{#1}}
\pgfmathsetmacro{\yb}{{#2}}
\pgfmathsetmacro{\yheight}{{#3}}
\pgfmathsetmacro{\xscale}{{#4}}
\pgfmathsetmacro{\yt}{\yb+\yheight};
\pgfmathsetmacro{\xr}{\xscale+\xl};
\pgfmathsetmacro{\yc}{\yb+\yheight/2};
\pgfmathsetmacro{\xc}{\xscale/2+\xl};

	\pgfmathparse{\xl}		\let\Xa\pgfmathresult
    \pgfmathparse{\yt}		\let\Ya\pgfmathresult
    \coordinate (A) at (\Xa,\Ya);
	\pgfmathparse{\xl+(0.3*\xscale)}		\let\Xaa\pgfmathresult
    \pgfmathparse{\yb+(0.9*\yheight)}		\let\Yaa\pgfmathresult
    \coordinate (A') at (\Xaa,\Yaa);
	\pgfmathparse{\xl}		\let\Xb\pgfmathresult
    \pgfmathparse{\yb}		\let\Yb\pgfmathresult
    \coordinate (B) at (\Xb,\Yb);
	\pgfmathparse{\xl+(0.3*\xscale)}		\let\Xbb\pgfmathresult
    \pgfmathparse{\yb+(0.1*\yheight)}		\let\Ybb\pgfmathresult
    \coordinate (B') at (\Xbb,\Ybb);
	\pgfmathparse{\xr}		\let\Xc\pgfmathresult
    \pgfmathparse{\yb}		\let\Yc\pgfmathresult
    \coordinate (C) at (\Xc,\Yc);
	\pgfmathparse{\xl+(0.7*\xscale)}		\let\Xcc\pgfmathresult
    \pgfmathparse{\yb+(0.1*\yheight)}		\let\Ycc\pgfmathresult
    \coordinate (C') at (\Xcc,\Ycc);
	\pgfmathparse{\xr}		\let\Xd\pgfmathresult
    \pgfmathparse{\yt}		\let\Yd\pgfmathresult
    \coordinate (D) at (\Xd,\Yd);
	\pgfmathparse{\xl+(0.7*\xscale)}		\let\Xdd\pgfmathresult
    \pgfmathparse{\yb+(0.9*\yheight)}		\let\Ydd\pgfmathresult
    \coordinate (D') at (\Xdd,\Ydd);
	\pgfmathparse{\xc}		\let\Xe\pgfmathresult
    \pgfmathparse{\yc}		\let\Ye\pgfmathresult
    \coordinate (E) at (\Xe,\Ye);

\draw[line width=1pt, ->] (A) .. controls (B') and (D') .. (C);
\draw (\xc, \yc) node[shape=circle, fill=white] {};
\draw[line width=1pt, <-] (B) .. controls (A') and (C') .. (D);
}

\newcommand{\lcross}[4]{
\tikzstyle{cross line}=[preaction={draw=white, -,line width=10pt}];
\pgfmathsetmacro{\xl}{{#1}}
\pgfmathsetmacro{\yb}{{#2}}
\pgfmathsetmacro{\yheight}{{#3}}
\pgfmathsetmacro{\xscale}{{#4}}
\pgfmathsetmacro{\yt}{\yb+\yheight};
\pgfmathsetmacro{\xr}{\xscale+\xl};
\pgfmathsetmacro{\yc}{\yb+\yheight/2};
\pgfmathsetmacro{\xc}{\xscale/2+\xl};

	\pgfmathparse{\xl}		\let\Xa\pgfmathresult
    \pgfmathparse{\yt}		\let\Ya\pgfmathresult
    \coordinate (A) at (\Xa,\Ya);
	\pgfmathparse{\xl+(0.3*\xscale)}		\let\Xaa\pgfmathresult
    \pgfmathparse{\yb+(0.9*\yheight)}		\let\Yaa\pgfmathresult
    \coordinate (A') at (\Xaa,\Yaa);
	\pgfmathparse{\xl}		\let\Xb\pgfmathresult
    \pgfmathparse{\yb}		\let\Yb\pgfmathresult
    \coordinate (B) at (\Xb,\Yb);
	\pgfmathparse{\xl+(0.3*\xscale)}		\let\Xbb\pgfmathresult
    \pgfmathparse{\yb+(0.1*\yheight)}		\let\Ybb\pgfmathresult
    \coordinate (B') at (\Xbb,\Ybb);
	\pgfmathparse{\xr}		\let\Xc\pgfmathresult
    \pgfmathparse{\yb}		\let\Yc\pgfmathresult
    \coordinate (C) at (\Xc,\Yc);
	\pgfmathparse{\xl+(0.7*\xscale)}		\let\Xcc\pgfmathresult
    \pgfmathparse{\yb+(0.1*\yheight)}		\let\Ycc\pgfmathresult
    \coordinate (C') at (\Xcc,\Ycc);
	\pgfmathparse{\xr}		\let\Xd\pgfmathresult
    \pgfmathparse{\yt}		\let\Yd\pgfmathresult
    \coordinate (D) at (\Xd,\Yd);
	\pgfmathparse{\xl+(0.7*\xscale)}		\let\Xdd\pgfmathresult
    \pgfmathparse{\yb+(0.9*\yheight)}		\let\Ydd\pgfmathresult
    \coordinate (D') at (\Xdd,\Ydd);
	\pgfmathparse{\xc}		\let\Xe\pgfmathresult
    \pgfmathparse{\yc}		\let\Ye\pgfmathresult
    \coordinate (E) at (\Xe,\Ye);
    
\draw[line width=1pt] (B) .. controls (A') and (C') .. (D);
\draw (\xc, \yc) node[shape=circle, fill=white] {};
\draw[line width=1pt] (A) .. controls (B') and (D') .. (C);
}

\newcommand{\lcrossup}[4]{
\tikzstyle{cross line}=[preaction={draw=white, -,line width=10pt}];
\pgfmathsetmacro{\xl}{{#1}}
\pgfmathsetmacro{\yb}{{#2}}
\pgfmathsetmacro{\yheight}{{#3}}
\pgfmathsetmacro{\xscale}{{#4}}
\pgfmathsetmacro{\yt}{\yb+\yheight};
\pgfmathsetmacro{\xr}{\xscale+\xl};
\pgfmathsetmacro{\yc}{\yb+\yheight/2};
\pgfmathsetmacro{\xc}{\xscale/2+\xl};

	\pgfmathparse{\xl}		\let\Xa\pgfmathresult
    \pgfmathparse{\yt}		\let\Ya\pgfmathresult
    \coordinate (A) at (\Xa,\Ya);
	\pgfmathparse{\xl+(0.3*\xscale)}		\let\Xaa\pgfmathresult
    \pgfmathparse{\yb+(0.9*\yheight)}		\let\Yaa\pgfmathresult
    \coordinate (A') at (\Xaa,\Yaa);
	\pgfmathparse{\xl}		\let\Xb\pgfmathresult
    \pgfmathparse{\yb}		\let\Yb\pgfmathresult
    \coordinate (B) at (\Xb,\Yb);
	\pgfmathparse{\xl+(0.3*\xscale)}		\let\Xbb\pgfmathresult
    \pgfmathparse{\yb+(0.1*\yheight)}		\let\Ybb\pgfmathresult
    \coordinate (B') at (\Xbb,\Ybb);
	\pgfmathparse{\xr}		\let\Xc\pgfmathresult
    \pgfmathparse{\yb}		\let\Yc\pgfmathresult
    \coordinate (C) at (\Xc,\Yc);
	\pgfmathparse{\xl+(0.7*\xscale)}		\let\Xcc\pgfmathresult
    \pgfmathparse{\yb+(0.1*\yheight)}		\let\Ycc\pgfmathresult
    \coordinate (C') at (\Xcc,\Ycc);
	\pgfmathparse{\xr}		\let\Xd\pgfmathresult
    \pgfmathparse{\yt}		\let\Yd\pgfmathresult
    \coordinate (D) at (\Xd,\Yd);
	\pgfmathparse{\xl+(0.7*\xscale)}		\let\Xdd\pgfmathresult
    \pgfmathparse{\yb+(0.9*\yheight)}		\let\Ydd\pgfmathresult
    \coordinate (D') at (\Xdd,\Ydd);
	\pgfmathparse{\xc}		\let\Xe\pgfmathresult
    \pgfmathparse{\yc}		\let\Ye\pgfmathresult
    \coordinate (E) at (\Xe,\Ye);

\draw[line width=1pt, ->] (B) .. controls (A') and (C') .. (D);
\draw (\xc, \yc) node[shape=circle, fill=white] {};
\draw[line width=1pt, <-] (A) .. controls (B') and (D') .. (C);
}

\newcommand{\lcrossdown}[4]{
\tikzstyle{cross line}=[preaction={draw=white, -,line width=10pt}];
\pgfmathsetmacro{\xl}{{#1}}
\pgfmathsetmacro{\yb}{{#2}}
\pgfmathsetmacro{\yheight}{{#3}}
\pgfmathsetmacro{\xscale}{{#4}}
\pgfmathsetmacro{\yt}{\yb+\yheight};
\pgfmathsetmacro{\xr}{\xscale+\xl};
\pgfmathsetmacro{\yc}{\yb+\yheight/2};
\pgfmathsetmacro{\xc}{\xscale/2+\xl};

	\pgfmathparse{\xl}		\let\Xa\pgfmathresult
    \pgfmathparse{\yt}		\let\Ya\pgfmathresult
    \coordinate (A) at (\Xa,\Ya);
	\pgfmathparse{\xl+(0.3*\xscale)}		\let\Xaa\pgfmathresult
    \pgfmathparse{\yb+(0.9*\yheight)}		\let\Yaa\pgfmathresult
    \coordinate (A') at (\Xaa,\Yaa);
	\pgfmathparse{\xl}		\let\Xb\pgfmathresult
    \pgfmathparse{\yb}		\let\Yb\pgfmathresult
    \coordinate (B) at (\Xb,\Yb);
	\pgfmathparse{\xl+(0.3*\xscale)}		\let\Xbb\pgfmathresult
    \pgfmathparse{\yb+(0.1*\yheight)}		\let\Ybb\pgfmathresult
    \coordinate (B') at (\Xbb,\Ybb);
	\pgfmathparse{\xr}		\let\Xc\pgfmathresult
    \pgfmathparse{\yb}		\let\Yc\pgfmathresult
    \coordinate (C) at (\Xc,\Yc);
	\pgfmathparse{\xl+(0.7*\xscale)}		\let\Xcc\pgfmathresult
    \pgfmathparse{\yb+(0.1*\yheight)}		\let\Ycc\pgfmathresult
    \coordinate (C') at (\Xcc,\Ycc);
	\pgfmathparse{\xr}		\let\Xd\pgfmathresult
    \pgfmathparse{\yt}		\let\Yd\pgfmathresult
    \coordinate (D) at (\Xd,\Yd);
	\pgfmathparse{\xl+(0.7*\xscale)}		\let\Xdd\pgfmathresult
    \pgfmathparse{\yb+(0.9*\yheight)}		\let\Ydd\pgfmathresult
    \coordinate (D') at (\Xdd,\Ydd);
	\pgfmathparse{\xc}		\let\Xe\pgfmathresult
    \pgfmathparse{\yc}		\let\Ye\pgfmathresult
    \coordinate (E) at (\Xe,\Ye);

\draw[line width=1pt, <-] (B) .. controls (A') and (C') .. (D);
\draw (\xc, \yc) node[shape=circle, fill=white] {};
\draw[line width=1pt, ->] (A) .. controls (B') and (D') .. (C);
}

\newcommand{\standin}[5]{
\tikzstyle{cross line}=[preaction={draw=white, -,line width=10pt}];
\pgfmathsetmacro{\xl}{{#1}}
\pgfmathsetmacro{\yb}{{#2}}
\pgfmathsetmacro{\yheight}{{#3}}
\pgfmathsetmacro{\xscale}{{#4}}
\pgfmathsetmacro{\yt}{\yb+\yheight};
\pgfmathsetmacro{\xr}{\xscale+\xl};
\pgfmathsetmacro{\yc}{\yb+\yheight/2};
\pgfmathsetmacro{\xc}{\xscale/2+\xl};
\draw[line width=1pt] (\xl,\yb) rectangle (\xr, \yt);
\draw (\xc, \yc) node {\textbf{#5}};
}

\newcommand{\rf}[1]{{\overline{#1}}}
\newcommand{\bigrelt}[1]{{\hat{#1}}}
\newcommand{\bigrset}{{\bigrelt{X}}}
\newcommand{\bigrdeg}[1]{{\norm{#1}}}
\newcommand{\sig}{{\rm sig}}
\newcommand{\Qqt}{{\Q(q,\bt)}}
\newcommand{\Qqtt}{{\Qqt^\tau}}
\newcommand{\hotimes}{{\widehat{\otimes}}}
\newcommand{\hV}{{\widehat{V}}}
\newcommand{\Twnat}{{\Tw^\natural}}
\renewcommand{\wr}{{\rm wr}}

\begin{document}

\title
[]{Quantum $\osp(1|2n)$ knot invariants are the same as \\quantum $\so(2n+1)$ knot invariants}
\author{Sean Clark}
\address{Department of Mathematics\\ Northeastern University\\ Boston, MA 02115\\USA}
\curraddr{Max Planck Institute for Mathematics\\ 53111 Bonn\\ Germany}
\email{se.clark@neu.edu\\se.clark@mpim-bonn.mpg.de}
\begin{abstract}
We show that the quantum covering group associated to
$\osp(1|2n)$ has an associated colored quantum knot invariant 
\`a la Reshetikhin-Turaev, which specializes to a quantum
knot invariant for $\osp(1|2n)$, and to the usual quantum knot
invariant for $\so(1+2n)$.
We then show that these knot invariants are the same, up
to a change of variables and a constant factor depending on the knot
and weight.
\end{abstract}
\maketitle
\section{Introduction}

\subsection{}
Quantum enveloping algebras associated to Kac-Moody Lie algebras are central
objects in mathematics, which have many remarkable connections to geometry, combinatorics,
mathematical physics, and other areas. One such connection was produced by
Reshetikhin and Turaev \cite{T, RT} by relating the representation theory of these quantum enveloping
algebras to Laurent polynomial knot invariants, such as the 
(colored) Jones polynomial and the HOMFLYPT polynomial. 
Many other connections have arisen from
the categorification of quantum enveloping algebras and their representations \cite{KL, R}. 
It was recently shown by Webster \cite{Web} that in fact, one can categorify all Reshetikhin-Turaev
invariants using the machinery of categorified quantum enveloping algebras. This procedure
generalizes Khovanov's homological categorification of the 
Jones polynomial \cite{Kh}. We can summarize some of these connections in the
picture in Figure 1, where ``Decat.'' refers to the appropriate decategorification, ``RT'' stands for the Reshetikhin-Turaev procedure
for constructing the Jones polynomial from the standard quantum $\fsl(2)$
representation, and ``Web'' stands for Webster's categorification of RT
which produces Khovanov homology.

\begin{figure}
\begin{minipage}{.45\textwidth}\label{fig:non-odd}
\centering
\begin{tikzpicture}[scale=.9]
\draw (-1,1) node[left]{KH};
\draw (-1,-1) node[left]{Jones};
\draw (1,1) node[right]{$\dot{\mathcal U_q}(\fsl(2))$};
\draw (1,-1) node[right]{$U_q(\fsl(2))$};
\draw[snake, ->] (1.6,.6)--(1.6,-.6) node[midway,right]{Decat.};
\draw[snake, ->] (-1.6,.6)--(-1.6,-.6) node[midway,right]{Decat.};
\draw[<-] (-1,1)--(1,1) node[midway,above]{Web};
\draw[<-] (-1,-1)--(1,-1) node[midway,above]{RT};
\end{tikzpicture}
\caption{}
\end{minipage}
\begin{minipage}{.45\textwidth}\label{fig:odd}
\centering
\begin{tikzpicture}[scale=.9]
\draw (-1,1) node[left]{oKH};
\draw (-1,-1) node[left]{Jones};
\draw (1,1) node[right]{$\dot{\mathcal U}_{q,\pi}(\osp(1|2))$};
\draw (1,-1) node[right]{$U_{q,\pi}(\osp(1|2))$};
\draw[snake, ->] (1.6,.6)--(1.6,-.6) node[midway,right]{Decat.};
\draw[snake, ->] (-1.6,.6)--(-1.6,-.6) node[midway,right]{Decat.};
\draw[dotted, <-] (-1,1)--(1,1) node[midway,above]{?};
\draw[dotted, <-] (-1,-1)--(1,-1) node[midway,above]{RT?};
\end{tikzpicture}
\caption{}
\end{minipage}
\end{figure}

This beautiful picture recently developed a twist with the discovery of 
``odd Khovanov homology'' \cite{ORS},
an alternate homological categorification of the Jones polynomial. 
This discovery has spurred
a program of ``oddification'': providing analogues of (categorified) 
quantum groups for this odd Khovanov homology
by developing ``odd'' analogues of standard constructions \cite{EKL, EL, MW}.
In particular, one would like an ``odd (categorified) $U_q(\fsl(2))$'' 
which could produce odd Khovanov homology in a similar way to that described
in Figure 1. In particular, the decategorified ``odd'' quantum group 
should produce the Jones polynomial through some analogue of the 
Reshetikhin-Turaev procedure.
It has been proposed \cite{HW,EL} that such categorifications might naturally
arise through categorifying the quantum covering group $U_{q,\pi}(\osp(1|2))$;
in other words, producing a diagram such as in Figure 2. 

This proposal has some heuristic evidence
from the work of Mikhaylov and Witten \cite{MW}, who have produced
candidates for ``odd link homologies'' categorifying $\so(1+2n)$-invariants
via topological quantum field theories using the orthosymplectic supergroups.
This suggests that the conjecture represented by Figure 2 
should be generalized to include colored link
invariants associated to $\osp(1|2n)$ for any $n\geq 1$. 
Moreover, it has been shown by Blumen \cite{Bl} that
the $\osp(1|2n)$ and $\so(2n+1)$ invariants which are colored by the standard
$(2n+1)$-dimensional representations are relattabed up to a 
variable substitution. However, it
has not been known that the Jones polynomial can be constructed from
the Reshetikhin-Turaev procedure on $U_{q,\pi}(\osp(1|2))$, much less
any relation between super and non-super colored knot invariants in 
higher rank.

\subsection{}
A quantum covering group is an algebra $\UU=U_{q,\pi}(\mathfrak{g})$ 
that marries the quantum enveloping superalgebra
of an anisotropic Kac-Moody Lie superalgebra (e.g. $\mathfrak g=\osp(1|2n)$) with the quantum enveloping algebra of 
its associated Kac-Moody Lie algebra, which is obtained by forgetting the parity
in the root datum (e.g. $\so(1+2n)$). This is done by introducing a new ``half-parameter'' $\pi$ satisfying $\pi^2=1$, and substituting $\pi$ everywhere a sign associated to the superalgebra braiding
should appear; such algebras were defined and studied in detail
in the series of papers \cite{CW, CHW1, CHW2, CFLW, C, CH}.

These quantum covering groups retain the many nice properties of usual quantum groups, such
as a Hopf structure; a quasi-$\cR$-matrix \`a la Lusztig \cite[Chapter 4]{L93};
a category $\catO$; and even canonical bases.
A key feature of a quantum covering group is that by specializing $\pi=1$ (respectively, $\pi=-1$),
we obtain the quantum enveloping (super)algebra associated to the Kac-Moody Lie (super)algebra.
Moreover, as discovered in \cite{CFLW}, the quantum algebra and quantum superalgebra can be identified 
by a twistor map; that is, an automorphism of (an extension of) the covering quantum group which sends $\pi\mapsto -\pi$ and $q\mapsto \bt^{-1}q$,
where $\bt^2=-1$. 

In this paper, we use the machinery of covering quantum groups to construct 
``quantum covering knot invariants'':  knot invariants which arise from 
the representation theory of the finite type quantum covering groups  
\`{a} la Turaev \cite{T}. (For our purpose, we do not need the additional
ribbon structure of \cite{RT}.)
To wit, consider the quantum covering group associated to the Lie superalgebra
$\osp(1|2n)$.
We first associate a $\UU$-module homomorphism to each elementary
tangle (cups, caps, crossings) such that a straight strand is just the identity map, 
along with an interpretation
of combining tangles (with joining top-to-bottom 
being composition of the associated maps, 
and placing along-side
being tensor products of the maps). An arbitrary tangle can then be framed and 
associated with a $\UU$-module homomorphism
by ``slicing'' the diagram (that is, cutting it into vertical chunks containing at most one elementary diagram alongside any number of straight strands). Each slice corresponds to a $\UU$-module homomorphism, and
the tangle is sent to the composition of these maps. Note that a priori,
this assignment is not unique, as many distinct slice diagrams and framings
exist for an arbitrary tangle.

We then derive some identities with these maps that
are versions of Turaev moves on the associated diagrams. 
These identities show that the map
isn't dependent on the choice of slice diagram, but factors of $\pi$
keep it from being an invariant of oriented framed tangles.
In order to eliminate these factors, we need
to expand our base ring to $\Qqtt$, where $\tau^2=\pi$,
and renormalize the maps corresponding to certain elementary diagrams.
Finally, a normalization factor (depending on the writhe of the tangle)
yields a oriented tangle invariant (see Theorem \ref{thm:knot invariant}).

In the rank 1 uncolored case, this invariant is simply the (unnormalized)
Jones polynomial in the variable $\tau^{-1}q$ (see Example \ref{ex:rk1knot}).
This suggests that the $\pi=-1$ (i.e. $\tau=\bt$) specialization 
of the knot invariant, viewed as a function of $q$, should be related to the $\pi=1$ (i.e. $\tau=1$)
specialization, viewed as a function of $\bt^{-1}q$.
To make this connection precise, we further develop the theory of twistors
(cf. \cite{CFLW, C}) to define a general operator on tensor powers of $\UU$ and compatible operators on its representations.
In particular, we show that the twistors $\Tw$ on representations 
$\bt$-commute with the maps $S$ representing slices of tangles; that is,
$\Tw\circ S=\bt^{x}S\circ \Tw$ for some $x\in\Z$. 

Once this is done, we obtain the following theorem 
(combining Theorems \ref{thm:knot invariant} and \ref{thm:twistor vs knot invariant}).
\begin{thm*}
Let $K$ be any oriented knot and $\lambda\in X^+$ a dominant weight.
There is a functor from the category $\mathcal{OTAN}$ of oriented tangles 
modulo isotopy to the category $\catO$ of $\UU$-module representations
which sends $K$ to a constant $J_K^\lambda(q,\tau)\in \Qqtt$,
which we call the {\em covering knot invariant} of $K$.
Moreover, let 
${}_{\so}J_K^\lambda(q)=J_K^\lambda(q,1)$ and 
${}_{\osp}J_K^\lambda(q)=J_K^\lambda(q,\bt)$ denote the specializations 
of the covering knot invariant to $\tau=1$ and $\tau=\bt$.
Then 
\[{}_{\osp}J_K^\lambda(q)=\bt^{\star(K,\lambda)}{}_{\so}J_K^\lambda(\bt^{-1} q),\]
for some
$\star(K,\lambda)\in\Z$.
\end{thm*}

In particular, this shows that, after extending scalars, there is indeed
a map RT as in Figure 2, and in fact such a map exists for all colored
link invariants of any rank. It remains to develop an analogue of the
construction in \cite{Web} to complete the picture, 
though difficulties abound. For example, it is not necessarily clear how
to extend the categorification to $\Qqtt$.
Moreover, the categorification of covering algebra representations is not yet
developed enough to produce the analogous machinery to \cite{Web}.
We hope that these results will help cast light on these
remaining questions.

\subsection{} The paper is organized as follows. In Section 2, we recall the definition of quantum covering $\osp(1|2n)$, denoted by $\UU$, and
set our conventions. We also develop some additional facts about representations of $\UU$,
specifically about dual modules and (co)evaluation morphisms, 
and produce a universal-$\cR$-matrix, which we will simply denote by $\cR$, from the quasi-$\cR$-matrix defined in \cite{CHW1}.
In Section 3, we use $\cR$ and the (co)evaluation morphisms to define an associated knot invariant
by interpreting the maps in terms of the usual graphical calculus; that is, maps are represented
by a finite number of labeled, non-intersecting oriented strands such that 
the $\cR$-matrix is a positive crossing, the (co)evaluation morphisms are various cups and caps, 
and orientation is determined by whether the associated module in the domain/range is the dual module or not. We show that this graphical calculus 
is almost an framed oriented tangle invariant, and is indeed
an oriented tangle invariant
after renormalizing these elementary diagrams by an integer power of $\tau$
and a factor depending on the writhe.
Finally, in Section 4 we use the twistor maps
introduced in \cite{CFLW,C} to relate the morphisms in the $\pi=\pm 1$ cases.
In particular, we develop some further details about the Hopf structure and representation theory of the 
enhanced quantum group $\hatU$, and construct twistors on tensor products of simple modules
and their duals. We then show that these twistors almost commute (up to an integer power of $\bt$)
with the cups, caps, and crosssings, allowing us to relate the $\so$ and $\osp$ knot invariants.
\vspace{1em}

\noindent\textbf{Acknowledgements.} I would like to thank
David Hill for first suggesting this project of constructing
the quantum covering knot invariants,
and Matt Hogancamp for helping me learn
about quantum knot invariants at this projects inception.
I would also like to thank 
Aaron Lauda, Weiqiang Wang, and Ben Webster for 
stimulating conversations about this project.

\section{Quantum covering $\osp(1|2n)$}

We begin by recalling the definition of quantum covering
algebra associated to $\osp(1|2n)$ and setting our notations. 
We then elaborate on the representation theory of this algebra.

\subsection{Root data}\label{subsec:rootdata}
Let $I=I_0\coprod I_1$ with $I_0=\set{\rf 1,\ldots, \rf{n-1}}$ and $I_1=\set{\rf n}$, and
define the parity $p(i)$ of $i\in I$ by $i\in I_{p(i)}$.
For $1\leq r,s\leq n$, we define \[{\rf r}\cdot {\rf s}=\begin{cases} 2 &\text{ if } r=s=n\\
4 &\text{ if } r=s\neq n\\
-2&\text{ if }r=s\pm 1\\
0&\text{otherwise}
\end{cases},\qquad d_{\rf r}=\frac{{\rf r}\cdot{\rf r}}{2},\]
and note that $p({\rf r})\equiv d_{\rf r}\ ({\rm mod}\ 2)$. Then $(I,\cdot)$ is a bar-consistent anisotropic super Cartan datum (see \cite{CHW1}).
We extend $\cdot$ to a symmetric bilinear pairing on $\Z[I]$ and $p$ to a parity function $p:\Z[I]\rightarrow \Z/2\Z$.
Moreover, for $\nu=i_1+\ldots+i_t\in\N[I]$, we set
\begin{equation}\label{eq:bp and bullet}
\height\ \nu=t,\quad \bp(\nu)=\sum_{1\leq r<s\leq t} p(i_r)p(i_s),\qquad \bullet(\nu)=\sum_{1\leq r<s\leq t} i_r\cdot i_s.
\end{equation}
Let $\Phi^+\subset\N[I]$ denote the set of positive roots, and set 
\begin{equation}\label{eq:rho}
\rho=\sum_{\alpha\in \Phi^+} \alpha=\sum_{i\in I} \rho_i i\in\N[I].
\end{equation} 
Note that we have $i\cdot\rho =i\cdot i$ for all $i\in I$.

Let $Y=\Z[I]$ be the root lattice and $X=\Hom(\Z[I],\Z)$ be the weight lattice, and 
let $\ang{\cdot,\cdot}:Y\times X\rightarrow \Z$ be the natural pairing.
We also identify $\Z[I]$ as a subspace of $X$ so that $\ang{\rf r,\rf s}=2\frac{{\bf r}\cdot{\rf s}}{{\bf r}\cdot{\rf r}}$.
If $\nu=\sum_{i\in I} \nu_i i\in \Z[I]$, we set
\begin{equation}\label{eq:tilderoot}
\tilde\nu=\sum_{i\in I} d_i\nu_i i\in \Z[I]
\end{equation}
and note $\ang{\tilde\nu,\mu}=\nu\cdot\mu$ for any $\nu,\mu\in \Z[I]$;
in particular, observe that for any $i\in I$,
\begin{equation}\label{eq:tilderho}
\ang{\tilde\rho,i}=i\cdot i.
\end{equation}

Then $((I,\cdot), X, Y, \ang{\cdot,\cdot})$ is the root datum associated to $\osp(1|2n)$,
and forgetting the parity on the root datum yields the root datum associated to $\so(1+2n)$.
As usual, we define the dominant weights to be $X^+=\set{\lambda\in X\mid \ang{i,\lambda}\geq 0\text{ for all } i\in I}$.

\begin{example}\label{ex:rank1lattices}
	Throughout the paper, we will discuss some examples in the simplest case:
	$n=1$. In this case, we identify $X=\Z$ where 
	$\ang{\rf 1, k}=k$ for $k\in \Z$. 
	Then $Y=\Z\rf 1$ can be identified with subset $2\Z\subset X$.
	We will freely use these identifications in later examples.
\end{example}

Note that the weight lattice $X$ doesn't naturally have a parity grading compatible
with that on $\Z[I]$. However, a parity grading on $X$ can be defined as follows.
First observe that $X$ carries an action of the Weyl group $W$ of type $B_n$,
and that in particular $\lambda-w\lambda\in\Z[I]$ for any $\lambda\in X$.
Let $w_0$ denote the longest element of $B_n$.
If $\lambda\in X$, then $w_0\lambda=-\lambda$ hence  
$2\lambda=\lambda-w_0\lambda\in \Z[I]$.
We write $2\lambda=\sum_{i\in I} (2\lambda)_i i$ and define 
\begin{equation}\label{eq:weight parity}
P(\lambda)=p(2\lambda)\equiv (2\lambda)_{\rf n}\text{ (mod 2)}.
\end{equation} 
This defines a parity grading on $X$, though it is obviously not compatible with the grading on $\Z[I]$ (indeed, for any $i\in I$ we have $P(i)=p(2i)=2p(i)\equiv 0$ modulo 2).
In particular, $P$ is constant on cosets $X/\Z[I]$. 
This parity can be expressed explicitly in terms of the rank and weight as follows.
\begin{lem} \label{lem:weight parity equiv} Let notations be as above.
Then $P(\lambda)\equiv n\ang{\rf n,\lambda}$ mod 2.
\end{lem}
\begin{proof}
	Let $1\leq s\leq n-1$ and for convenience set the notation $(2\lambda)_{\rf 0}=0$.
	We have 
	\[\ang{{\rf s},\lambda}=\frac12\ang{{\rf s},2\lambda}=\frac{1}{2}\sum_{i\in I} (2\lambda)_i\ang{{\rf s},i}
	=(2\lambda)_{\rf s}-\frac12((2\lambda)_{\rf{s+1}}+(2\lambda)_{\rf{s-1}}),\]
	\[\ang{\rf n,\lambda}=(2\lambda)_{\rf n}-(2\lambda)_{\rf{n-1}}.\]
	In particular, we see that
	$\frac12((2\lambda)_{\rf{s+1}}+(2\lambda)_{\rf{s-1}})=(2\lambda)_{\rf s}-\ang{{\rf s},\lambda}\in \N$, thus
	$(2\lambda)_{\rf{s-1}}\equiv (2\lambda)_{\rf{s+1}}$ modulo 2 for all $1\leq s\leq n-1$.
	Therefore, $(2\lambda)_{\rf r}\equiv (2\lambda)_{\rf s}$ modulo 2 
	whenever $r\equiv s$ modulo 2.
	
	In particular, since $(2\lambda)_{\rf 0}=0$, we see that $(2\lambda)_{\rf s}\equiv 0$ modulo $2$
	for each $s\equiv 0$ modulo $2$. If $n\equiv 0$ modulo 2, then $P(\lambda)\equiv (2\lambda)_{\rf n} \equiv 0$ modulo $2$.
	If $n\equiv 1$ modulo 2, then $(2\lambda)_{\rf n}=\ang{\rf n,\lambda}-(2\lambda)_{\rf{n-1}} \equiv \ang{\rf n,\lambda}$ modulo 2.
\end{proof}

\begin{example}\label{ex:rk1weightparity}
	When $n=1$, recall from Example \ref{ex:rank1lattices} that
	we identify $X=\Z$. Then for any $k\in \Z$, $P(k)\equiv(1)\ang{\mathbf 1, k}\equiv k$ modulo 2, hence our $P$-grading is just the natural parity grading on $\Z$.
\end{example}

Throughout, we will consider objects graded by $\bigrset=X\times(\Z/2\Z)$.
If $M$ is $\bigrset$-graded and $m\in M$ is homogeneous,
we let $\bigrdeg{m}$ (resp. $|m|$; $p(m)$)
denote its $\bigrset$-degree (resp. $X$-degree; $\Z/2\Z$-degree or parity).
Further, for $\zeta=(\lambda, \epsilon)\in\bigrset$,
we will set $|\zeta|=\lambda$, $p(\zeta)=\epsilon$, and $P(\zeta)=P(\lambda)$.
(Note that $P(\zeta)$ is not the same as $p(\zeta)$ in general! They are independent
quantities.)

For $\lambda\in X$, let $\bigrelt\lambda=(\lambda,0)\in \bigrset$.
We will freely identify $\Z[I]$ with $\set{(\nu,p(\nu))\mid \nu\in \Z[I]}\subset \bigrset$.
In particular, if $\zeta=(\lambda,\epsilon)\in\bigrset$ and $\nu\in\Z[I]$, then
\begin{equation}\label{eq:ZI in X hat}
\zeta+\nu=(\lambda+\nu,\epsilon+p(\nu))\in\bigrset.
\end{equation}
With that in mind, the action of $W$ on $X$ generalizes naturally to $\bigrset$ by setting
\begin{equation}
s_i(\lambda,\epsilon)=(\lambda,\epsilon)-\ang{i,\lambda} i=(\lambda-\ang{i,\lambda} i,\epsilon-\ang{i,\lambda} p(i))
\end{equation}
where $i\in I$ and $s_i$ is the corresponding simple reflection.

Lastly, we have the parity swap function $\Pi:\bigrset\rightarrow \bigrset$
defined by 
\begin{equation}\Pi((\lambda,\epsilon))=(\lambda,1-\epsilon).\end{equation}

\subsection{Parameters}\label{sec:param}

Let $\bt\in\mathbb C$ such that $\bt^2=-1$.
Let $q$ be a formal parameter and let $\tau$ be an indeterminate
such that 
$$
\tau^4=1.
$$ 
For convenience, we will also define 
\[\pi=\tau^2.\]
If $R$ is a commutative ring with 1, define the notations
\begin{equation}
R^\tau=R[\tau]/(\tau^4=1),\quad
R^\pi=R[\pi]/(\pi^2=1).
\end{equation}
Throughout, our base ring will be $\Qqtt$, though
occasionally we will also refer to the subring generated by $\Qq$ and $\pi$,
which we identify with $\Qqp$.

We denote by $\bar{\cdot}:\Qqtt\rightarrow\Qqtt$ the $\Q(\bt)^\tau$-algebra
automorphism satisfying $\bar q=\pi q^{-1}$. We also define the $\Q(\bt)$-algebra
automorphism $\Tw$ given by $\Tw(q)=\bt^{-1} q$ and $\Tw(\tau)=\bt\tau$.
We caution the reader that $\bar{\cdot}$ and $\Tw$ will be used later to denote
extensions of these algebra automorphisms which are defined on 
$\Qqtt$-algebras and $\Qqtt$-modules. 

Given an $\Qqtt$-module (or algebra) $M$ and $x\in\set{\pm 1,\pm\bt}$,
the $\Qqt$-module
(or algebra) $M|_{\tau=x} =\Qqt_{x}\otimes_{\Qqtt} M$, 
where $\Qqt_x=\Qqt$ is viewed
as a $\Qqtt$-module on which $\tau$ acts as multiplication by $x$.
We call this the {\em specialization of $M$ at $\tau=x$ }. Moreover, $\Qqtt$ has orthogonal idempotents
\begin{equation}\label{eq:tau idempotent}
\var_{\bt^k}=\frac{1+ \bt^k\tau+(\bt^k\tau)^2+(\bt^k\tau)^3}{4},\quad 0\leq k\leq 3
\end{equation} 
such that $\Qqtt=\Qqt\var_1\oplus\Qqt\var_\bt\oplus\Qqt\var_{-1}\oplus\Qqt\var_{-\bt}$.
In particular, since $\tau\var_{x}=x \var_{x}$,
we see that for any $\Qqtt$-module $M$, 
\[M|_{\tau=x}\cong \var_{x} M.\]

For $k \in \Z_{\ge 0}$ and $n\in \Z$,
the $(q,\pi)$-quantum integers, along with quantum factorial and quantum binomial coefficients,
are defined as follows (cf. \cite{CHW1}):

\begin{equation}
\label{eq:nvpi}
\begin{split}
\bra{n}_{q,\pi} & 
=\frac{(\pi q)^n-q^{-n}}{\pi q-q^{-1}}, 
\\
\bra{n}_{q,\pi}^!  &= \prod_{l=1}^n \bra{l}_{q,\pi}, 
\\
\bbinom{n}{k}_{q,\pi}
&=\frac{\prod_{l=n-k+1}^n  \big( (\pi q)^{l} -q^{-l} \big)}{\prod_{m=1}^k \big( (\pi q)^{m}- q^{-m} \big)}.
\end{split}
\end{equation}

If $\nu=\sum_{i\in I} \nu_i i\in \Z[I]$, we write
\[q_\nu=\prod_{i\in I}q^{\nu_i d_i},\qquad 
\tau_\nu=\prod_{i\in I}\tau^{\nu_i d_i},\qquad 
\pi_\nu=\prod_{i\in I}\pi^{\nu_i d_i}=\pi^{p(\nu)}, \qquad 
\bt_\nu=\prod_{i\in I}\bt^{\nu_i d_i.}\]
In particular, note that $q_i=q^{d_i}$ and $\pi_i=\pi^{d_i}=\pi^{p(i)}$ and set
\[\bra n_i=\bra n_{q_i,\pi_i},\qquad
\bra{n}_{i}^!=
\bra{n}_{q_i,\pi_i}^!, \qquad 
\bbinom{n}{k}_{i}=\bbinom{n}{k}_{q_i,\pi_i}.\]

\subsection{The covering quantum group}

The covering quantum group associated to $\osp(1|2n)$ (as well as some variants) 
was introduced and studied in the series of papers starting with
\cite{CHW1}. We will recall the necessary definitions and elementary facts
now.

\begin{rmk}\label{rem: coefficients}
	Note that contrary to \cite{CHW1} and further papers in that series, 
	we will take coefficients in the larger ring $\Qqtt\supset\Qqp$.
	Nevertheless, all of the results until \S\ref{subsec:renorm} are essentially
	statements over $\Qqp$ which remain true after extending scalars to $\Qqtt$, 
	so the reader may effectively ignore $\tau$ and $\bt$ for the present.
\end{rmk}

\begin{dfn}\cite{CHW1}\label{def:hcqg}
	The half-quantum covering group $\ff$ associated to the anisotropic datum 
	$(I,\cdot)$ is the $\N[I]$-graded $\Qqtt$-algebra
	on the generators $\theta_i$ for $i\in I$ with $|\theta_i|=i$, 
	satisfying the relations
	\begin{equation}\label{eq:thetaserrerel}
	\sum_{k=0}^{b_{ij}} (-1)^k\pi^{\binom{k}{2}p(i)+kp(i)p(j)}
	\bbinom{b_{ij}}{k}_{i}  \theta_i^{b_{ij}-k}\theta_j\theta_i^k=0 
	\;\; (i\neq j),
	\end{equation}
	where $b_{ij}=1-\ang{i,j}$.
\end{dfn}

The algebra $\ff$ carries a non-degenerate bilinear form $(\cdot,\cdot)$
which satisfies 
\begin{equation}\label{eq:ff bilinear form}
(1,1)=1;\quad (\theta_i,\theta_i)=\frac{1}{1-\pi_iq_i^{-2}};
\quad (\theta_ix,y)=(\theta_i,\theta_i)(x,\ir(y));
\end{equation}
where $\ir:\ff\rightarrow \ff$ is the $\Qqtt$-linear map
satisfying $\ir(1)=0$, $\ir(\theta_j)=\delta_{ij}$, and $\ir(xy)=\ir(x)y+\pi^{p(i)p(x)}q^{i\cdot |x|}x\ \ir(y)$.
(Here, and henceforth, $\delta_{x,y}$ is set to be $\delta_{x,y}=1$ if $x=y$ and $0$ otherwise.)
We define the  $\Q(\bt)^\tau$-linear bar involution $\barmap$ on $\ff$ by
\[\bar \theta_i=\theta_i, \quad \bar q=\pi q^{-1}.\]
We also define the $\Qqtt$-linear anti-involution $\sigma$ on $\ff$ by
\[\sigma(\theta_i)=\theta_i,\qquad \sigma(xy)=\sigma(y)\sigma(x),\] 
and the divided powers
\[\theta_i^{(n)}=\theta_i^{n}/\bra{n}_i^!.\]

\begin{dfn} \cite{CHW1}
	\label{definition:cqg}
	The quantum covering group 
	$\UU$ associated to $((I,\cdot),\ Y,\ X,\ \ang{\cdot,\cdot})$ is the $\Qqtt$-algebra with generators
	$E_i, F_i$, $K_\mu$, and $J_\mu$, for  $i\in I$ and $\mu\in Y$, subject to the 
	relations:
	\begin{equation}\label{eq:JKrels}
	J_\mu J_\nu=J_{\mu+\nu},\quad K_\mu K_\nu=K_{\mu+\nu},\quad K_0=J_0=J_\nu^2=1,\quad
	J_\mu K_\nu=K_\nu J_\mu,
	\end{equation} 
	\begin{equation}\label{eq:Jweightrels}
	J_\mu E_i=\pi^{\ang{\mu,i}} E_i J_\mu,\quad J_\mu F_i=\pi^{-\ang{\mu,i}} F_i J_\mu,
	\end{equation} 
	\begin{equation}\label{eq:Kweightrels}
	K_\mu E_i=q^{\ang{\mu,i}} E_i K_\mu,\quad K_\mu F_i=q^{-\ang{\mu,i}} F_i K_\mu,
	\end{equation} 
	\begin{equation}\label{eq:commutatorrelation}
	E_iF_j-\pi^{p(i)p(j)}F_jE_i=\delta_{ij}\frac{J_{d_i i}K_{d_i i}-K_{-d_i i}}{\pi_i q_i- q_i^{-1}},
	\end{equation}  
	\begin{equation}\label{eq:Eserrerel}
	\sum_{k=0}^{b_{ij}} (-1)^k\pi^{\binom{k}{2}p(i)+kp(i)p(j)}\bbinom{b_{ij}}{k}_{q_i,\pi_i} 
	E_i^{b_{ij}-k}E_jE_i^k=0 \;\; (i\neq j),
	\end{equation} 
	\begin{equation}\label{eq:Fserrerel}
	\sum_{k=0}^{b_{ij}} (-1)^k\pi^{\binom{k}{2}p(i)+kp(i)p(j)}\bbinom{b_{ij}}{k}_{q_i,\pi_i} 
	F_i^{b_{ij}-k}F_jF_i^k=0 \;\; (i\neq j),
	\end{equation} 
	for $i,j\in I$ and $\mu,\nu\in Y$. 
\end{dfn}
We note that since in this case $Y=\Z[I]$, $\UU$ is actually generated
by $E_i, F_i, K_i,J_i$ for $i\in I$.
For notational convenience, we set $\tJ_\nu=J_{\tilde\nu}$ and $\tK_\nu=K_{\tilde \nu}$ 
so that \eqref{eq:commutatorrelation} becomes
\[E_iF_j-\pi^{p(i)p(j)}F_jE_i=\delta_{ij}\frac{\tJ_{i}\tK_{i}-\tK_{i}^{-1}}{\pi_i q_i- q_i^{-1}}.\]
We also equip $\UU$ with a bar involution $\bar{\cdot}:\UU\rightarrow \UU$ extending
that on $\Qqtt$ by setting $\bar{E_i}=E_i$, $\bar{F_i}=F_i$, $\bar{K_\mu}=J_\mu K_{-\mu}$,
$\bar J_\mu=J_\mu$.

The algebras $\UU$ and $\ff$ are related in the following way. 
Let $\Um$ be the subalgebra generated by $F_i$ with $i\in I$,
$\Up$ be the subalgebra generated by $E_i$ with $i\in I$,
and $\Uz$ be the subalgebra generated by $K_\nu$ and $J_\nu$ for
$\nu\in Y$. There is an isomorphisms $\ff\rightarrow\Um$ 
(resp. $\ff\rightarrow \Up$) defined by $\theta_i\mapsto \theta_i^-=F_i$
(resp. $\theta_i\mapsto \theta_i^+=E_i$).
We let $E_i^{(n)}=(\theta_i^{(n)})^+$ and $F_i^{(n)}=(\theta_i)^{(n)})^-$.
As shown in \cite{CHW1}, there is a triangular decomposition
\[\UU\cong \Um\otimes\Uz\otimes\Up\cong \Up\otimes \Uz\otimes \Um.\]
There is also a root space decomposition
\[\UU=\bigoplus_{\nu\in\Z[I]} \UU_\nu,\qquad \UU_\nu=\set{x\in\UU\mid J_\mu K_{\xi}m=\pi^{\ang{\mu,\nu}} q^{\ang{\xi,\nu}} m}.\]
The root space decomposition induces a parity grading via $p(u)=p(|u|)$,
hence in particular $\UU$ is $\bigrset$-graded.

We say an algebra is a ``Hopf covering algebra''
if it is a $\Z/2\Z$-graded algebra over $R^\pi$, for some commutative ring with identity $R$,
with a coproduct, antipode, and counit satisfying
the usual axioms of a Hopf superalgebra, but with the
braiding replaced by 
$x\otimes y\mapsto \pi^{p(x)p(y)}y\otimes x$.
Then the algebra $\UU$ is a Hopf covering algebra under the
coproduct $\Delta:\UU\rightarrow \UU\otimes \UU$ satisfying
\[\Delta(E_i)=E_i\otimes 1+K_i\otimes E_i,\qquad \Delta(F_i)=F_i\otimes \tK_i^{-1}+1\otimes F_i,\qquad
\Delta(K_\nu)=K_\nu\otimes K_\nu,\qquad \Delta(J_\nu)=J_\nu\otimes J_\nu;\]
the antipode $S:\UU\rightarrow \UU$ satisfying $S(xy)=\pi^{p(x)p(y)}S(y)S(x)$ for $x,y\in \UU$ and
\[S(E_i)=-\tJ_i^{-1}\tK_i^{-1}E_i,\quad S(F_i)=-F_i\tK_i,\quad
S(K_\nu)=K_\nu^{-1},\quad S(J_\nu)=J_\nu^{-1};\]
and the counit $\epsilon:\UU\rightarrow\Qqtt$ satisfying
\[\epsilon(E_i)=\epsilon(F_i)=0,\qquad
\epsilon(K_\nu)=\epsilon(J_\nu)=1.\]
Moreover, for $x\in \ff$, we have that
\begin{equation}\label{eq:antipode formula}
\begin{array}{c}
S^{\pm 1}(x^+)=(-1)^{\height\nu}\pi^{\bp(\nu)} q^{\frac{\nu\cdot \nu}{2}}q_{\mp\nu}\tJ_{-\nu}\tK_{-\nu}\sigma(x)^+\\
S^{\pm 1}(x^-)=(-1)^{\height\nu}\pi^{\bp(\nu)} q^{\frac{-\nu\cdot \nu}{2}}q_{\pm \nu}\sigma(x)^-\tK_{\nu}
\end{array}
\end{equation}

\subsection{$\UU$-modules}\label{subsec:modules}
In this paper, a weight $\UU$-module is a $\UU$-module $M$ with 
a $\bigrset$-grading compatible with the grading on $\UU$, 
such that 
\[M=\bigoplus_{\lambda\in X} M_{\lambda,0}\oplus M_{\lambda,1},\quad M_{\lambda,s}=\set{m\in M\mid p(m)=s,\ J_\mu K_\nu m=\pi^{\ang{\mu,\lambda}}q^{\ang{\nu,\lambda}}m}\]
and each $M_{\lambda,s}$ is a 
free $\Qqtt$-module of finite rank.
For $\lambda\in X$, denote $M_\lambda=M_{\lambda,0}\oplus M_{\lambda,1}$.
We also define the parity-swapped module $\Pi M$ to be $M$ as a vector
space with the same action of $\UU$, but with $\Pi M_{\lambda,s}= M_{\lambda, 1-s}$.
We let $\catO_{\rm fin}$ be the category of weight $\UU$-modules of finite rank over $\Qqtt$.
Henceforth, we shall {\em always} assume our $\UU$-modules
are in $\catO_{\rm fin}$.

We define the (restricted) linear dual of a $\UU$-module $M$
\[M^*=\bigoplus_{\lambda\in X} (M_{\lambda,0})^*\oplus (M_{\lambda,1})^*,\quad
(M_{\lambda,s})^*={\rm Hom}_{\Qqtt}(M_{\lambda,s},\Qqtt).\]
This is again a free $\Qqtt$-module, which has a $\Z/2\Z$ grading induced by that of 
$V$: namely, $p(f)=0$ if $f(v)=0$ for $p(v)=1$, and vice-versa. 
Moreover, the Hopf superalgebra structure of $\UU$
induces an action of $\UU$: for $f\in V^*$ and $x\in \UU$, we define
$xf\in V^*$ by $xf(v)=\pi^{p(f)p(x)}f(S(x)v)$. 
In particular, note that $V^*$ is a $\UU$-module
with $(V^*)_{\lambda,s}=(V_{-\lambda,s})^*$. While $V^*_\lambda$ is therefore ambiguous,
we will always take it to denote $(V^*)_\lambda$. (In other words, our convention is that
taking duals has precedence over taking weight spaces.)

For any $\UU$-modules $V$ and $W$, we can construct the $\UU$-module
$V\otimes W=V\otimes_{\Qqtt} W$ via the coproduct. In particular, we have $\UU$-modules
$V^*\otimes V$ and $V\otimes V^*$,
both of which contain a copy of
the trivial module $V(0)=\Qqtt$ as a direct summand. As the following lemma shows,
there are natural projection and inclusion maps to a copy of the trivial module. We borrow notation from \cite{Ti}.

\begin{lem}\label{lemma:evsandcoevs} Fix a $\UU$-module $V$ and recall
the definition of $\rho$ from \eqref{eq:rho}.
	\begin{enumerate}
		\item Let $\ev_V:V^*\otimes V\rightarrow \Qqtt$ be the $\Qqtt$-linear map
		defined by $v^*\otimes w\rightarrow v^*(w)$. Then $\ev_V$ is a $\UU$-module
		epimorphism.
		\item Let $\qtr_V:V\otimes V^*\rightarrow \Qqtt$ 
		be the $\Qqtt$-linear map defined by 
		$v\otimes w^*\rightarrow \pi^{p(v)p(w)}q^{-\ang{\tilde \rho,|v|}}w^*(v)$. 
		Then $\qtr_V$ is a $\UU$-module epimorphism.
		\item Let $\coev_V:\Qqtt\rightarrow V^*\otimes V$ be the 
		$\Qqtt$-linear map defined by 
		$1\rightarrow \sum_{b\in B} \pi^{p(b)}q^{\ang{\tilde\rho,|b|}}b^*\otimes b$ 
		for some homogeneous $\Qqtt$-basis $B$ of $V$. Then $\coev_V$ is a $\UU$-module
		monomorphism.
		\item Let $\coqtr_V:\Qqtt\rightarrow V\otimes V^*$ be the $\Qqtt$-linear map
		defined by $1\rightarrow \sum_{b\in B} b\otimes b^*$ for some homogeneous $\Qqtt$-basis 
		$B$ of $V$. Then $\coqtr_V$ is a $\UU$-modulemonomorphism.
	\end{enumerate}
\end{lem}
\begin{proof} In the proof, we shall surpress the $V$ subscript on the maps.
	First note that the maps $\coev$ and $\coqtr$ are independent of the choice of basis.
	It is clear that all these maps are $\Qqtt$-linear maps, and it is elementary
	to verify the claims about surjectivity and injectivity.
	Moreover, all the maps are clearly homogeneous since $|v^*|=-|v|$ and $p(v^*)=p(v)$; in particular, 
	the maps $\qtr$ and $\ev$ are homogeneous since $v^*(w)=0$ whenever $|v|\neq |w|$ or 
	$p(v)\neq p(w)$, which occurs exactly when $|v^*\otimes w|\neq 0$ or 
	$p(v^*\otimes w)=1$. 
	
	Then it remains to show these maps preserve 
	the action of $E_i$ and $F_i$ for all $i\in I$, which is equivalent to showing
	\[\ev(\Delta(E_i)v^*\otimes w)=\ev(\Delta(F_i)v^*\otimes w)=0\text{ for all }v,w\in V,\]
	\[\qtr(\Delta(E_i)v\otimes w^*)=\qtr(\Delta(F_i)v\otimes w^*)=0\text{ for all }v,w\in V,\tag{$\star$}\]
	\[\Delta(E_i)\sum_{b\in B} \pi^{p(b)}q^{\ang{\tilde\rho,|b|}}b^*\otimes b=\Delta(F_i)\sum_{b\in B}\pi^{p(b)}q^{\ang{\tilde\rho,|b|}} b^*\otimes b=0\text{ for all }b\in B,\text{ and}\tag{$\star\star$}\]
	\[\Delta(E_i)\sum_{b\in B} b\otimes b^*=\Delta(F_i)\sum_{b\in B} b\otimes b^*=0\text{ for all }b\in B. \]
	We will prove ($\star$) and ($\star\star$) for the action of $E_i$; the remaining cases follow from similar arguments.
	
	First, we show $\qtr(\Delta(E_i)v^*\otimes w)=0$. 
	From weight considerations we have that $\ev(\Delta(E_i)v^*\otimes w)=0$ unless
	$|v|+i=|w|$. In this case, 
	\begin{align*}
	\qtr(\Delta(E_i)&v\otimes w^*)
	=\qtr(E_iv\otimes w^*+\pi_i^{p(v)}(\pi_iq_i)^{\ang{i,|v|}} v\otimes E_iw^*)\\
	&=\pi^{p(E_iv)p(w)}q^{-\ang{\tilde\rho,|E_iv|}}w^*(E_iv)+\pi_i^{p(v)}(\pi_iq_i)^{\ang{i,|v|}} \pi^{p(v)p(E_iw)}q^{-\ang{\tilde\rho,|v|}}(E_iw^*)(v)\\
	&=\pi^{p(v)p(w)+p(i)p(w)}q^{-\ang{\tilde\rho,|v|+i'}}
	\parens{w^*(E_iv)-q_i^{2}(\pi_iq_i)^{\ang{i,|v|}}w^*(\tJ_i^{-1}\tK_i^{-1} E_i v)}\\
	&=\pi^{p(v)p(w)+p(i)p(w)}q^{-\ang{\tilde\rho,|v|+i'}}
	\parens{w^*(E_iv)-w^*(E_i v)}=0.
	\end{align*}
	
	Next, we show that $\Delta(E_i)\sum_{b\in B} \pi^{p(b)}q^{\ang{\tilde\rho,|b|}}b^*\otimes b=0$. 
	Set $B_\lambda=B\cap V_\lambda$, so $B=\coprod_\lambda B_\lambda$.
	First observe that $x=\sum v^*\otimes w=0$ if and only if $x(v'):=\sum v^*(v')w=0$ for all $v'\in V$.
	Then setting $x=\Delta(E_i)\sum_{b\in B} b^*\otimes b$, if
	\[0\neq x=\sum_{b\in B} \pi^{p(b)}q^{\ang{\tilde\rho,|b|}}(E_ib^*\otimes b+\pi_i^{p(b)}(\pi_iq_i)^{-\ang{i,|b|}}b^*\otimes E_ib),\]
	then there must be some $v\in V$ such that \[x(v)=\sum_{b\in B}\pi^{p(b)} q^{\ang{\tilde\rho,|b|}}((E_ib^*)(v)b+\pi_i^{p(b)}(\pi_iq_i)^{-\ang{i,|b|}}b^*(v)E_ib )\neq 0.\]
	However, if $b'\in B$, 
	\begin{align*}
	x(b')&=(\pi_iq_i)^{-\ang{i,|b'|}}E_ib'+\sum_{b\in B_{|b'|+i}} b^*(-\tJ_i^{-1}\tK_i^{-1}E_ib')b\\
	&= q^{\ang{\tilde\rho,|b'|}}(\pi_iq_i)^{-\ang{i,|b'|}}E_ib'-\sum_{b\in B_{|b'|+i}}q^{\ang{\tilde\rho,|b|}}(\pi_iq_i)^{-\ang{i,|b'|+i}}b^*(E_ib')b\\
	&= q^{\ang{\tilde\rho,|b'|}}(\pi_iq_i)^{-\ang{i,|b'|}}\parens{E_ib'-\sum_{b\in B_{|b'|+i}}b^*(E_ib')b}=0.
	\end{align*}
	\if
	Next, we show that $\Delta(E_i)\sum_{b\in B} b\otimes b^*=0$. We may assume that $B$ consists of homogeneous elements.
	First observe that $x=\sum v\otimes w^*=0$ if and only if $x(v'):=\sum w^*(v')v=0$ for all $v'\in V$.
	Then setting $x=\Delta(E_i)\sum_{b\in B} b\otimes b^*$
	\[0\neq x=\sum_{b\in B} (E_ib\otimes b^*+\pi_i^{p(b)}(\pi_iq_i)^{\ang{i,|b|}}b\otimes E_ib^*),\]
	there must be some $v\in V$ such that \[x(v)=\sum_{b\in B} (b^*(v)E_ib+\pi_i^{p(b)}(\pi_iq_i)^{\ang{i,|b|}}(E_ib^*)(v)b )\neq 0.\]
	However, if $b'\in B$, 
	\begin{align*}
	x(b')&=E_ib'+\sum_{b\in B_{|b'|+i}}(\pi_iq_i)^{\ang{i,|b|}}b^*(-\tJ_i^{-1}\tK_i^{-1}E_ib')b\\
	&= E_ib'-\sum_{b\in B_{|b'|+i}}\pi_i^{p(b)}b^*(E_ib')b
	\end{align*}
	\fi
\end{proof}

\subsection{Simple modules and their duals}
Let $\lambda\in X^+$ and recall from \cite{CHW1} that
$V(\lambda)$ is the simple $\UU$-module of highest weight $\lambda$ such that the highest weight
space has even parity. Then $V(\lambda)$ has finite rank and has the same character as the $\so(2n+1)$ module of
highest weight $\lambda$. In particular, using the Weyl 
character formula for $V(\lambda)$, the lowest weight vector has weight $w_0\lambda=-\lambda$,
hence  the parity of the lowest weight vector of $V(\lambda)$ is 
$P(\lambda)$.
Using standard arguments (for example, analogues of \cite[\S 5.3 and \S 5.16]{Jan}),
and considering the above analysis, we obtain the following lemma.

\begin{lem}\label{lem:dual isos}
	For each $\lambda\in X^+$,
	there is an isomorphism $V(\lambda)^{*}\cong \Pi^{P(\lambda)} V(\lambda)$ 
	and a natural isomorphism $V(\lambda)^{**}\rightarrow V(\lambda)$.
\end{lem}

\begin{example}\label{ex:rk1mods}
	In the case $n=1$, 
	the module $V=V(m)$ for $m\in\Z_{\geq 0}$  has basis $v_{m-2k}=F^{(k)}v_m$ 
	with $0\leq k\leq m$, where $v_m$ is a choice of 
	highest weight vector. Note that by convention $p(v_m)=0$, 
	so $p(v_{m-2k})\equiv k$ (mod 2).  The dual module $V(m)^*$ 
	has a dual basis $v_{m-2k}^*$, $0\leq k\leq m$,
	and the actions of $E=E_{\rf 1}$ and $F=F_{\rf 1}$ are given by 
	\begin{align*}
	Ev_{m-2k}^*&=-\pi^k(\pi q)^{m-2k}\bra{n+1-k}v_{m-2(k+1)}^*\\
	Fv_{m-2k}^*&=-\pi^k(\pi q)^{m-2k+2}\bra{k}v_{m-2(k-1)}^*
	\end{align*}
	In particular, this is a simple module generated by the highest weight vector $v_{-m}^*$, 
	where $|v_{-m}^*|=-|v_{-m}|=m$ and
	$p(v_{-m}^*)=p(v_{-m})\equiv m$ (mod 2), hence we have an isomorphism
	$V(m)^*\cong \Pi^m V(m)$.
\end{example}

For convenience, we will use the notation 
\begin{equation}\label{eq:dual notation}
V(-\lambda)=V(\lambda)^*,\quad \lambda\in X^+.
\end{equation}
We denote the maps in Lemma \ref{lemma:evsandcoevs} in the case $V=V(\lambda)$
with the subscript $\lambda$ instead of $V(\lambda)$; for instance, $\ev_\lambda=\ev_{V(\lambda)}$.
Note that \[\ev_\lambda\ \circ\ \coev_\lambda=\sum_{\nu \in \N[I]}{\rm rank}_{\Qqtt}(V_{\lambda-\nu}) \pi^{p(\nu)}q^{\ang{\tilde{\rho},\lambda-\nu}}=\pi^{P(\lambda)}\qtr_\lambda\ \circ\ \coqtr_\lambda\]
\begin{example}\label{ex:rk1evcoev}
	For $n=1$, we have $\rho=\tilde\rho=\mathbf 1$ hence for $\lambda=m$, 
	$\ang{\tilde\rho,\lambda}=m$. Then 
	\[\ev_m\circ\coev_m=q^m+\pi q^{m-2}+\ldots+\pi^{m} q^{-m}=\pi^m[m+1]
	=\pi^m\qtr_m\circ\coqtr_m.\]
\end{example}

\subsection{Further properties of the quasi-$\cR$-matrix}

Let us recall the quasi-$\cR$-matrix from \cite[\S 4]{CHW1}

\begin{prop}\label{prop:quasiR}\cite{CHW1}
Let $\BB$ be any $\Qqtt$-basis of $\ff$ such that $\BB_\nu=\BB\cap \ff_\nu$ is a basis of $\ff_\nu$ for any $\nu\in \N[I]$, with $\BB_0=\set 1$. 
Let $\BB^*=\set{b^*\mid b\in \BB}$ be the basis of $\ff$ dual to $\BB$ under $(\cdot,\cdot)$. Define
\[
\Theta_\nu=(-1)^{\height\,\nu}\pi^{\bp(\nu)}\pi_\nu q_\nu\sum_{b\in \BB_\nu} b^-\otimes (b^*)^+\in\UU_{-\nu}^-\otimes \UU_\nu^+.
\]
Then if $M,M'$ are integrable modules of $\UU$, then $\Theta=\sum_{\nu}\Theta_\nu$ is a well defined
operator on $M\otimes M'$ which satisfies $\Delta(u)\Theta=\Theta\bar\Delta(u)$ as endomorphisms of $M\otimes M'$, where $\bar\Delta(u)=\bar{\Delta(\bar u)}$. Moreover, $\Theta$ is independent of the choice of basis $\BB$, and is invertible with inverse $\bar\Theta$.
\end{prop}

In particular, note that all modules considered in this paper
are of finite rank over $\Qqtt$, hence are integrable.

\begin{example}\label{ex:rank1rmat}
When $n=1$, the quasi-$\mathcal R$-matrix
$\Theta$ can be explicitly given by the formula
\[\Theta=\sum_{n\geq 0} (-1)^n (\pi q)^{-\binom{n}{2}}[n]^!(\pi q-q^{-1})^n F^{(n)}\otimes E^{(n)}=1-(\pi q-q^{-1})F\otimes E+\ldots.\]
(NB. there is a typo in the power of $\pi q$ in \cite[Example 3.1.2]{CHW1}.)
\end{example}

While $\bar\Theta$ can be evaluated easily, it will be more convenient
to have the following alternate description of $\bar\Theta$
using the properties of the bilinear form on $\ff$ 
(cf. \cite[\S 1.4]{CHW1}).

\begin{lem}\label{lem:quasiRinv} With the same notations as in Proposition \ref{prop:quasiR},
$\bar\Theta=\sum_\nu \bar\Theta_\nu$ is given by
\[
\bar \Theta_\nu=\pi_\nu q^{\frac{\nu\cdot\nu}2}\sum_{b\in \BB_\nu} b^-\otimes \sigma(b^*)^+\in\UU_{-\nu}^-\otimes \UU_\nu^+.
\]
\end{lem}
\begin{proof}
Let $\bar\BB=\set{\bar b\mid b\in\BB}$, with dual basis $\bar\BB^*$.
Then since $\Theta$ is independent of the choice of basis, we see that for $\nu\in\N[I]$,
$\Theta_\nu=(-1)^{\height\,\nu}\pi^{\bp(\nu)}\pi_\nu q_\nu\sum_{b\in \BB_\nu} \bar{b}^-\otimes (\bar{b}^*)^+$.
We have $\bar\Theta_\nu=(-1)^{\height \nu} \pi^{\bp(\nu)}q_{-\nu}\sum_{b\in \BB_\nu} \bar{(\bar{b}^-)}\otimes \bar{(\bar{b^*}^+)}$, and note that $(\bar{x})^{\pm}=\bar{(x^{\pm})}$,
so $\bar{(\bar{b}^-)}=b^-$.

On the other hand, recall from \cite[\S 1.4]{CHW1} the variant bilinear form $\set{-,-}$
defined by $\set{x,y}=\bar{(\bar x,\bar y)}$. 
Note that by construction, $(\bar{b}^*,\bar{b'})=\delta_{b,b'}$.
Then for any $b,b'\in\BB$, we apply Lemma 1.4.3 (b) of {\em loc. cit.} to deduce that
\[\delta_{b,b'}=\bar{(\bar{b}^*, \bar{b'})}=\set{\bar{\bar{b}^*},b'}
=(-1)^{\height \nu}\pi^{\bp(\nu)}\pi_\nu q^{-\frac{\nu\cdot\nu}{2}} q_{-\nu}(\bar{\bar{b}^*},\sigma(b')).\]
(We note that while the power of $\pi$ appears different from that in {\em loc. cit.},
it is equivalent.)
Therefore, we have
\[\bar{\bar{b}^*}=(-1)^{\height \nu}\pi^{\bp(\nu)}q^{\frac{\nu\cdot\nu}{2}}\pi_\nu q_{\nu} \sigma(b)^*.\]
Then the lemma follows from the observation 
that since $(\sigma(x),\sigma(y))=(x,y)$, $\sigma(b)^*=\sigma(b^*)$.
\end{proof}

Now we will proceed to use $\Theta$ to define a universal map $\cR:M\otimes N\rightarrow N\otimes M$
for any modules $M$ and $N$. These constructions will be modified versions of the standard arguments in the non-super case; cf.
\cite[\SS 7.3-7.6]{Jan} or \cite[\S 4.2 and Chapter 32]{L93}.

For $1\leq s<t\leq 3$, let $\Theta_\nu^{st}\in \UU\otimes \UU\otimes \UU$ be defined by 
$\Theta_\nu^{st}=(-1)^{\height\,\nu}\pi^{\bp(\nu)}\pi_\nu q_\nu\sum_{b\in \BB_\nu} b_1\otimes b_2\otimes b_3$
where $b_s=b^-$,$b_t=(b^*)^+$, and $b_m=1$ for $m\neq s,t$.

\begin{prop}
We have the following identities.
\[(\Delta\otimes 1)(\Theta_\nu)=\sum_{\nu'+\nu''=\nu}\Theta_{\nu'}^{23}(1\otimes \tK_{-\nu''}\otimes 1)\Theta^{13}_{\nu''}.\]
\[(\bar\Delta\otimes 1)(\Theta_\nu)=\sum_{\nu'+\nu''=\nu}\Theta_{\nu'}^{13}(1\otimes \tJ_{\nu'}\tK_{\nu'}\otimes 1)\Theta^{23}_{\nu''}.\]
\[(1\otimes \Delta )(\Theta_\nu)=\sum_{\nu'+\nu''=\nu}\Theta_{\nu'}^{12}(1\otimes \tJ_{\nu''}\tK_{\nu''}\otimes 1)\Theta^{13}_{\nu''}.\]
\[(1\otimes \bar\Delta )(\Theta_\nu)=\sum_{\nu'+\nu''=\nu}\Theta_{\nu'}^{13}(1\otimes \tK_{-\nu''}\otimes 1)\Theta^{12}_{\nu''}.\]
\end{prop}
\begin{proof}
These identities are proved exactly as in \cite[\S 4.2]{L93}. We will prove
the first identity here.
For $x\in \ff$ and $b_1,b_2\in \BB$, define $f(x,b_1,b_2),f'(x,b_1,b_2)\in \Qqtt$ via 
\[r(x)=\sum_{b_1,b_2\in B} f(x,b_1,b_2)\,b_1\otimes b_2,\]
\[\bar r(x)=\sum_{b_1,b_2\in B} f'(x,b_1,b_2)\,b_1\otimes b_2.\]
Then it suffices to show that 
\[\sum_{b,b_1,b_2;|b_1|+|b_2|=|b|=\nu} f'(b,b_1,b_2) b_1^-\otimes \tK_{-|b_1|}b_2^-\otimes (b^*)^+=\sum_{b_1,b_2;|b_1|+|b_2|=\nu}
\pi^{p(b_1)p(b_2)}b_1^-\otimes b_2^-\tK_{-|b_1|}\otimes (b_2^*b_1^*)^+.\]
In particular, it is enough to show that for all $b_1,b_2\in B$ such that $|b_1|+|b_2|=\nu$, we have
\[\sum_{b;|b|=\nu} f'(b,b_1,b_2)b^*=\pi^{p(b_1)p(b_2)}q^{-|b_1|\cdot |b_2|}b_2^*b_1^*.\]
This follows from the equalities 
\[\pi^{p(b_1)p(b_2)}q^{|b_1|\cdot |b_2|}f'(b,b_1,b_2)=f(b,b_2,b_1)=(r(b),b_2^*\otimes b_1^*)=(b,b_2^*b_1^*),\]
which in turn follow from elementary properties of $\ff$;
cf. \cite[Lemmas 1.4.1, 1.4.3]{CHW1}
\end{proof}

To construct a universal $\UU$-module homomorphism from $\Theta$,
we will need some additional maps. The first is the swap map;
that is, the algebra $\UU\otimes \UU$ is equipped with an 
involution $\frs$
defined by $\frs(x\otimes y)=\pi^{p(x)p(y)} y\otimes x$.
This induces involutions on $\UU^{\otimes m}$ by applying $\frs$ to sequential
pairs of tensor factors; specifically, these involutions are the maps
$\frs_{t,t+1}=1^{\otimes t-1}\otimes \frs\otimes 1^{m-t-1}$,
and it is not hard to see they satisfy the braid relations
$\frs_{t-1,t}\frs_{t,t+1}\frs_{t-1,t}=\frs_{t,t+1}\frs_{t-1,t}\frs_{t,t+1}$.
In particular, we see that to each element $\gamma$ of the permutation group
$\mathfrak S_m$, there is an
automorphism $\frs_\gamma$ of $\UU^{m}$; for example $\frs_{(23)}=\frs_{2,3}$ and $\frs_{(123)}=\frs_{1,2}\frs_{2,3}$.
Similarly, to any tensor product of modules $N=\bigotimes_{i=1}^m M_t$ and $\gamma\in \mathfrak S_m$,
we can define $N_\gamma=\bigotimes_{i=1}^m M_{\gamma(t)}$ and
a map $\frs_\gamma: N\rightarrow N_\gamma$ given by 
\[\frs(v)=\pi^{p(\gamma,v)}v_\gamma,\]
where $v=v_1\otimes\ldots\otimes v_m$, $v_\gamma=v_{\gamma(1)}\otimes\ldots\otimes v_{\gamma(m)}$,
and \[p(\gamma,v)=\displaystyle\sum_{\shortstack{$1\leq s<t\leq n$ \\ $\gamma(s)>\gamma(t)$}}p(v_t)p(v_s).\]
These maps are compatible in the sense that for 
$v=\bigotimes_{t=1}^n v_t\in N$ and $u\in \UU$, \[\frs_\gamma(\Delta^{m-1}(u)v)=\frs_\gamma(\Delta^{m-1}(u)) \frs_\gamma(v).\]
When $m=2$, we will just write $\frs=\frs_{1,2}$.

The other ingredient is a weight-renormalization operator. This operator
is induced by the weight function defined in the following lemma.
\begin{lem}\label{lem:f function}
There exists a function $\mathfrak f:X\times X\rightarrow (\Qqtt)^\times$  satisfying 
\[\mathfrak f(\zeta+\mu',\zeta'+\nu')\mathfrak f(\zeta,\zeta')^{-1}
= (\pi q)^{-\ang{\tilde\mu,\zeta'}}q^{-\ang{\tilde\nu,\zeta}-\mu\cdot \nu}\]
for $\zeta,\zeta'\in X$ and $\mu,\nu\in \Z[I]$.
Moreover,
\begin{enumerate}
\item The function 
$\mathfrak r(\zeta,\zeta')=\mathfrak f(\zeta,\zeta')\mathfrak f(\zeta,-\zeta')$ satisfies $\mathfrak r(\zeta+\mu,\zeta'+\nu)=\mathfrak r(\zeta,\zeta')$ for any $\mu,\nu\in\Z[I]$.
\item The function 
$\mathfrak l(\zeta,\zeta')=\mathfrak f(\zeta,\zeta')\mathfrak f(-\zeta,\zeta')$ satisfies $\mathfrak l(\zeta+\mu,\zeta'+\nu)=\mathfrak l(\zeta,\zeta')$ for any $\mu,\nu\in\Z[I]$.
\item We have $\frf(\zeta,\zeta')\frf(-\zeta,-\zeta')^{-1}=\pi^{P(\zeta)P(\zeta')}$;
in particular, $\mathfrak l(\zeta,\zeta')=\pi^{P(\zeta)P(\zeta')}\mathfrak r(\zeta,\zeta')$.
\end{enumerate}
\end{lem}
\begin{proof}
It is easy to verify that such a function $\frf$ 
exists by choosing a set of coset representatives $R$ for $\Z[I]$ in $X$.
It is similar to verify (1) and (2), so let us show (1).
Let $\xi=\zeta+\nu$ and $\xi'=\zeta'+\mu$
for some $\mu,\nu\in \N[I]$ and $\zeta,\zeta'\in X$. Then
\begin{align*}
\mathfrak f(\xi,\xi')\mathfrak f(\xi,-\xi')&=\mathfrak f(\zeta,\zeta')\mathfrak f(\zeta,-\zeta')(\pi q)^{\ang{\tilde\mu,\zeta'}+\ang{\tilde\mu,-\zeta'}}q^{-\ang{\tilde\nu,\zeta}-\nu\cdot\mu-\ang{-\tilde\nu,\zeta}-(-\nu\cdot\mu)}\\
&=\mathfrak f(\zeta,\zeta')\mathfrak f(\zeta,-\zeta').
\end{align*} 

Finally, let $\zeta,\zeta'\in X$.
Then $-\zeta=\zeta-2\zeta$, $-\zeta'=\zeta'-2{\zeta'}$
so
\[\frf(\zeta,\zeta')\frf(-\zeta,-\zeta')^{-1}=\mathfrak f(-\zeta+(2\zeta),-\zeta'+(2{\zeta'}))\mathfrak f(-\zeta,-\zeta')^{-1}
=\pi^{-\ang{\widetilde{(2\zeta)},-\zeta'}}
q^{-\ang{\widetilde{(2\zeta)},-\zeta'}-\ang{\widetilde{(2\zeta')},-\zeta}
-2{\zeta}\cdot2{\zeta'}}\]
Now note that for any $\eta,\eta'\in X$, we have
$-\ang{\widetilde{(2\eta)},-\eta'}
=\frac{1}{2}(2\eta)\cdot (2\eta')$. Moreover, by \eqref{eq:weight parity}
and Lemma \ref{lem:weight parity equiv}, we see that
$\ang{\widetilde{(2\eta)},\eta'}\equiv (2\eta)_{\rf n}\ang{\rf n,\eta'}\equiv n\ang{\rf n,\eta}\ang{\rf n,\eta'}\equiv  p(\eta)p(\eta')\mod 2$ .
Therefore, we see that
\[\mathfrak f(-\zeta+2\zeta,-\zeta'+2{\zeta'})\mathfrak f(-\zeta,-\zeta')^{-1}
=\pi^{P(\eta)P(\eta')}.\]
This finishes the proof.
\end{proof}

\begin{example}\label{ex:rank1f}
Let us consider the case $n=1$. 
Then the function $\mathfrak f$ is determined by the values 
$\mathfrak f(0,0)$, $\mathfrak f(0,1)$, $\mathfrak f(1,0)$, 
and $\mathfrak f(1,1)$. 
Then for any $\epsilon_1,\epsilon_2\in\set{0,1}$,
\[\mathfrak f(\epsilon_1+2s,\epsilon_2+2t)=\mathfrak f(\epsilon_1,\epsilon_2)\pi^{s\epsilon_2} q^{-t\epsilon_1-s\epsilon_2-2st}.\]
By direct computation, one finds the corresponding coset functions to be
\[\mathfrak r(\epsilon_1+2s,\epsilon_2+2t)=\mathfrak f(\epsilon_1,\epsilon_2)^2q^{\epsilon_1\epsilon_2}.\]
\[\mathfrak l(\epsilon_1+2s,\epsilon_2+2t)=\mathfrak f(\epsilon_1,\epsilon_2)^2\pi^{\epsilon_1\epsilon_2}q^{\epsilon_1\epsilon_2}.\]
\end{example}

Given $\UU$-modules $M,M'$, define the $\Qqtt$-linear bijection $\mathfrak F:M\otimes M'\rightarrow M\otimes M'$ 
by $\mathfrak F(m\otimes m')=\mathfrak f(|m|,|m'|) m\otimes m'$. 
For $1\leq s<t\leq 3$, we define $\mathfrak F^{st}$ on $M_1\otimes M_2\otimes M_3$
via $\mathfrak F^{st}(m_1\otimes m_2\otimes m_3)=\mathfrak f(|m_s|,|m_t|)m_1\otimes m_2
\otimes m_3$. Let ${}^\mathfrak{F}\Theta^{st}=\Theta^{st}\circ \mathfrak F^{st}$.

\begin{prop}[Yang-Baxter equation]
As operators on $M_1\otimes M_2\otimes M_3$,
\[{}^\mathfrak{F}\Theta^{12}\circ{}^\mathfrak{F}
\Theta^{13}\circ{}^\mathfrak{F}\Theta^{23}=
{}^\mathfrak{F}\Theta^{23}\circ{}^\mathfrak{F}\Theta^{13}\circ
{}^\mathfrak{F}\Theta^{12}\]
\end{prop}

\begin{proof}
First note that the maps $\mathfrak F^{st}$ are bijections which commute with one another.
One verifies directly that, as operators on $M_1\otimes M_2\otimes M_3$,
\[\mathfrak F^{12}\Theta_\nu^{13}=\Theta_\nu^{13}(1\otimes \tJ_\nu\tK_{\nu}\otimes 1)\mathfrak F^{12}\text{, } \quad
\mathfrak F^{12}\mathfrak{F}^{13}\Theta_\nu^{23}=
\Theta_\nu^{23}\mathfrak F^{12}\mathfrak{F}^{13},\]
\[\mathfrak F^{23}\Theta_\nu^{13}=\Theta_\nu^{13}(1\otimes \tK_{-\nu} \otimes 1)\mathfrak F^{23}\text{, } \quad
\mathfrak F^{23}\mathfrak{F}^{13}\Theta_\nu^{12}=
\Theta_\nu^{12}\mathfrak F^{23}\mathfrak{F}^{13}.\]
In particular, it suffices to show that
\[\Theta^{12}\parens{\sum_{\nu}\Theta_\nu^{13}(1\otimes \tJ_\nu\tK_{\nu}\otimes 1)}
\Theta^{23}=\Theta^{23}\parens{\sum_{\nu}\Theta_\nu^{13}(1\otimes \tK_{-\nu}\otimes 1)}
\Theta^{12}.\]
Writing $\Theta^{12}=\sum_{\mu}\Theta_\mu^{12}$, we have
\[\Theta^{12}\parens{\sum_{\nu}\Theta_\nu^{13}(1\otimes \tJ_\nu\tK_{\nu}\otimes 1)}
=\sum_{\mu,\nu}\Theta_\mu^{12}(1\otimes \tJ_\nu\tK_{\nu}\otimes 1)\Theta_\nu^{13}
=\sum_{\nu}(1\otimes \Delta)(\Theta_{\nu}),\]
and similarly \[\parens{\sum_{\nu}\Theta_\nu^{13}(1\otimes \tK_{-\nu}\otimes 1)}
\Theta^{12}=\sum_{\nu}(1\otimes \bar \Delta)(\Theta_{\nu}).\]
Then we are reduced to showing the equality
\[\sum_{\nu}(1\otimes \Delta)(\Theta_{\nu})\Theta^{23}=\Theta^{23}\sum_{\nu}(1\otimes \bar \Delta)(\Theta_{\nu}),\]
which follows from the defining property of $\Theta$.
\end{proof}

\begin{prop}\label{prop:Rmat}
Define $\mathcal R: M\otimes M'\rightarrow M'\otimes M$ by $\mathcal R=\Theta\circ\mathfrak F\circ\frs$.
Then $\mathcal R$ is a $\UU$-module isomorphism.
\end{prop}

\begin{proof}
That $\mathcal R$ is bijective and homogeneous in parity is clear from construction. 
Note that \[\Delta(u)\mathcal R(m\otimes m')=\Theta(\bar\Delta(u)\mathfrak F\circ \frs(m\otimes m'))=\Theta(\mathfrak f(|m'|,|m|)\pi^{p(m)p(m')}\bar\Delta(u)(m'\otimes m)),\] 
so it suffices to show 
\[\mathfrak F\circ \frs(\Delta(u)m\otimes m'))=\mathfrak f(|m'|,|m|)\pi^{p(m)p(m')}\bar\Delta(u)(m'\otimes m)\]
for all $u\in \UU$, hence it is enough to show this equality holds when $u$ is a generator. For $u=J_\nu, K_\nu$, this
is straightforward. The cases $u=E_i$ and $u=F_i$ are similar, so we shall prove the first case:
\begin{align*}
\mathfrak F\circ \frs(\Delta(E_i)&m\otimes m')=\mathfrak f(|m'|,i+|m|)\pi^{p(i)p(m')+p(m)p(m')} m'\otimes E_im\\
&\hspace{5em}+\mathfrak f(i+|m'|,|m|)\pi^{p(m)p(m')}(\pi q)^{d_i\ang{i,|m|}} E_im'\otimes m\\
&=f(|m'|,|m|)\pi^{p(m)p(m')}(E_im'\otimes m+\pi^{p(i)p(m')} q^{-d_i\ang{i,|m'|}} m'\otimes E_im)\\
&=f(|m'|,|m|)\pi^{p(m)p(m')}\bar\Delta(E_i)(m'\otimes m).
\end{align*}
\end{proof}

We thus obtain the following crucial property of $\cR$.

\begin{prop}\label{prop:braidrel}
For any modules $M_1$, $M_2$, and $M_3$, let $\cR_{st}={}^{\mathfrak F}\Theta^{st}\circ\frs_{(st)}$.
Then
\[\cR_{12}\cR_{23}\cR_{12}=\cR_{23}\cR_{12}\cR_{23}:M_1\otimes M_2\otimes M_3\rightarrow M_3\otimes M_2\otimes M_1.\]
\end{prop}
\begin{proof}
First note that if $\sigma(s)<\sigma(t)$, 
$\frs_{\sigma}{}^{\mathfrak F}\Theta^{st}
={}^{\mathfrak F}\Theta^{\sigma(s)\sigma(t)}\frs_{\sigma}$.
Therefore we have $\frs_{(12)}{}^{\mathfrak F}\Theta^{23}={}^{\mathfrak F}\Theta^{13}\frs_{(12)}$,
and $\frs_{(123)}{}^{\mathfrak F}\Theta^{12}={}^{\mathfrak F}\Theta^{23}\frs_{(123)}$,
hence in particular
\[R_{12}R_{23}R_{12}
={}^{\mathfrak F}\Theta^{12}
\circ{}^{\mathfrak F}\Theta^{13}
\circ{}^{\mathfrak F}\Theta^{23}
\circ\frs_{(13)}.\]
Similar manipulations of the right-hand side yield the equality
\[R_{23}R_{12}R_{23}
={}^{\mathfrak F}\Theta^{23}
\circ{}^{\mathfrak F}\Theta^{13}
\circ{}^{\mathfrak F}\Theta^{12}
\circ\frs_{13}.\]
Since $s_{13}$ is a bijection, the proposition follows from the Yang-Baxter equation.
\end{proof}

\begin{rmk}
In \cite[\S 32]{L93}, it is shown that for $\mathfrak g=\mathfrak{sl}(2)$,
which we can view as the $\pi=1$ (i.e. $\tau=\pm 1$) specialization 
of Example \ref{ex:rank1rmat},
we can extend our field $\Qq$ to $\Q(\sqrt q)$ and normalize so that $\mathfrak f$ is
bi-multiplicative; that is $\mathfrak f(m+a,n+b)=\mathfrak f(m,n)\mathfrak f(m,b)
\mathfrak f(a,n)\mathfrak f(a,b)$. 
This is necessary for the maps $\mathcal R$
to satisfy the Hexagon Identities and thus 
define a braiding on the category of finite
dimensional modules.

Note that Example \ref{ex:rank1rmat} shows such a 
renormalization is impossible in general in the $\pi=-1$ case,
so in particular the maps $\mathcal R$ can not be normalized to 
define a braiding on the category of finite dimensional weight modules.
It is possible to overcome this difficulty by restricting the class of modules to those of "even" highest weight,
or by expanding the definition of $\frf$ to a function on $\bigrset\times\bigrset$, but we shall not
need this at present.
\end{rmk}

\section{Diagrammatic Calculus and Knot invariants}
\label{sec:diagcalc}

We will now interpret the $\UU$-module homomorphisms 
in terms of planar diagrams. At first, these diagrams 
should be interpreted as slice diagrams; that is, diagrams
together with vertical slices at various heights 
such that between consecutive slices
is an elementary diagram corresponding to a $\UU$-module 
homomorphism. However, we will
ultimately see that diagrams which can be identified by planar 
isotopies yield the same morphisms.

\subsection{Cups, caps, and crossings}

Recall that $\coqtr_\lambda$, $\coev_\lambda$, $\qtr_\lambda$, 
and $\ev_\lambda$ are the maps defined
in Lemma \ref{lemma:evsandcoevs} where $V=V(\lambda)$. 
Likewise, let $\cR_{\pm \lambda, \pm \mu}:V(\pm \lambda)\otimes V(\pm \mu)\rightarrow V(\pm \mu)\otimes V(\pm \lambda)$ be the map defined in
Proposition \ref{prop:Rmat}. Furthermore, we will use the notation $1_{\pm \lambda}=1_{V(\pm\lambda)}$.

We will now begin to represent our maps via a graphical calculus 
in anticipation of constructing tangle invariants.
Specifically, we follow \cite{T,ADO} and interpret maps between tensor products of the
modules $V(\pm \lambda)$ for various $\lambda\in X^+$ as
sliced oriented tangle diagrams with $X^+$-labeled strands; 
a concise exposition of this approach is lain out in \cite[Chapter 3]{Oht}.
The elementary oriented tangle diagrams are interpreted as follows. (Note that while sideways-oriented crossings aren't
considered elementary, we include them here for convenience
in later arguments.)
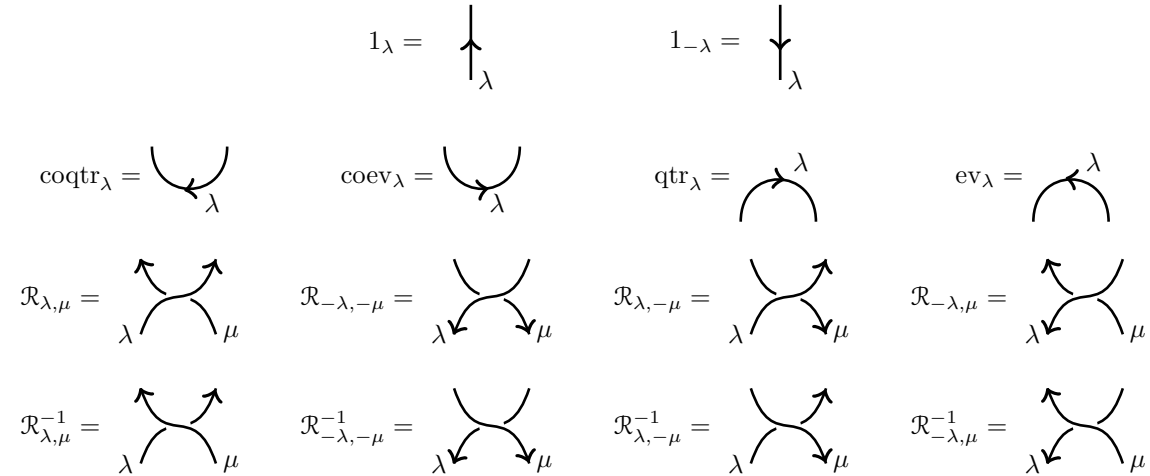
\begin{center}
\begin{tabular}{ccccccc}
&&
\begin{tikzpicture}
\draw (-2, 2) node {$1_{\lambda}=$};
\idsup{-1}{1.5}{1}{1}{1}
\draw (-.8,1.5) node {$\lambda$};
\end{tikzpicture}
&&
\begin{tikzpicture}
\draw (0, 2) node {$1_{-\lambda}=$};
\idsdown{1}{1.5}{1}{1}{1}
\draw (1.2,1.5) node {$\lambda$};
\end{tikzpicture}\\\\
$\coqtr_{\lambda}=$ \hctikz{
\cwcup{1}{-.5}{1}{.75}
\draw (1.8,-.5) node {$\lambda$};
}
&&$\coev_{\lambda}=$ 
\hctikz{
\ccwcup{1}{-.5}{1}{.75}
\draw (1.7,-.5) node {$\lambda$};}
&&
$\qtr_{\lambda}=$ 
\hctikz{
\cwcap{1}{-.5}{1}{.75}
\draw (1.8,.3) node {$\lambda$};
}
&&
$\ev_{\lambda}=$ 
\hctikz{
\ccwcap{1}{-.5}{1}{.75}
\draw (1.8,.3) node {$\lambda$};
}
\\\\
$\cR_{\lambda,\mu}= $
\hctikz{
\rcrossup{0}{0}{1}{1}{1}
\draw (-.2,0) node {$\lambda$};
\draw (1.2,0) node {$\mu$};
}
&&

$\cR_{-\lambda,-\mu}= $
\hctikz{
\rcrossdown{0}{0}{1}{1}{1}
\draw (-.2,0) node {$\lambda$};
\draw (1.2,0) node {$\mu$};
}
&&
$\cR_{\lambda,-\mu}= $
\hctikz{
\NESE{0}{0}{1}{1}{1}
\draw (-.2,0) node {$\lambda$};
\draw (1.2,0) node {$\mu$};
}
&&
$\cR_{-\lambda,\mu}= $
\hctikz{
\SWNW{0}{0}{1}{1}{1}
\draw (-.2,0) node {$\lambda$};
\draw (1.2,0) node {$\mu$};
}

\\\\
$\cR_{\lambda,\mu}^{-1}= $
\hctikz{
\lcrossup{0}{0}{1}{1}{1}
\draw (-.2,0) node {$\lambda$};
\draw (1.2,0) node {$\mu$};
}
&&

$\cR_{-\lambda,-\mu}^{-1}= $
\hctikz{
\lcrossdown{0}{0}{1}{1}{1}
\draw (-.2,0) node {$\lambda$};
\draw (1.2,0) node {$\mu$};
}
&&
$\cR_{\lambda,-\mu}^{-1}= $
\hctikz{
\SENE{0}{0}{1}{1}{1}
\draw (-.2,0) node {$\lambda$};
\draw (1.2,0) node {$\mu$};
}
&&
$\cR_{-\lambda,\mu}^{-1}= $
\hctikz{
\NWSW{0}{0}{1}{1}{1}
\draw (-.2,0) node {$\lambda$};
\draw (1.2,0) node {$\mu$};
}
\end{tabular}
\end{center}
We construct more general diagrams from these elementary ones by
the following constructions. If $\hctikz{
\ids{.2}{.9}{.2}{.1}{3}
\standin{0}{.1}{.8}{.8}{T}
\ids{.1}{0}{.2}{.1}{4}
}$ is some diagram denoting the morphism $\phi$ and $\hctikz{
\ids{.1}{.9}{.2}{.1}{4}
\standin{0}{.1}{.8}{.8}{S}
\ids{.4}{0}{.2}{.1}{1}
}$ is some diagram denoting the morphism $\psi$, then we can combine them
as
\begin{itemize}
\item the {\em horizontal} composition 
$\hctikz{
\ids{.2}{.9}{.2}{.1}{3}
\standin{0}{.1}{.8}{.8}{T}
\ids{.1}{0}{.2}{.1}{4}
\ids{1.1}{.9}{.2}{.1}{4}
\standin{1}{.1}{.8}{.8}{S}
\ids{1.4}{0}{.2}{.1}{1}
}$ which denotes the tensor product $\phi\otimes \psi$
\item the {\em vertical} composition 
$\hctikz{
\ids{.2}{.9}{.2}{.1}{3}
\standin{0}{.1}{.8}{.8}{T}
\ids{.1}{0}{.2}{.1}{4}
\ids{.1}{-.1}{.2}{.1}{4}
\standin{0}{-.9}{.8}{.8}{S}
\ids{.4}{-1}{.2}{.1}{1}}$ which denotes the composition $\phi\circ \psi$,
or zero if this composition is undefined (which is to say, when the
strands on top of S don't match the number and labelling of the strands on the bottom of T).
\end{itemize}

We will say two diagrams are equal if the corresponding morphisms
agree. Note that, by construction and by Lemma \ref{prop:braidrel}, 
the following diagrams are equal for any choice
of orientation and labeling of strands:

\begin{equation}\label{eq:comm and ids}
\hctikz{
\ids{.2}{.9}{.2}{.1}{3}
\standin{0}{.1}{.8}{.8}{T}
\ids{.1}{0}{.2}{.1}{4}
}\quad =\quad
\hctikz{
\ids{.2}{.9}{.2}{.5}{3}
\standin{0}{.1}{.8}{.8}{T}
\ids{.1}{0}{.2}{.1}{4}
}\quad =\quad
\hctikz{
\ids{.2}{.9}{.2}{.1}{3}
\standin{0}{.1}{.8}{.8}{T}
\ids{.1}{-.4}{.2}{.5}{4}
},\qquad
\hctikz{
\ids{.2}{.9}{.2}{1.1}{3}
\standin{0}{.1}{.8}{.8}{T}
\ids{.1}{0}{.2}{.1}{4}
\ids{1.3}{1.9}{.2}{.1}{2}
\standin{1}{1.1}{.8}{.8}{S}
\ids{1.4}{0}{.2}{1.1}{1}
}\quad =\quad
\hctikz{
\ids{.2}{1.9}{.2}{.1}{3}
\standin{0}{1.1}{.8}{.8}{T}
\ids{.1}{0}{.2}{1.1}{4}
\ids{1.3}{.9}{.2}{1.1}{2}
\standin{1}{.1}{.8}{.8}{S}
\ids{1.4}{0}{.2}{.1}{1}
}
\end{equation}

\begin{equation}\label{eq:diagbraidrels}
\hctikz{\lcross{0}{0}{.75}{.5}
\rcross{0}{.75}{.75}{.5}
}
= \quad
\hctikz{
\ids{0}{0}{1}{1.5}{1}
\ids{.5}{0}{1}{1.5}{1}
}
= \quad
\hctikz{\rcross{0}{0}{.75}{.5}
\lcross{0}{.75}{.75}{.5}
}\ ,\qquad
\hctikz{\rcross{0}{0}{.5}{.5}\ids{1}{0}{.5}{.5}{1} \rcross{.5}{.5}{.5}{.5}\ids{0}{.5}{.5}{.5}{1}\rcross{0}{1}{.5}{.5}\ids{1}{1}{.5}{.5}{1}}
=
\hctikz{\rcross{.5}{0}{.5}{.5}\ids{0}{0}{.5}{.5}{1} \rcross{0}{.5}{.5}{.5}\ids{1}{.5}{.5}{.5}{1}\rcross{.5}{1}{.5}{.5}\ids{0}{1}{.5}{.5}{1}}
\end{equation}

\vspace{1em}

In \eqref{eq:comm and ids}, the symbols $\hctikz{
\ids{.2}{.9}{.2}{.1}{3}
\standin{0}{.1}{.8}{.8}{T}
\ids{.1}{0}{.2}{.1}{4}
}$ and $\hctikz{
\ids{.3}{.9}{.2}{.1}{2}
\standin{0}{.1}{.8}{.8}{S}
\ids{.4}{0}{.2}{.1}{1}
}$ stand for arbitrary sub-diagrams  with an arbitrary number of strands protruding from the top and bottom and with an arbitrary labeling of strands.
\subsection{Graphical identities}\label{sec:graphid}

Now we shall prove some more substantial diagrammatic identities.

\begin{lem}\label{lem:straightening}
We have an equality of diagrams
\[\hctikz{
\ids{0}{.75}{.5}{.75}{1}\dcap{.5}{.75}{.5}{.5}
\dcup{0}{.25}{.5}{.5}\ids{1}{0}{.5}{.75}{1}
}\quad =\quad  \hctikz{\ids{0}{0}{1}{1.5}{1}} \quad=\quad
\hctikz{
\ids{1}{.75}{.5}{.75}{1}\dcap{0}{.75}{.5}{.5}
\dcup{.5}{.25}{.5}{.5}\ids{0}{0}{.5}{.75}{1}
}\]
For any choice of orientation or labeling of the strand.
\end{lem}
\begin{proof}
This follows by choosing a homogeneous basis for the module and applying the
definitions; we will prove the equality
\[\hctikz{
\idsdown{0}{.75}{.5}{.75}{1}\dcap{.5}{.75}{.5}{.5}
\dcup{0}{.25}{.5}{.5}\idsdown{1}{0}{.5}{.75}{1}
\draw (1.3,0) node {$\lambda$};
}\quad =\quad  \hctikz{\idsdown{0}{0}{1}{1.5}{1}\draw (0.3,0) node {$\lambda$};}
\] 
in detail, as the other cases are similar.
In terms of morphisms, we wish to show 
$(1_{-\lambda}\otimes \qtr_\lambda)\circ (\coev_\lambda\otimes 1_{-\lambda})=1_{-\lambda}$. Let $B$ be a homogeneous basis of $V(\lambda)$ 
and $B^*$ the dual basis of $V(\lambda)^*$.
Then for any $b_0\in B$,
\begin{align*}
(1_{-\lambda}\otimes \qtr_\lambda)(\coev_\lambda\otimes 1_{-\lambda})(b_0^*)
&=\sum_{b\in B} \pi^{p(b)}q^{\ang{\tilde\rho,|b|}}(1_{-\lambda}\otimes \qtr_\lambda)(b^*\otimes b\otimes b_0^*)\\
&=\sum_{b\in B} b_0^*(b)b^*=b_0^*.
\end{align*}
\end{proof}

\begin{lem}\label{lem:reide2} For $\lambda\in X^+$, we have an equality of diagrams
\[\tag{a}\hctikz{
\ids{0}{.75}{.5}{.75}{1}\dcap{.5}{.75}{.5}{.5}
\rcrossup{0}{0}{.75}{.5}\idsdown{1}{0}{.5}{.75}{1}
\ids{0}{-.75}{.5}{.75}{1}\dcup{.5}{-.5}{.5}{.5}
\draw (-.3,0) node {$\lambda$};
}\quad =\mathfrak f(\lambda,\lambda)q^{-\ang{\tilde\rho,\lambda}}  \hctikz{\idsup{0}{0}{1}{1.5}{1}
\draw (-.3,0) node {$\lambda$};} \quad= \pi^{P(\lambda)}\hctikz{
\ids{1}{.75}{.5}{.75}{1}\dcap{0}{.75}{.5}{.5}
\rcrossup{.5}{0}{.75}{.5}\idsdown{0}{0}{.5}{.75}{1}
\ids{1}{-.75}{.5}{.75}{1}\dcup{0}{-.5}{.5}{.5}
\draw (-.3,0) node {$\lambda$};
}
\]
\[\tag{b}\hctikz{
\ids{1}{.75}{.5}{.75}{1}\dcap{0}{.75}{.5}{.5}
\lcrossup{.5}{0}{.75}{.5}\idsdown{0}{0}{.5}{.75}{1}
\ids{1}{-.75}{.5}{.75}{1}\dcup{0}{-.5}{.5}{.5}
\draw (-.3,0) node {$\lambda$};
}\quad =\mathfrak f(\lambda,\lambda)^{-1}q^{\ang{\tilde\rho,\lambda}}  \hctikz{\idsup{0}{0}{1}{1.5}{1}
\draw (-.3,0) node {$\lambda$};}  \quad= \pi^{P(\lambda)}\hctikz{
\ids{0}{.75}{.5}{.75}{1}\dcap{.5}{.75}{.5}{.5}
\lcrossup{0}{0}{.75}{.5}\idsdown{1}{0}{.5}{.75}{1}
\ids{0}{-.75}{.5}{.75}{1}\dcup{.5}{-.5}{.5}{.5}
\draw (-.3,0) node {$\lambda$};
}
\]
\[\tag{c}\hctikz{
\ids{0}{.75}{.5}{.75}{1}\dcap{.5}{.75}{.5}{.5}
\rcrossdown{0}{0}{.75}{.5}\idsup{1}{0}{.5}{.75}{1}
\ids{0}{-.75}{.5}{.75}{1}\dcup{.5}{-.5}{.5}{.5}
\draw (-.3,0) node {$\lambda$};
}\quad =\mathfrak f(\lambda,\lambda)q^{-\ang{\tilde\rho,\lambda}}  \hctikz{\idsdown{0}{0}{1}{1.5}{1}
\draw (-.3,0) node {$\lambda$};} \quad= \pi^{P(\lambda)}\hctikz{
\ids{1}{.75}{.5}{.75}{1}\dcap{0}{.75}{.5}{.5}
\rcrossdown{.5}{0}{.75}{.5}\idsup{0}{0}{.5}{.75}{1}
\ids{1}{-.75}{.5}{.75}{1}\dcup{0}{-.5}{.5}{.5}
\draw (-.3,0) node {$\lambda$};
}
\]
\[\tag{d}\hctikz{
\ids{1}{.75}{.5}{.75}{1}\dcap{0}{.75}{.5}{.5}
\lcrossdown{.5}{0}{.75}{.5}\idsup{0}{0}{.5}{.75}{1}
\ids{1}{-.75}{.5}{.75}{1}\dcup{0}{-.5}{.5}{.5}
\draw (-.3,0) node {$\lambda$};
}\quad =\mathfrak f(\lambda,\lambda)^{-1}q^{\ang{\tilde\rho,\lambda}}  \hctikz{\idsdown{0}{0}{1}{1.5}{1}
\draw (-.3,0) node {$\lambda$};}  \quad= \pi^{P(\lambda)}\hctikz{
\ids{0}{.75}{.5}{.75}{1}\dcap{.5}{.75}{.5}{.5}
\lcrossdown{0}{0}{.75}{.5}\idsup{1}{0}{.5}{.75}{1}
\ids{0}{-.75}{.5}{.75}{1}\dcup{.5}{-.5}{.5}{.5}
\draw (-.3,0) node {$\lambda$};
}
\]
\end{lem}
\begin{proof}
The proofs of (a)-(d) are all similar, so
we will only prove (a).
First, let us denote
\[\phi=\hctikz{
	\ids{0}{.75}{.5}{.75}{1}\dcap{.5}{.75}{.5}{.5}
	\rcrossup{0}{0}{.75}{.5}\idsdown{1}{0}{.5}{.75}{1}
	\ids{0}{-.75}{.5}{.75}{1}\dcup{.5}{-.5}{.5}{.5}
	\draw (-.3,0) node {$\lambda$};
}=(1_\lambda\otimes \qtr_{\lambda})\circ (\cR_{\lambda,\lambda}\otimes 1_\lambda)\circ (1_\lambda\otimes \coqtr_{\lambda}),\]
\[\psi=\hctikz{
	\ids{1}{.75}{.5}{.75}{1}\dcap{0}{.75}{.5}{.5}
	\rcrossup{.5}{0}{.75}{.5}\idsdown{0}{0}{.5}{.75}{1}
	\ids{1}{-.75}{.5}{.75}{1}\dcup{0}{-.5}{.5}{.5}
	\draw (-.3,0) node {$\lambda$};
}=(\ev_{\lambda}\otimes 1_\lambda)\circ (1_\lambda\otimes \cR_{\lambda,\lambda})\circ (\coev_\lambda\otimes 1_\lambda).\] 

Since $\phi$ and $\psi$ are $\UU$-module homomorphisms from $V(\lambda)$ to $V(\lambda)$.,
$\phi$ and $\psi$ must each be a multiple
of the identity which is completely determined by
the image of an extremal weight vector, so let $v_\lambda\in V(\lambda)_\lambda$ and $v_{-\lambda}\in V(\lambda)_{-\lambda}$
be nonzero highest- and lowest-weight vectors.
Then if $B(\lambda)$ is a homogeneous basis of $V(\lambda)$, then 
\begin{align*}
\phi(v_{\lambda})&=(1_\lambda\otimes \qtr_{\lambda})\circ (\cR_{\lambda,\lambda}\otimes 1_\lambda)
\parens{\sum_{v\in B(\lambda)} v_\lambda\otimes v\otimes v^*}\\
&=(1_\lambda\otimes \qtr_{\lambda})\parens{\sum_{v\in B(\lambda)}\mathfrak f(|v|,\lambda)
v\otimes v_\lambda\otimes v^*}=\mathfrak f(\lambda,\lambda)q^{-\ang{\tilde\rho,\lambda}}v_\lambda,
\end{align*}
and thus $\phi=\frf(\lambda,\lambda)q^{-\ang{\tilde\rho,\lambda}} 1_\lambda$.
Likewise,we compute
\begin{align*}
\psi(v_{-\lambda})&=(\ev_{\lambda}\otimes 1_\lambda)\circ (1_\lambda\otimes \cR_{\lambda,\lambda})
\parens{\sum_{v\in B(\lambda)} \pi^{p(v)}q^{\ang{\tilde\rho,|v|}}v^*\otimes v\otimes v_{-\lambda}}\\
&=(\ev_{\lambda}\otimes 1_\lambda)\parens{\sum_{v\in B(\lambda)}\pi^{p(v)}q^{\ang{\tilde\rho,|v|}}\mathfrak f(|v|,-\lambda) \pi^{p(v)p(\lambda)}  v^*\otimes v_{-\lambda}\otimes v}=q^{\ang{\tilde\rho,-\lambda}}\mathfrak f(-\lambda,-\lambda)  v_{-\lambda},
\end{align*}
and thus $\psi=\frf(-\lambda,-\lambda)q^{-\ang{\tilde\rho,\lambda}} 1_\lambda$.
Now the result follows from Lemma \ref{lem:f function} (3).
\end{proof}

\begin{lem}\label{lem:rotating crossings} We have an equality of diagrams
\begin{equation*}\tag{a}
\hctikz{
\NESE{0}{0}{1}{1}
\draw (-.3,0) node {$\lambda$};
\draw (-.3,1) node {$\mu$};
}\quad = \mathfrak r(\mu,\lambda)
\hctikz{
\idsdown{1}{.75}{.5}{.75}{1}\idsup{1.5}{0}{.5}{1.5}{1}\dcap{0}{.75}{.5}{.5}
\lcrossdown{.5}{0}{.75}{.5}\idsup{0}{-.75}{.5}{1.5}{1}\ids{.5}{-.75}{.5}{.75}{1}
\dcup{1}{-.5}{.5}{.5}
\draw (-.3,-.8) node {$\lambda$};
\draw (.7,-.8) node {$\mu$};
}=\pi^{P(\mu)P(\lambda)}\frr(\mu,\lambda)\quad
\hctikz{
\idsdown{0}{0}{.5}{1.5}{1}\ids{.5}{.75}{.5}{.75}{1}\dcap{1}{.75}{.5}{.5}
\lcrossup{.5}{0}{.75}{.5}\idsdown{1.5}{-.75}{.5}{1.5}{1}\ids{1}{-.75}{.5}{.75}{1}
\dcup{0}{-.5}{.5}{.5}
\draw (.7,-.8) node {$\lambda$};
\draw (1.3,-.8) node {$\mu$};
}
\end{equation*}
\begin{equation*}\tag{b}
\hctikz{
\NWSW{0}{0}{1}{1}
\draw (-.3,0) node {$\mu$};
\draw (-.3,1) node {$\lambda$};
}\quad = \pi^{P(\mu)P(\lambda)}\frr(\mu,\lambda)^{-1}
\hctikz{
\ids{1}{.75}{.5}{.75}{1}\idsdown{1.5}{0}{.5}{1.5}{1}\dcap{0}{.75}{.5}{.5}
\rcrossup{.5}{0}{.75}{.5}\idsdown{0}{-.75}{.5}{1.5}{1}\ids{.5}{-.75}{.5}{.75}{1}
\dcup{1}{-.5}{.5}{.5}
\draw (-.3,-.8) node {$\mu$};
\draw (.7,-.8) node {$\lambda$};
}=\mathfrak r(\mu,\lambda)^{-1}\quad
\hctikz{
\idsup{0}{0}{.5}{1.5}{1}\ids{.5}{.75}{.5}{.75}{1}\dcap{1}{.75}{.5}{.5}
\rcrossdown{.5}{0}{.75}{.5}\idsup{1.5}{-.75}{.5}{1.5}{1}\ids{1}{-.75}{.5}{.75}{1}
\dcup{0}{-.5}{.5}{.5}
\draw (.7,-.8) node {$\mu$};
\draw (1.3,-.8) node {$\lambda$};
}
\end{equation*}
for any $\lambda,\mu\in X^+$.
\end{lem}

\begin{proof}
The proof of (a) and (b) being similar, we shall only prove (a) here.
First, unpacking the graphical representation, we see that (a) is equivalent to
\[\cR_{\lambda,-\mu}=\frr(\mu,\lambda) \phi=\pi^{p(\mu)p(\lambda)}\frr(\mu,\lambda)\psi\]
where $\phi$ and $\psi$ are the compositions
\[\phi=(\qtr_\lambda\otimes 1_{-\mu}\otimes 1_\lambda)\circ (1_\lambda\otimes \cR_{-\mu,-\lambda}^{-1}\otimes 1_\lambda)\circ(1_\lambda\otimes 1_{-\mu}\otimes \coev_\lambda),\]
\[\psi=(1_{-\mu}\otimes 1_\lambda\otimes \qtr_\mu)\circ (1_{-\mu}\otimes \cR_{\mu,\lambda}^{-1}\otimes 1_{-\mu})\circ( \coev_\mu\otimes 1_\lambda\otimes 1_{-\mu}).\]
Let $B(\lambda)$ be a homogeneous basis for $V(\lambda)$.
Let $v_0\in B(\lambda)_{\kappa}$ and $w_0\in V(\mu)_{\xi}$ for some $\kappa,\xi\in X$. 
We shall compare the images of our three maps on $v_0\otimes w_0^*$.

First, note that 
\begin{equation}\label{eq:Rlammu*}
\cR_{\lambda,-\mu}(v_0\otimes w_0^*)=\pi^{p(w_0)p(v_0)}\mathfrak f(-\xi,\kappa)\sum_{\nu}(-1)^{\height\,\nu}\pi^{\bp(\nu)}\pi_\nu q_\nu\sum_{b\in \BB_\nu} \pi^{p(\nu)p(w_0)} b^-w_0^*\otimes (b^*)^+v_0.
\end{equation}

For $\phi$, first let us note the effect of each  map
in the composition separately. The graphical representation
tells us which tensor factors are impacted at each step, so
we restrict our view to these tensor factors when computing these maps. First,
we have the coevaluation which adds two tensor factors on the right:
\[\coev_\lambda(1)=\sum_{v\in B(\lambda)} 
\pi^{p(v)}q^{\ang{\tilde\rho,|v|}} v^*\otimes v.\]
Next, we apply $\cR_{-\mu,-\lambda}^{-1}=\frs \circ \frF^{-1}\circ \bar\Theta$ to the middle
tensor factors, so
\[\cR_{-\mu,-\lambda}^{-1}(w_0^*\otimes v^*)=
\sum_\nu\mathfrak f(-\xi-\nu,-|v|+\nu)^{-1}\pi^{p(w_0)p(v)+p(\nu)p(v)} q^{\frac{\nu\cdot\nu}2}\sum_{b\in \BB_\nu} \sigma(b^*)^+v^*\otimes b^-w_0^*;
\]
Finally, we apply the quantum trace to the two tensor factors on the left, 
hence we need to compute $\qtr(v_0\otimes\sigma(b^*)^+v^*)$. Since $x^*(y)=0$ unless $\bigrdeg{x}=\bigrdeg{y}$ (that is, unless $x$ and $y$ have the same weight and parity),
we can assume $|v|=\kappa+\nu$ and
$p(v)=p(v_0)+p(\nu)$. Then we have
\begin{align*}\qtr_\lambda&(v_0\otimes\sigma(b^*)^+v^*)=
\pi^{p(v_0)}q^{-\ang{\tilde\rho,|v_0|}} (\sigma(b^*)^+v^*)(v_0)\\
&=(-1)^{\height \nu} \pi^{\bp(\nu)+p(v_0)+p(\nu)p(v_0)+p(\nu)}q^{-\frac{\nu\cdot\nu}{2}-\ang{\tilde\rho,\kappa}}(\pi q)^{-\ang{\tilde\nu,\kappa}}q_{-\nu}v^*((b^*)^+v_0).
\end{align*}
Putting these computations together, we see that 
\begin{align*}
\phi(v_0\otimes w_0^*)&=\sum_{v\in B(\lambda)} \sum_\nu\sum_{b\in \BB_\nu}
\pi^{p(v_0)+p(\nu)}q^{\ang{\tilde\rho,\kappa+\nu}}\\
&\hspace{2em}\times \mathfrak f(-\xi-\nu,-\kappa)^{-1}\pi^{p(w_0)p(v_0)+p(w_0)p(\nu)+p(\nu)p(v_0)+p(\nu)}q^{\frac{\nu\cdot\nu}2}\\
&\hspace{2em}\times (-1)^{\height \nu} \pi^{\bp(\nu)+p(v_0)+p(\nu)p(v_0)+p(\nu)}q^{-\frac{\nu\cdot\nu}{2}-\ang{\tilde\rho,\kappa}}(\pi q)^{-\ang{\tilde\nu,\kappa}}q_{-\nu} v^*((b^*)^+v_0) b^-w_0^*\otimes v\\
&=\sum_\nu(-1)^{\height \nu}\mathfrak f(-\xi-\nu,-\kappa)^{-1}(\pi q)^{-\ang{\tilde\nu,\kappa}}
\pi_\nu q^{\ang{\tilde\rho,\nu}}q_{-\nu}\pi^{p(v_0)p(w_0)+\bp(\nu)}\\
&\hspace{2em}\times  \sum_{b\in \BB_\nu}
 \pi^{p(w_0)p(\nu)}b^-w_0^*\otimes\parens{\sum_{v\in B(\lambda)}v^*((b^*)^+v_0)v}.
 \end{align*}
But note that 
$\frf(-\xi-\nu,-\kappa)(\pi q)^{\ang{\tilde\nu,\kappa}}=\frf(-\xi,-\kappa)$, $q^{\ang{\tilde\rho,\nu}}=q_\nu^2$, and $\sum_{v\in B(\lambda)}v^*((b^*)^+v_0)v=(b^*)^+v_0$.
Therefore, we have
 \begin{align*}
 \phi(v_0\otimes w_0^*)&=\mathfrak f(-\xi,-\kappa)^{-1}\pi^{p(v_0)p(w_0)}\sum_\nu(-1)^{\height \nu}\pi^{\bp(\nu)}\sum_{b\in \BB_\nu}\pi_\nu q_{\nu}
 \pi^{p(w_0)p(\nu)}b^-w_0^*\otimes(b^*)^+v_0\\
 &=\mathfrak r(-\xi,\kappa)^{-1}R_{\lambda,-\mu}(v_0\otimes w_0^*).
 \end{align*}
Finally, since $-\xi\in\mu+\Z[I]$ and $\kappa\in \lambda+\Z[I]$,
we can apply Lemma \ref{lem:f function}(1) to conclude that
$\phi=\frr (\mu,\lambda)^{-1}R_{\lambda,-\mu}$.
A similar computation shows that $\psi=\frl(\mu,\lambda)^{-1}\cR_{\lambda,-\mu}$, and the 
result then follows from Lemma \ref{lem:f function}.
\end{proof}

Note that by identifying inverse maps in Lemma \ref{lem:rotating crossings} (a) and (b), we obtain the following corollary.

\begin{cor}\label{cor:rotating crossings inverse}
We have an equality of diagrams
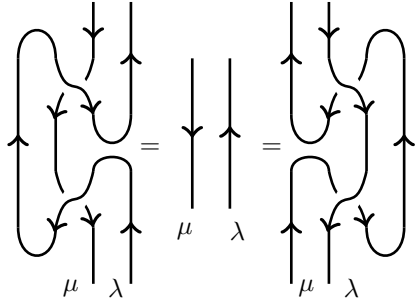
\[\hctikz{
\idsdown{1}{2.25}{.5}{.75}{1}\idsup{1.5}{1.5}{.5}{1.5}{1}\dcap{0}{2.25}{.5}{.5}
\lcrossdown{.5}{1.5}{.75}{.5}
\dcup{1}{1}{.5}{.5}
\idsup{0}{0}{.5}{2.25}{1}\ids{.5}{.75}{.5}{.75}{1}\dcap{1}{.75}{.5}{.25}
\rcrossdown{.5}{0}{.75}{.5}\idsup{1.5}{-.75}{.5}{1.5}{1}\ids{1}{-.75}{.5}{.75}{1}
\dcup{0}{-.5}{.5}{.5}
\draw (.7,-.8) node {$\mu$};
\draw (1.3,-.8) node {$\lambda$};
}=\hctikz{
\idsdown{0}{0}{.5}{2}{1}
\idsup{.5}{0}{.5}{2}{1}
\draw (-.1,-.3) node {$\mu$};
\draw (.6,-.3) node {$\lambda$};
}=\hctikz{
\idsdown{.5}{2.25}{.5}{.75}{1}\idsup{0}{1.5}{.5}{1.5}{1}\dcap{1}{2.25}{.5}{.5}
\lcrossdown{.5}{1.5}{.75}{.5}
\dcup{0}{1}{.5}{.5}
\idsup{1.5}{0}{.5}{2.25}{1}\ids{1}{.75}{.5}{.75}{1}\dcap{0}{.75}{.5}{.25}
\rcrossdown{.5}{0}{.75}{.5}\idsup{0}{-.75}{.5}{1.5}{1}\ids{.5}{-.75}{.5}{.75}{1}
\dcup{1}{-.5}{.5}{.5}
\draw (.2,-.8) node {$\mu$};
\draw (.8,-.8) node {$\lambda$};
}
\]
\end{cor}

Finally, we show a somewhat more involved identity, which will lead us to our the final
result.

\begin{lem}\label{lem:crossthruturn}
 We have an equality of diagrams
\[
\hctikz{
\dcap{0}{.75}{1.5}{1}\dcap{.5}{.75}{.5}{.5}
\rcross{1}{0}{.75}{.5}
\ids{0}{0}{.5}{.75}{2}
\draw (1,-.2) node {$\lambda$};
\draw (1.6,-.2) node {$\mu$};
}\quad = \pi^{P(\mu)P(\lambda)}
\hctikz{
\dcap{0}{.75}{1.5}{1}\dcap{.5}{.75}{.5}{.5}
\rcross{0}{0}{.75}{.5}
\ids{1}{0}{.5}{.75}{2}
\draw (1,-.2) node {$\lambda$};
\draw (1.6,-.2) node {$\mu$};
}
\]
for any choice of orientation.
\end{lem}
\begin{proof}
In order to prove the identity without referring 
to a particular orientation,
it will be convenient to introduce the following notation.
Suppose $m\in V(\zeta)$ and $n\in V(-\zeta)$ for some $\zeta\in X^+$. Let us denote
by $(n,m)$ (respectively $(m,n)$) the evaluation $\ev_\zeta(n\otimes m)$ (respectively, the quantum
trace $\qtr_\zeta(m\otimes n)$).  In particular, one may think of
$(-,-)$ as a pairing on $V(\zeta)\oplus V(-\zeta)$ 
satisfying, for $v,w\in V(\zeta)$, 
\begin{equation}\label{eq:pairing v+dual}
\begin{array}{c}
(v,w)=(v^*,w^*)=0,\quad (v,w^*)=\pi^{p(v)p(w)}q^{-\ang{\tilde\rho,|v|}} (w^*,v),\\
(uv,w^*)=\pi^{p(u)p(v)}(v,S(u)w^*),\quad (uw^*,v)=\pi^{p(u)p(w)}(w^*,S(u)v).
\end{array}
\end{equation}
Indeed, all the statements of \eqref{eq:pairing v+dual} are obvious except
$(uv,w^*)=\pi^{p(u)p(v)}(v,S(u)w^*)$, which follows from a simple calculation
on the generators: for example,
\[(E_iv,w^*)=
\pi^{p(v)p(w)}q^{-\ang{\tilde\rho,|v|}}q_i^{-2}(-E_i\tJ_i^{-1}\tK_i^{-1}w^*)(v)
=\pi^{p(v)p(i)}(v,S(E_i)w^*)\]

In this proof 
we will use the notation $(-,-)$ as shorthand for 
$\ev_\zeta$ and $\qtr_\zeta$ for both $\zeta=\lambda,\mu$
with the intended map (and highest weight) being clear from context.
Using this notation, the diagram equality is equivalent to showing that the maps
\[\psi=(-,-)\circ\parens{1_{s\mu}\otimes (-,-)\otimes 1_{-s\mu}}\circ \parens{\cR_{s \lambda,t \mu}\otimes 1_{-s\lambda}\otimes 1_{-t\mu}}\]
\[\phi=(-,-)\circ\parens{1_{t\lambda}\otimes (-,-)\otimes 1_{-t\lambda}}\circ \parens{1_{s\lambda}\otimes 1_{t\mu}\otimes \cR_{-s \lambda,-t \mu}}\]
are $\pi^{P(\mu)P(\lambda)}$ multiples of each other 
for any choice of $s,t\in\set{1,-1}$. 

Let $w\in V(s\lambda)$, $x\in V(t\mu)$, $y\in V(-s\lambda)$, and $z\in V(-t\mu)$, where $V(-\xi)=V(\xi)^*$ for $\xi\in X^+$.
Then on one hand,
\[\psi(w\otimes x\otimes y \otimes z)=\sum_{\nu}\sum_{b\in \mathbf B_\nu}\pi^{p(x)p(w)} \mathfrak f(|x|,|w|) (-1)^{\height \nu} \pi^{\mathbf p(\nu)}
\pi_\nu q_\nu \pi^{p(\nu)p(x)} (b^-x,z)((b^*)^+w,y).\]
On the other hand, using the representation of $\Theta$ in the basis $\sigma(\mathbf B)$,
\[\phi(w\otimes x\otimes y \otimes z)=\sum_{\nu}\sum_{b\in \mathbf B_\nu}\pi^{p(y)p(z)} \mathfrak f(|z|,|y|) (-1)^{\height \nu} \pi^{\mathbf p(\nu)}
\pi_\nu q_\nu \pi^{p(\nu)p(z)} (x,\sigma(b)^-z)(w,\sigma(b^*)^+y).\]
Thus to see that 
$\psi(w\otimes x\otimes y\otimes z)=\pi^{P(\mu)P(\lambda)}\phi(w\otimes x\otimes y\otimes z)$, and hence that $\psi=\pi^{p(\mu)p(\lambda)}\phi$
since $w,x,y,z$ are arbitrary, 
it is enough to show that $l=\pi^{P(\mu)P(\lambda)}r$, where
\[l=\pi^{p(y)p(z)+p(\nu)p(z)} \mathfrak f(|z|,|y|)(x,\sigma(b)^-z)(w,\sigma(b^*)^+y)\]
\[r=\pi^{p(w)p(x)+p(\nu)p(x)}\mathfrak f(|x|,|w|) (b^-x,z)((b^*)^+w,y)\]
Using the properties of $(-,-)$ (see \eqref{eq:pairing v+dual})
and $S$ (see \eqref{eq:antipode formula}), we see that
\[(x,\sigma(b)^-z)(w,\sigma(b^*)^+y)=
\pi^{p(x)p(\nu)+p(w)p(\nu)}q^{-\nu\cdot\nu+\ang{\tilde\nu,|x|}}(\pi q)^{-\ang{\tilde\nu,|w|}}(b^-x,z)((b^*)^+w,y)\]
Note that $l$,$r$ are both zero unless 
$-\bigrdeg{x}=\bigrdeg{z}-\nu$ and $-\bigrdeg{w}=\bigrdeg{y}+\nu$.
In particular, $l$ and $r$ are both zero unless
$p(y)=p(w)+p(\nu)$, $p(z)=p(x)+p(\nu)$, in which case
\[p(y)p(z)+p(\nu)p(z)+p(x)p(\nu)+p(w)p(\nu)\equiv p(w)p(x)+p(w)p(\nu)\text{ (mod 2)}.\]
Likewise, $l$,$r$ are both zero unless 
$-|y|=|w|+\nu$, $-|z|=|x|-\nu$, in which case
\[\mathfrak f(|z|,|y|)q^{-\nu\cdot\nu+\ang{\tilde\nu,|x|-|w|}}=\mathfrak f(-|x|,-|w|).\]
Finally, note that $\mathfrak f(-|x|,-|w|)=\pi^{P(-|x|)P(-|w|)}\mathfrak f(|x|,|w|)$. Putting these observations together,
\[l=\pi^{p(w)p(x)+p(\nu)p(x)+P(-|x|)P(-|w|)}\mathfrak f(|x|,|w|) (b^-x,z)((b^*)^+w,y)=\pi^{P(-|x|)P(-|w|)}r\]
Since parity in $X$ only depends on the $X/\Z[I]$ cosets and we have $-|x|\in \mu+\Z[I]$ and $-|w|\in\lambda+\Z[I]$, the result follows. 
\end{proof}
Lastly, note that Lemmas \ref{lem:crossthruturn} and \ref{lem:straightening} 
immediately imply the following corollary.
\begin{cor}\label{cor:180rotcross}We have an equality of diagrams

\[\pi^{P(\mu)P(\lambda)}
\hctikz{
\rcrossup{0}{0}{1}{1}
\draw (-.3,0) node {$\lambda$};
\draw (-.3,1) node {$\mu$};
}\quad = 
\hctikz{
\dcap{0}{.75}{1.5}{1}\dcap{.5}{.75}{.5}{.5}\idsup{2}{0}{.5}{1.5}{2}
\rcrossdown{1}{0}{.75}{.5}
\idsup{0}{-.75}{.5}{1.5}{2}\dcup{1.5}{-.5}{.5}{.5}\dcup{1}{-1}{1.5}{1}
\draw (-.3,-.8) node {$\lambda$};
\draw (.7,-.8) node {$\mu$};
}\quad = 
\hctikz{
\dcap{1}{.75}{1.5}{1}\dcap{1.5}{.75}{.5}{.5}\idsup{0}{0}{.5}{1.5}{2}
\rcrossdown{1}{0}{.75}{.5}
\idsup{2}{-.75}{.5}{1.5}{2}\dcup{.5}{-.5}{.5}{.5}\dcup{0}{-1}{1.5}{1}
\draw (1.7,-.8) node {$\lambda$};
\draw (2.7,-.8) node {$\mu$};
}\]
for any $\lambda,\mu\in X^+$.
\end{cor}

\subsection{Renormalization}
\label{subsec:renorm}

In the previous section, we deduced a number of identities between various
slice diagrams. These identities are almost the Turaev moves for 
(framed) oriented tangles, except for factors of $\pi$. Now we shall correct
these factors.

As noted in Remark \ref{rem: coefficients}, all of the previous statements about $\UU$ and it's modules hold verbatim
over the subring $\Qqp$ of $\Qqtt$. Now we will use the fact that $\pi=\tau^2$ to renormalize our maps.
These renormalized $\UU$-module homomorphisms will always be represented by a diagrammatic calculus
with red strands and labels to differentiate them.

\begin{center}
\begin{tabular}{ccccccc}

\red{\hctikz{
\idsup{-1}{1.5}{1}{1}{1}
\draw (-.8,1.5) node {$\lambda$};}}
$=$\quad
\hctikz{
\idsup{-1}{1.5}{1}{1}{1}
\draw (-.8,1.5) node {$\lambda$};}
&&
\red{\hctikz{
\idsdown{-1}{1.5}{1}{1}{1}
\draw (-.8,1.5) node {$\lambda$};}}
$=$\quad
\hctikz{
\idsdown{-1}{1.5}{1}{1}{1}
\draw (-.8,1.5) node {$\lambda$};}

\\\\
\red{\hctikz{
\cwcup{1}{-.5}{1}{.75}
\draw (1.8,-.5) node {$\lambda$};
}}\quad $=$\quad \hctikz{
\cwcup{1}{-.5}{1}{.75}
\draw (1.8,-.5) node {$\lambda$};
}
&&
\red{\hctikz{
\ccwcup{1}{-.5}{1}{.75}
\draw (1.7,-.5) node {$\lambda$};}}
\quad$=\tau^{3P(\lambda)}$\quad
\hctikz{
\ccwcup{1}{-.5}{1}{.75}
\draw (1.7,-.5) node {$\lambda$};}
\\\\
\red{\hctikz{
\cwcap{1}{-.5}{1}{.75}
\draw (1.8,.3) node {$\lambda$};
}}
\quad$=\tau^{P(\lambda)}$
\hctikz{
\cwcap{1}{-.5}{1}{.75}
\draw (1.8,.3) node {$\lambda$};
}
&&
\red{\hctikz{
\ccwcap{1}{-.5}{1}{.75}
\draw (1.8,.3) node {$\lambda$};
}}
\quad $=$\quad
\hctikz{
\ccwcap{1}{-.5}{1}{.75}
\draw (1.8,.3) node {$\lambda$};
}
\\\\
\red{\hctikz{
\rcrossup{0}{0}{1}{1}{1}
\draw (-.2,0) node {$\lambda$};
\draw (1.2,0) node {$\mu$};
}}
\quad$=\tau^{P(\lambda)P(\mu)}$
\hctikz{
\rcrossup{0}{0}{1}{1}{1}
\draw (-.2,0) node {$\lambda$};
\draw (1.2,0) node {$\mu$};
}
&&

\red{\hctikz{
\rcrossdown{0}{0}{1}{1}{1}
\draw (-.2,0) node {$\lambda$};
\draw (1.2,0) node {$\mu$};
}}
$=\tau^{3P(\lambda)P(\mu)}$
\hctikz{
\rcrossdown{0}{0}{1}{1}{1}
\draw (-.2,0) node {$\lambda$};
\draw (1.2,0) node {$\mu$};
}
\\\\
\red{\hctikz{
\lcrossup{0}{0}{1}{1}{1}
\draw (-.2,0) node {$\lambda$};
\draw (1.2,0) node {$\mu$};
}}
\quad$=\tau^{3P(\lambda)P(\mu)}$
\hctikz{
\lcrossup{0}{0}{1}{1}{1}
\draw (-.2,0) node {$\lambda$};
\draw (1.2,0) node {$\mu$};
}
&&

\red{\hctikz{
\lcrossdown{0}{0}{1}{1}{1}
\draw (-.2,0) node {$\lambda$};
\draw (1.2,0) node {$\mu$};
}}
$=\tau^{P(\lambda)P(\mu)}$
\hctikz{
\lcrossdown{0}{0}{1}{1}{1}
\draw (-.2,0) node {$\lambda$};
\draw (1.2,0) node {$\mu$};
}
\end{tabular}
\end{center}
\begin{rmk} \label{rmk:red diagrams}
We make two remarks about the red diagrammatic calculus.
\begin{enumerate}
\item We observe that whenever $P(\lambda)=0$, the maps represented by the red and black
diagrams are the same. By Lemma \ref{lem:weight parity equiv}, this holds whenever
$\lambda$ is an even weight ($\ang{{\rf n}, \lambda}\in 2\N$) or $n$ is even,
thus in these cases we can work over $\Qqp$.
\item Note that we don't define sideways-oriented crossings in the red strands.
This can be done using these renormalizations and 
Lemma \ref{lem:rotating crossings}, but we shall not need these diagrams
here.
\end{enumerate}
\end{rmk}
Recall that the {\em writhe} $\wr(T)$ of an oriented tangle $T$ is defined by forgetting the orientation and
setting 
\[\wr\parens{\hctikz{\rcross{0}{0}{.5}{.35}}}=1,\qquad\wr\parens{\hctikz{\lcross{0}{0}{.5}{.35}}}=-1,
\qquad \wr(T)=\sum \wr(X),\]
where the sum is over all crossings $X$ in $T$.

\begin{thm}\label{thm:knot invariant}
Let $T$ be an oriented tangle, and $\lambda\in X^+$ be a dominant weight.
For any slice diagram $S(T)$ of $T$, let $S(T)_\lambda$ be the associated map 
defined by the red diagrammatic calculus with strands colored by $\lambda$. 
Then $S(T)_\lambda$ is is independent of the choice of slice diagram, and 
$T_\lambda=S(T)_\lambda$ is an isotopy invariant of oriented framed tangles. Moreover, if 
$J_T^\lambda=(\pi^{p(\lambda)}\mathfrak f(\lambda,\lambda)^{-1}q^{\ang{\tilde \rho,\lambda}})^{{\rm wr}(T)}T_\lambda$,
then $J_T^\lambda$ is independent of the framing, hence is an invariant of $T$.
\end{thm}

\begin{proof}
To prove the theorem, it suffices to show that the maps $S(T)_\lambda$ (resp. $J_T^\lambda$)
are invariant under the Turaev moves (cf. \cite[Theorem 3.2]{T}, \cite[Theorem 3.3, Equations (3.9)-(3.16)]{Oht})
for framed (resp. unframed) oriented tangles.
First, observe that the identities \eqref{eq:comm and ids} and
Lemma \ref{lem:straightening} hold for red strands as well.
We also see that \eqref{eq:diagbraidrels} holds for red strands 
which all have the same orientation. (In fact, if we define
sideways-oriented crossings of red strands 
as described in Remark \ref{rmk:red diagrams} (2), then \eqref{eq:diagbraidrels}
would hold for red strands with any orientation.)

Furthermore, applying the normalizations and rearranging the $\tau$ factors in
Lemma \ref{lem:reide2} shows that, for either orientation, we have
\[\red{\hctikz{
\ids{0}{.75}{.5}{.75}{1}\dcap{.5}{.75}{.5}{.5}
\rcross{0}{0}{.75}{.5}\ids{1}{0}{.5}{.75}{1}
\ids{0}{-.75}{.5}{.75}{1}\dcup{.5}{-.5}{.5}{.5}
\draw (-.3,0) node {$\lambda$};
}}\quad =\pi^{P(\lambda)}\mathfrak f(\lambda,\lambda)q^{-\ang{\tilde\rho,\lambda}}  \red{\hctikz{\ids{0}{0}{1}{1.5}{1};
\draw (-.3,0) node {$\lambda$};}}
\quad=\quad\red{\hctikz{
\ids{1}{.75}{.5}{.75}{1}\dcap{0}{.75}{.5}{.5}
\rcross{.5}{0}{.75}{.5}\ids{0}{0}{.5}{.75}{1}
\ids{1}{-.75}{.5}{.75}{1}\dcup{0}{-.5}{.5}{.5}
\draw (-.3,0) node {$\lambda$};
}}.
\]
\[\red{\hctikz{
\ids{0}{.75}{.5}{.75}{1}\dcap{.5}{.75}{.5}{.5}
\lcross{0}{0}{.75}{.5}\ids{1}{0}{.5}{.75}{1}
\ids{0}{-.75}{.5}{.75}{1}\dcup{.5}{-.5}{.5}{.5}
\draw (-.3,0) node {$\lambda$};
}}\quad =\pi^{P(\lambda)}\mathfrak f(\lambda,\lambda)^{-1}q^{\ang{\tilde\rho,\lambda}}  \red{\hctikz{\ids{0}{0}{1}{1.5}{1};
\draw (-.3,0) node {$\lambda$};}}
\quad=\quad\red{\hctikz{
\ids{1}{.75}{.5}{.75}{1}\dcap{0}{.75}{.5}{.5}
\lcross{.5}{0}{.75}{.5}\ids{0}{0}{.5}{.75}{1}
\ids{1}{-.75}{.5}{.75}{1}\dcup{0}{-.5}{.5}{.5}
\draw (-.3,0) node {$\lambda$};
}}.
\]

Similarly, we see that Corollaries \ref{cor:rotating crossings inverse} and \ref{cor:180rotcross} gives us the identities
\[
\red{\hctikz{
\dcap{0}{.75}{1.5}{1}\dcap{.5}{.75}{.5}{.5}\idsup{2}{0}{.5}{1.5}{2}
\rcrossdown{1}{0}{.75}{.5}
\idsup{0}{-.75}{.5}{1.5}{2}\dcup{1.5}{-.5}{.5}{.5}\dcup{1}{-1}{1.5}{1}
}}\quad =\ 
\red{\hctikz{
\rcrossup{0}{0}{1}{1}
}}\quad = \ 
\red{\hctikz{
\dcap{1}{.75}{1.5}{1}\dcap{1.5}{.75}{.5}{.5}\idsup{0}{0}{.5}{1.5}{2}
\rcrossdown{1}{0}{.75}{.5}
\idsup{2}{-.75}{.5}{1.5}{2}\dcup{.5}{-.5}{.5}{.5}\dcup{0}{-1}{1.5}{1}
}}\]

\[\red{\hctikz{
\idsdown{1}{2.25}{.5}{.75}{1}\idsup{1.5}{1.5}{.5}{1.5}{1}\dcap{0}{2.25}{.5}{.5}
\lcrossdown{.5}{1.5}{.75}{.5}
\dcup{1}{1}{.5}{.5}
\idsup{0}{0}{.5}{2.25}{1}\ids{.5}{.75}{.5}{.75}{1}\dcap{1}{.75}{.5}{.25}
\rcrossdown{.5}{0}{.75}{.5}\idsup{1.5}{-.75}{.5}{1.5}{1}\ids{1}{-.75}{.5}{.75}{1}
\dcup{0}{-.5}{.5}{.5}
}}\quad =\quad\red{\hctikz{
\idsdown{0}{0}{.5}{2}{1}
\idsup{.5}{0}{.5}{2}{1}
}}\quad=\quad\red{\hctikz{
\idsdown{.5}{2.25}{.5}{.75}{1}\idsup{0}{1.5}{.5}{1.5}{1}\dcap{1}{2.25}{.5}{.5}
\lcrossdown{.5}{1.5}{.75}{.5}
\dcup{0}{1}{.5}{.5}
\idsup{1.5}{0}{.5}{2.25}{1}\ids{1}{.75}{.5}{.75}{1}\dcap{0}{.75}{.5}{.25}
\rcrossdown{.5}{0}{.75}{.5}\idsup{0}{-.75}{.5}{1.5}{1}\ids{.5}{-.75}{.5}{.75}{1}
\dcup{1}{-.5}{.5}{.5}
}}
\]
for any choice of labeling of the strands. In particular, we see
that the Turaev moves for oriented framed tangles are satisfied,
which proves that $T_\lambda$ is indeed an isotopy invariant of oriented
framed tangles. Moreover, note that $J_T^\lambda$
then satisfies the Turaev moves for oriented {\em unframed} tangles,
since the only Turaev move that changes the writhe is Reidemeister 2 (which
is to say the move straightening the crossings in Lemma \ref{lem:reide2}).
\end{proof}

We note that the proof of Theorem \ref{thm:knot invariant} actually implies a more
general result, though we first need to recall some notions.
The category of $X^+$-colored oriented tangles is the strict monoidal category
whose objects are finite sequences of pairs $(\lambda,s)$ where $\lambda\in X^+$
and $s\in \set{\pm 1}$, and whose morphisms from $(\lambda_a,s_a)_{1\leq a\leq b}$ to 
$(\mu_c,s_c)_{1\leq c\leq d}$  are tangle diagrams where
the labeling and orientation of the $r^{\rm th} $ strand from the left 
at the lower (respectively, upper) boundary corresponds to $(\lambda_r, s_r)$
(respectively, $(\mu_c, s_c)$); c.f. \cite{T,ADO} 
for more details. In particular, morphisms in this category 
(and thus colored tangles) are generated
from the elementary morphisms 
\newcommand{\undercurvearrowleft}{\raisebox{.5em}{\rotatebox{180}{$\curvearrowright$}}}
\newcommand{\undercurvearrowright}{\raisebox{.5em}{\rotatebox{180}{$\curvearrowleft$}}}
 \[\curvearrowright_\lambda,\quad \undercurvearrowleft_{\lambda},\quad \undercurvearrowright_{\lambda},\quad \undercurvearrowleft_{\lambda},\quad
(\searrow\hspace{-1em}\swarrow)^{\pm}_{\lambda,\mu},\quad
(\nearrow\hspace{-1em}\nwarrow)^{\pm}_{\lambda,\mu}.\]
subject to relations which are simply colored versions
of the Turaev moves.

We can extend Theorem \ref{thm:knot invariant} to framed multicolored tangles
with the same proof. To obtain the unframed invariant, the normalization constant
is replaced by
$\prod_{\lambda\in X^+}(\pi^{p(\lambda)}\mathfrak f(\lambda,\lambda)^{-1}q^{\ang{\tilde \rho,\lambda}})^{\wr_\lambda(T)}$, where $\wr_\lambda$ is defined to be the writhe where
we exclude from the sum any crossings where there is a strand not labeled by $\lambda$.
Therefore, we obtain the following corollary.
\begin{cor}
There exists
a covariant functor $J$ from the category of $X^+$-colored oriented tangles modulo
isotopy to $\catO_{\rm fin}$ which sends the object $((\lambda_1,s_1),\ldots, (\lambda_r,s_r))$ to the module 
$V(s_1\lambda_1)\otimes\ldots\otimes V(s_r\lambda_r)$ and is given
on morphisms by
\[\curvearrowright_\lambda\mapsto\tau^{P(\lambda)} \ev_\lambda,\qquad
\curvearrowleft_\lambda\mapsto\qtr_\lambda,\qquad
\undercurvearrowright_{\lambda}\mapsto\coqtr_\lambda,\qquad
\undercurvearrowleft_{\lambda}\mapsto\tau^{-P(\lambda)}\coev_\lambda,
\]
\[(\nearrow\hspace{-1em}\nwarrow)^{\pm}_{\lambda,\mu}\mapsto 
(\pi^{p(\lambda)}\mathfrak f(\lambda,\lambda)^{-1}q^{\ang{\tilde \rho,\lambda}})^{\pm \delta_{\lambda,\mu}}
\tau^{\pm P(\lambda)P(\mu)}\cR_{\lambda,\mu}^{\pm1},\]
\[(\searrow\hspace{-1em}\swarrow)^{\pm}_{\lambda,\mu}\mapsto 
(\pi^{p(\lambda)}\mathfrak f(\lambda,\lambda)^{-1}q^{\ang{\tilde \rho,\lambda}})^{\pm \delta_{\lambda,\mu}}
\tau^{\mp P(\lambda)P(\mu)}\cR_{-\lambda,-\mu}^{\pm1}.\]
In particular, if $L$ is an oriented colored link, then
$J(L)\in\Qqtt$ is the associated quantum covering $\osp(1|2n)$
colored link invariant.
\end{cor}

\begin{example}\label{ex:rk1knot}
Let's take $n=1$ and $\lambda=1$. 
Fix $\mathfrak f(1,1)=1$, and note that $\ang{\tilde \rho,\lambda}=1$ and $p(\lambda)=1$.
We can explicitly compute the maps represented by our diagrams on $V(1)\otimes V(1)$.
Let $v_1,v_{-1}$ be the basis of $V(1)$ from Example \ref{ex:rk1mods}. Then with respect to the ordered basis $\set{v_1\otimes v_1, v_1\otimes v_{-1}, v_{-1}\otimes v_1, 
v_{-1}\otimes v_{-1}}$ of $V(1)\otimes V(1)$, we have 
\[\Theta=
\begin{bmatrix}1 & 0& 0& 0\\ 0 & 1& 0& 0\\ 0 &  q^{-1}-\pi q & 1& 0\\
0 & 0& 0& 1\end{bmatrix},\qquad \mathfrak{F}=\begin{bmatrix} 1 & 0& 0& 0\\ 0 & q& 0& 0\\
0 & 0 & \pi q & 0\\ 0 & 0& 0& \pi
\end{bmatrix}, \qquad \mathbf s=\begin{bmatrix} 1 & 0& 0& 0\\ 0 & 0& 1& 0\\
0 & 1 & 0 & 0\\ 0 & 0& 0& \pi
\end{bmatrix}\]
and thus
\[
\red{\hctikz{\rcrossup{0}{0}{1}{1}}}= 
\begin{bmatrix}
\tau & 0& 0& 0\\
0 & 0& \tau q& 0\\ 
0 & \tau^3 q & \tau -  \tau^3 q^2 &  0\\
0 & 0& 0& \tau
\end{bmatrix},
\quad 
\red{\hctikz{\lcrossup{0}{0}{1}{1}}}=
\begin{bmatrix}
\tau^3& 0& 0& 0\\
0 & \tau^3- \tau q^{-2}& \tau q^{-1}& 0\\
 0 & \tau^3 q^{-1} & 0 &  0\\
0 & 0& 0& \tau^3
\end{bmatrix}
\]
Note that $\pi^{p(\lambda)}q^{\ang{\tilde \rho,\lambda}}=\pi q$. 
Then it is easy to verify directly that
\[\red{\hctikz{\rcrossup{0}{0}{1}{1}}}-q^2
\ \red{\hctikz{\lcrossup{0}{0}{1}{1}}}\ \ =(\tau-\tau^3 q^2)\ \ \red{\hctikz{\idsup{0}{0}{1}{1}{1}\idsup{.5}{0}{1}{1}{1}}}\]
Now let $T$ be an unframed oriented link with all strands colored by $\lambda$, and fix a subdiagram which consists of two strands with 
either no crossing or a single crossing. Since $T^\sharp$ is isotopy
invariant and independent of framing, we may assume that the strands are directed upward.
Let $T_+$ (resp. $T_0$, $T_-$) be $T$ with the subdiagram replaced by $\red{\hctikz{\rcrossup{0}{0}{1}{1}}}$ 
(resp. \red{\hctikz{\idsup{0}{0}{1}{1}{1}\idsup{.5}{0}{1}{1}{1}}}\ ,\ \red{\hctikz{\lcrossup{0}{0}{1}{1}}}).
Then using the above relation and the definition in Theorem
\ref{thm:knot invariant},
\[(\pi q^{-1}) J_{T_+}^1-(\pi q^3)J_{T_-}^1=(\tau-\tau^3 q^2) J_{T_0}^1\]
hence
\[(\pi q^2)^{-1}J_{T_+}^1-\pi q^2J_{T_-}^1=(\tau q^{-1}-\tau^3 q)J_{T_0}^1.\]
Moreover, if $T$ is the unknot, then for either orientation we have $J_{T}^1=\tau^3 q+\tau  q^{-1}$.
In particular, we see that for any link $K$, $J_{K}^1$ 
is simply a multiple of the Jones polynomial of $T$
in the variable $\tau^3 q=\tau^{-1}q$. In particular, note that using the specialization $\tau=\bt$,
which corresponds to $\pi=-1$, this shows the uncolored $U_q(\osp(1|2))$ link invariant is equal to the $U_{\bt^{-1}q}(\mathfrak{sl}_2)$ link invariant.
\end{example}

\section{Relating $\mathfrak{so}(2n+1)$ and $\osp(1|2n)$ invariants}

The results of Example \ref{ex:rk1knot} suggest a connection
between the specializations of the tangle invariants 
in Theorem \ref{thm:knot invariant}. 
We now make this
precise by extending the constructions in \cite{CFLW,C}.
We begin by recalling the definition of the twistor maps.

\subsection{Definition of Twistors}
An {\em enhancer} $\phi$ is an function $\phi:\Z[I]\times X\rightarrow \Z$ satisfying
\begin{equation}
\begin{array}{c}
\phi(\nu,\lambda+\mu)\equiv \phi(\nu,\mu)+\phi(\nu,\lambda)\mod 4
\text{ for }\nu,\mu\in \Z[I]\\
\phi(\nu+\mu,\lambda)\equiv \phi(\nu,\lambda)+\phi(\mu,\lambda) \mod 4
\text{ for }\nu,\mu\in \Z[I]\\
\phi(i,i)=d_i\text{ and }\phi(i,j)\in 2\Z\text{ for }i\neq j\in I.\\
\phi(i,j)-\phi(j,i)\equiv i\cdot j+2p(i)p(j)\mod 4 \text{ for }i,j\in I
\end{array}
\end{equation}
Note that $\phi(i,i)-\phi(i,i)=0\equiv i\cdot i+2p(i)p(i)$ modulo 4
since $i\cdot i=2d_i$ and $2p(i)p(i)=2p(i)=2d_i$.
In particular, note that these congruences imply that
\begin{equation}\label{eq:phi on ZI}
\begin{array}{c}
\phi_4:\Z[I]\times\Z[I]\rightarrow \Z/4\Z \text{ defined by }\phi_4(\mu,\nu)=\phi(\mu,\nu)\!\!\mod 4\text{ is a }\Z\text{-bilinear map}\\
\text{and }\phi(\mu,\nu)\equiv\phi(\nu,\mu)+\mu\cdot \nu+2p(\mu)p(\nu)\mod 4\text{  for }\mu,\nu\in\Z[I].
\end{array}
\end{equation}
Note that an enhancer can always be defined on $\Z[I]\times\Z[I]$
by defining it for $I$ and extending in $\Z$-bilinearly, 
and then it can be extended 
to $\Z[I]\times X$ by translation along a transversal of $X/\Z[I]$. 

When $I$ has a unique odd element, as in the present case, 
the enhancer is closely
related to the usual pairing.
\begin{lem}\label{lem:enhancer single odd}
Let $\phi$ be an enhancer.
Then $\phi(\mu,\nu)+\phi(\nu,\mu)\equiv \mu\cdot \nu$ modulo 4.
\end{lem}
\begin{proof}
First set $(,)_\phi, (,)_\bullet:\Z[I]\times\Z[I]\rightarrow \Z/4\Z$ by
$(\mu,\nu)_\phi=\phi_4(\mu,\nu)+\phi_4(\nu,\mu)$ and 
$(\mu,\nu)_\bullet=\mu\cdot\nu\mod 4$.
Both maps are $\Z$-bilinear, so it suffices to show they
take the same values on $I\times I$.
Well, if $i\neq j$, then at least one of $i$ or $j$
is even and thus $2p(i)p(j)=0$ since $|I_1|=1$.  On the other hand,
$\phi(i,j)\in 2\Z$ so 
$\phi(i,j)+\phi(j,i)\equiv_4 \phi(i,j)-\phi(j,i)
\equiv_4 i\cdot j +2p(i)p(j)=i\cdot j$. 
Finally, note that
$\phi(i,i)+\phi(i,i)=2d_i=i\cdot i$ for any $i\in I$. 
\end{proof}

The {\em $\phi$-enhanced quantum covering group}
$\hatU$ associated to $\UU$ and the enhancer $\phi$ is the semidirect product of $\UU$ with the algebra $\Qqtt[T_\mu, \Upsilon_\mu\mid \mu\in \Z[I]]$ subject to the relations
\begin{equation}\label{eq:TUpsrels}
T_\mu T_\nu=T_{\mu+\nu},\quad \Upsilon_\mu \Upsilon_\nu=\Upsilon_{\mu+\nu},\quad T_0=\Upsilon_0=T_\nu^4=\Upsilon_\nu^4=1,\quad
T_\mu \Upsilon_\nu=\Upsilon_\nu T_\mu,
\end{equation} 
\begin{equation}\label{eq:Tweightrels}
T_\mu u=\bt^{\ang{\mu,|u|}} u T_\mu,\quad u\in\UU,\ \mu\in\Z[I]
\end{equation} 
\begin{equation}\label{eq:Upsweightrels}
\Upsilon_\mu u=\bt^{\phi(\mu,|u|)} u \Upsilon_\mu,\quad u\in\UU,\ \mu\in\Z[I]
\end{equation} 
See \cite{CFLW,C} for a more formal definition.
The enhanced quantum covering group has a useful $\Q(\bt)$-linear automorphism called a {\em twistor}.
There are several ways to define such a twistor; we will need the following.

\begin{prop}{\bf \cite[Theorems 4.3, 4.12]{CFLW}} \label{prop:twistordef}
Define a product $*$ on $\ff$ by the following rule: if $x$ and $y$ are homogeneous elements of $\ff$,
let $x*y=\bt^{\phi(|x|,|y|)}xy$. Let $(\ff,*)$ denote $\ff$ with this multiplication.
\begin{enumerate}
\item Then there is a $\Q(\bt)$-linear algebra isomorphism $\Tw:\ff\rightarrow (\ff,*)$
defined by
\[\Tw(\theta_i)=\theta_i,\quad \Tw(q)=\bt^{-1}q,\quad \Tw(\tau)=\bt^{-1}\tau.\]
\item Let $\cB$ be the canonical basis of $\ff$ (cf. \cite{CHW2}). 
Then $\Tw$ on $\ff$ satisfies $\Tw(b)=\bt^{\ell(b)} b$ for all 
$b\in \cB$, where $\ell(b)$ is some integer depending on $b$.
\item There is a  $\Q(\bt)$-algebra automorphism $\Tw:\hatU\rightarrow \hatU$ defined by
\[\Tw(E_i)=\bt_i^{-1} \tT_i \Upsilon_iE_i,\quad \Tw(F_i)=F_i\Upsilon_{-i},\quad \Tw(K_\mu)=T_{-\mu}K_\mu,\quad \Tw(J_\mu)=T_{2\mu} J_\mu,\]
\[\Tw(T_\mu)=T_\mu,\quad \Tw(\Upsilon_\mu)=\Upsilon_\mu,\quad 
\Tw(q)=\bt^{-1} q,\quad \Tw(\tau)=\bt^{-1}\tau,\]
where if $\mu=\sum_{i\in I} \mu_i i$, $\tT_\mu=\prod_{i\in I} T_{\mu_i d_i i}$.
\item For $x\in \ff[\bt]$, we have
\subitem {\rm (a)} $\Tw(x^+)=t_\nu^{2}\bt^{\bullet(|x|)} \Tw(x)^+\tT_{|x|}\Upsilon_{|x|}$
\subitem {\rm (b)} $\Tw(x^-)=\Tw(x)^-\Upsilon_{-|x|}$
\end{enumerate}
\end{prop}
Later on, we will need some alternate versions of the results in 
Proposition \ref{prop:twistordef} 
which we shall prove now. First, we note the following analogue of
Proposition \ref{prop:twistordef} (2) for the dual canonical basis.

\begin{lem}\label{lem:twondualbasis}
Let $(-,-)$ be the bilinear form on $\ff$ defined in 
\eqref{eq:ff bilinear form}.
Then \[\Tw^{-1}((\Tw(x),\Tw(y)))=(-1)^{\bp(|x|)}(x,y).\]
In particular, $\Tw(b^*)=(-1)^{\bp(b)}\bt^{-\ell(b)} b^*$
for any $b\in \cB$.
\end{lem}
\begin{proof}
Let $(x,y)^\Tw=\Tw^{-1}((\Tw(x),\Tw(y)))$ and observe that this is
a $\Qqtt$-bilinear form on $\ff[\bt]$. Moreover, note that 
$\ir(\Tw(w))=\bt^{\phi(i,|w|-i)}\Tw(\ir(w))$ by the same proof as
\cite[Proposition 3.2(b)]{CFLW}.
To show $(x,y)^\Tw=(-1)^{\bp(|x|)}(x,y)$, 
we proceed by induction on the height. First note that
\[(1,1)^{\Tw}=1=(1,1),\qquad (\theta_i,\theta_i)^{\Tw}=\Tw^{-1}\parens{\frac{1}{1-\pi_iq_i^{-2}}}=(\theta_i,\theta_i).\]
Now if $x\in\ff[\bt]_{\nu-i}$ and $y\in \ff[\bt]_\nu$ for some $i\in I$ and
$\nu\in \Z_{\geq_0}[I]$ with $\height(\nu)>1$, we have
\begin{align*}
(\theta_ix,y)^\Tw&=\bt^{\phi(i,\nu-i)}\Tw^{-1}((\theta_i\Tw(x),\Tw(y)))=
\bt^{\phi(i,\nu-i)}(\theta_i,\theta_i)\Tw^{-1}((\Tw(x),\ir(\Tw(y))))\\
&=\bt^{2\phi(i,\nu-i)}(\theta_i,\theta_i)\Tw^{-1}((\Tw(x),\Tw(\ir(y))))\\
&=(-1)^{\bp(\nu-i)+\phi(i,\nu-i)}(\theta_i,\theta_i)(x,\ir(y))\\
&=(-1)^{\bp(\nu-i)+p(\nu-i)p(i)}(\theta_ix,y)=(-1)^{\bp(\theta_ix)}(\theta_ix,y)
\end{align*}
where in the last equality, note that if $\nu=\sum_{i\in I} \nu_i i$ then
we have $\phi(i,\nu-i)\equiv (\nu_i-1)d_i$ modulo 2.
The proof is finished by observing that $(\nu_i-1)d_i\equiv p(\nu-i)p(i)$
for any $i\in I$, since if $i\neq \rf n$ both sides are $0$ modulo
2, and if $i=\rf n$ both sides are equivalent to $\nu_{\rf n}-1$ modulo 2.
\end{proof}

\begin{rem}
Though Lemma \ref{lem:twondualbasis} as stated requires $|I_1|=1$,
a version of it also holds for arbitrary
enhanced quantum covering algebras. Indeed, if $|I_1|>1$, then 
$\Tw^{-1}((\Tw(x),\Tw(y))=\bt^{\binom{\nu}{2}}(x,y)$,
where $|x|=\nu=\sum_{i\in I}\nu_i i$ and
$\binom{\nu}{2}=\sum_{i\in I}\binom{\nu_i}{2}d_i$.
\end{rem}
It will also be more convenient to have the following 
variant of Proposition \ref{prop:twistordef} (4b).
\begin{lem}\label{lem:tw+alt}
We have \[\Tw(x^+)=\bt_{|x|}^{-1}\tT_{|x|}\Upsilon_{|x|}\Tw(x)^+.\]
\end{lem}
\begin{proof}
This is true if $x=\theta_i$. It suffices to show if it is true for $x$, then it is true
for $\theta_ix$.
\begin{align*}
\Tw(\theta_ix^+)&=\Tw(\theta_i^+)\Tw(x^+)=\bt_i^{-1}\tT_i\Upsilon_i E_i \bt_\nu^{-1}T_{\nu}\Upsilon_\nu \Tw(x)^+\\
&=\bt_{i+\nu}^{-1}\bt^{-i\cdot\nu-\phi(\nu,i)}\tT_i\Upsilon_i  \tT_{\nu}\Upsilon_\nu E_i\Tw(x)^+=\bt_{|\theta_ix|}^{-1}\bt^{-i\cdot\nu-\phi(\nu,i)-\phi(i,\nu)}\tT_{i+\nu}\Upsilon_{i+\nu} \Tw(\theta_i x)^+
\end{align*}
But then by
Lemma \ref{lem:enhancer single odd},
\[-i\cdot\nu-\phi(\nu,i)-\phi(i,\nu)\equiv_4 -2i\cdot\nu\equiv_4 0.\]
\end{proof}

\subsection{$\hatU$-modules and Hopf structure}

Let $M$ be a $\UU$-weight module. 
Then $M$ is canonically a $\hatU$-module
by defining 
\begin{align}
T_\mu m&= \bt^{\ang{\mu,\lambda}} m,\quad \mu\in\Z[I], m\in M_\lambda;\label{eq:T act}\\
\Upsilon_\mu m&=\bt^{\phi(\mu,\lambda)} m,\quad \mu\in\Z[I], m\in M_\lambda.\label{eq:Ups act}
\end{align}
To that end, we will call any $\hatU$-module which restricts to a 
$\UU$-weight module and satisfies \eqref{eq:T act} a 
{$\hatU$-weight module}. If it additionally satisfies \eqref{eq:Ups act},
we shall call it a {\em canonical} $\hatU$-weight module.

In particular, any tensor product of $\hatU$-modules can be given
a canonical $\hatU$-weight module structure.
However, such a procedure forgets the action
of the $\Upsilon$ elements on the factors due to the lack
of additivity in the second component of $\phi$. 

\begin{example}
Consider the case $n=1$. Then $\hatU$ has the canonical
weight module $\hV(1)=\Qqtt v_1\oplus \Qqtt v_{-1}$
which is isomorphic to $V(1)$ as a $\UU$-module and
satisfies $T_i v_1= \bt v_1$ and
$\Upsilon_{\rf 1} v_1=\bt^{\phi({\rf 1},1)} v_1$.
Then $\hV(1)\otimes \hV(1)$ is a $\UU$-weight module hence
has a canonical $\hatU$-module structure, but note that 
\[\Upsilon_{\rf 1} v_1\otimes v_1=\bt^{\phi({\rf 1},2)} v_1\otimes v_1\]
and by the definition of $\phi$, we have
 $\phi({\rf 1},2)=\phi({\rf 1},{\rf 1})=1\neq \phi({\rf 1},1)+\phi({\rf 1},1)$. 
\end{example}

In particular, canonical module structures will be too naive for 
our purposes. Instead, we will introduce Hopf structure which will inform
our classes of weight modules.

\begin{prop}
The algebra $\hatU$ has a Hopf covering algebra structure given by the following:
\begin{enumerate}
\item A coassociative coproduct $\Delta:\hatU\rightarrow \hatU\otimes_{\Qqtt} \hatU$
extending $\Delta:\UU\rightarrow\UU\otimes_\Qqtt \UU$ such that
$\Delta(T_\mu)=T_\mu\otimes T_\mu$ and
$\Delta(\Upsilon_\mu)=\Upsilon_\mu\otimes \Upsilon_\mu$
 for $\mu\in\Z[I]$. In particular, we inductively define 
 $\Delta^t=(\Delta\otimes 1^{t-1})\circ \Delta^{t-1}:\hatU\rightarrow \hatU^{\otimes (t+1)}$ 
 for any integer $t>1$.
\item An antipode $S:\hatU\rightarrow\hatU$ extending $S:\UU\rightarrow \UU$
such that $S(T_\mu)=T_{-\mu}$ and $S(\Upsilon_\mu)=\Upsilon_{-\mu}$  for $\mu\in\Z[I]$.
\item A counit map $\epsilon:\hatU\rightarrow \hatU$ extending 
$\epsilon:\UU\rightarrow\UU$ such that $\epsilon(T_\mu)=\epsilon(\Upsilon_\mu)=1$ for $\mu\in\Z[I]$.
\end{enumerate}
\end{prop}

\begin{proof}
To show that these maps define a Hopf structure, 
we need only check that these morphisms respect
\eqref{eq:TUpsrels}-\eqref{eq:Upsweightrels}.
This is obvious for \eqref{eq:TUpsrels}, and can be quickly verified
for \eqref{eq:Tweightrels} and \eqref{eq:Upsweightrels} by checking
it for the generators of $\UU$.
For instance,
\begin{align*}
\Delta(\Upsilon_\mu)\Delta(E_i)
&=(\Upsilon_\mu\otimes\Upsilon_\mu)(E_i\otimes 1+\tJ_i\tK_i\otimes E_i)\\
&=\Upsilon_\mu E_i\otimes \Upsilon_\mu+\Upsilon_\mu \tJ_i\tK_i\otimes \Upsilon_\mu E_i\\
&= \bt^{\phi(\mu,i)}E_i\Upsilon_\mu\otimes \Upsilon_\mu+ \tJ_i\tK_i\Upsilon_\mu\otimes \bt^{\phi(\mu,i)} E_i \Upsilon_\mu\\
&=\bt^{\phi(\mu,i)}\Delta(E_i)\Delta(\Upsilon_\mu);
\end{align*}
\[S(E_i)S(\Upsilon_\mu)=-\tJ_{-i}\tK_{-i}E_i\Upsilon_{-\mu}
=-\bt^{\phi(\mu,i)}\Upsilon_{-\mu}\tJ_{-i}\tK_{-i}E_i
=\bt^{\phi(\mu,i)}S(\Upsilon_\mu)S(E_i);\]
\[\epsilon(E_i)\epsilon(\Upsilon_\mu)=(0)(1)=\bt^{\phi(\mu,i)}(1)(0)
=\bt^{\phi(\mu,i)}\epsilon(\Upsilon_\mu)\epsilon(E_i).\]
Finally, the co-associativity of $\Delta$ on $\hatU$ 
follows immediately from the
co-associativity of $\Delta$ on $\UU$ and the fact that $T_\mu$ and $\Upsilon_\mu$ are grouplike elements.
\end{proof}

The coproduct gives us another way to define an action of $\hatU$ on tensor
products of canonical $\hatU$-weight modules.
Henceforth, given $\hatU$-weight modules $M$ and $N$,
we let $M\hotimes N$ denote the space
$M\otimes_{\Qqtp} N$ with the $\hatU$-weight module structure induced
by the coproduct on $\hatU$. (Note that in general, 
the module $M\hotimes N$ is {\em not} canonical!)
\begin{example}
Continuing the previous example,
the action of $\Upsilon_{\rf 1}$ on
$\hV(1)\hotimes \hV(1)$ 
is given by 
\[\Delta(\Upsilon_{\rf 1})
v_1\otimes v_1=\bt^{2\phi({\rf 1},1)} v_1\otimes v_1.\]
\end{example}

Another natural module to consider is the following. 
Given a canonical $\hatU$-weight module $M$,
we can construct the restricted linear dual $M^*$.
This space is naturally a $\UU$-weight module as in \S\ref{subsec:modules}, hence has a canonical $\hatU$ structure.
On the other hand, let $M^\natural$ denote the space
$M^*$ with the action of  $\hatU$ defined by
$(uf)(x)=\pi^{p(f)p(u)}f(S(u)x)$.  Note that $M^\natural$ is not canonical:
if $f\in (M_\lambda,s)^*$, then $|f|=-\lambda$ but nevertheless 
\[\Upsilon_\mu f=\bt^{-\phi(\mu,\lambda)} f.\]
Since modules with these unorthodox actions of the $\Upsilon_\mu$ will be 
of primary importance, we give the following definitions.
\begin{dfn}
We say that a $\hatU$-weight module $M$ is {\em anti-canonical}
if $\Upsilon_\mu m=\bt^{-\phi(\mu,-\lambda)} m$ for all $m\in M_\lambda$.
More generally, we say that $M$ is a {\em mixed} weight module
if there exists an integer $t\geq 1$ and a 
sequence $c=(c_1,\ldots, c_t)\in \set{\pm 1}^t$ such that
$M_\lambda=\bigoplus_{(\lambda_s)\in (X^t)_\lambda} M_{(\lambda_s)}$, 
where \[(X^t)_\lambda=\set{(\lambda_1,\ldots,\lambda_t)\in X^t\mid \lambda=\lambda_1+\ldots+\lambda_s},\]
\[M_{(\lambda_s)}=\set{m\in M\mid\Upsilon_\mu m=\bt^{\sum_{1\leq s\leq t} c_s\phi(\mu,c_s\lambda_s)} m\text{ for all }\mu\in\Z[I]}.\] 
We say $c$ is the {\em signature} of $M$, and denote it by $\sig(M)=c$. 
\end{dfn}
\begin{rem}
We note that just as any weight $\UU$-module can be given
a canonical $\hatU$-module structure, it can also be
given an anti-canonical $\hatU$-module structure.
Indeed, suppose $M$ is a weight $\UU$-module
and define $T_im=\bt^{\ang{i,|m|}}m$ and $\Upsilon_i m=\bt^{-\phi(i,-|m|)}$.
Then this defines an action of $\hatU$, since for any $i\in I$ and $u\in\UU$,
$\Upsilon_i um=\bt^{-\phi(i,-(|m|+|u|))} um
=\bt^{\phi(i,|u|)}u\Upsilon_i m$.
\end{rem}
In addition to classifying modules by the action of the $\Upsilon$ elements,
another property of $\hatU$-weight modules which will be important to us
is their interaction with the twistor map $\Tw:\hatU\rightarrow\hatU$.

\begin{dfn}
Let $M$ be a $\hatU$-weight module. We say $M$ carries a twistor $\Tw$ (or
$\Tw$ is a twistor on $M$) if there exists a homogeneous
$\Q(t)$-linear bijection $\Tw:M\rightarrow M$ 
such that $\Tw(um)=\Tw(u)\Tw(m)$.
\end{dfn}

Modules which carry twistors are not hard to find. Indeed, the simple $\UU$-modules $V(\lambda)$
are themselves examples when given canonical (or anti-canonical)
actions of $\hatU$.
\begin{lem}{\bf \cite[Lemma 6.9]{C}}\label{lem:modtwistor}
Let $\lambda\in X^+$.
Let $\hV(\lambda)$  be the space $V(\lambda)$ with the canonical action of $\hatU$.
There is a $\Q(\bt)$-linear map $\Tw:\hV(\lambda)\rightarrow \hV(\lambda)$ 
which satisfies $\Tw(v_\lambda)=v_\lambda$ and $\Tw_\lambda(um)=\Tw(u)\Tw(m)$ for all $u\in \hatU$
and $m\in \hV(\lambda)$.
\end{lem}

In light of Lemma \ref{lem:dual isos}, it follows that the $\UU$-module $V(\lambda)^*$, 
viewed as a canonical $\hatU$-module, also carries a twistor. A similar argument to
\cite[Lemma 6.9]{C} can be used to construct a twistor on $V(\lambda)$ with an anti-canonical
action of $\UU$, hence the $\hatU$-module $\hV(\lambda)^\natural$ carries a twistor.
However, this construction is not very compatible with the dual basis, since it relies
on an isomorphism $V(\lambda)\rightarrow \Pi^{P(\lambda)}V(\lambda)$ and
is defined by descent from the highest weight vector.
To obtain a convenient definition of a twistor on the dual modules, we will define a map
directly on $\hV(\lambda)^\natural$.

Define the {\em dual twistor} on $\hatU$ to be the map 
$\Twnat(u)=S\circ\Tw\circ S^{-1}(u)$. This map is clearly a bijection,
and for any $u,v\in\hatU$ we have
\begin{align*}
\Twnat(uv)&=S(\Tw(S^{-1}(uv)))\\
&=\bt^{2p(u)p(v)}S(\Tw(S^{-1}(u)))S(\Tw(S^{-1}(v)))\\
&=\bt^{2p(u)p(v)}\Twnat(u)\Twnat(v).
\end{align*}
Therefore, it is determined by the images of the generators, which are
\[\Twnat(E_i)=t_iE_i\Upsilon_{-i},\quad \Twnat(F_i)=\Upsilon_iF_i\tT_i,
\quad \Twnat(K_\mu)=T_{-\mu}K_{\mu},\quad \Twnat(J_\mu)=T_{2\mu}J_\mu\]
\[\Twnat(q)=\bt^{-1}q,\quad \Twnat(\tau)=\bt\tau.\]
In particular, note that
\begin{equation}\label{eq:twnat-}
	\Twnat(x^-)=\Upsilon_{\nu}\Tw(x)^-\tT_\nu.
\end{equation}

While $\Twnat$ is not an algebra automorphism of $\hatU$, it shares many properties
with $\Tw$. In particular, we have a version of Lemma \ref{lem:modtwistor}.
\begin{lem}\label{lem:mod dual twistor}
	Let $\lambda\in X^+$.
	There is a $\Q(\bt)$-linear map $\Twnat:\hV(\lambda)\rightarrow \hV(\lambda)$ 
	which satisfies $\Twnat(v_\lambda)=v_\lambda$ and $\Twnat(um)=\bt^{2p(u)p(m)}\Twnat(u)\Twnat(m)$ for all $u\in \hatU$
	and $m\in \hV(\lambda)$.
\end{lem}
\begin{proof}
	This follows from more or less the same proof as \cite[Lemmas 6.8, 6.9]{C}. To wit, we can identify the Verma module
	of highest weight $\lambda$ for $\UU$ with $\ff$ (cf. {\em loc. cit} for details), and in particular this is naturally a
	canonical $\hatU$-module. Then we define a map 
	$\Tw^\natural_\lambda:\ff\rightarrow \ff$ via 
	$\Tw^\natural_\lambda(x)=\bt^{\ang{\tilde\nu,\lambda}+\phi(\nu,\lambda-\nu)}\Tw(x)$.
	Then it is straightforward to verify that $\Twnat(ux)=\bt^{2p(u)p(x)}\Twnat(u)\Twnat(x)$
	for $x\in \hV(\lambda)$ and $u=F_i, T_\mu, J_\mu, K_\mu, \Upsilon_\mu$.
	From the calculations in {\em loc. cit} and the definition, we see that 
	\[\Tw^\natural_\lambda(E_ix)=\bt^{\star} 
	E_i\Tw^\natural_\lambda(x).\]
	where $\star=\ang{\tilde\nu-\tilde i,\lambda}-\ang{\tilde\nu,\lambda}+\phi(\nu-i,\lambda-\nu+i)-\phi(\nu,\lambda-\nu)
	-d_i+\ang{\tilde{i},\lambda-\nu+i}-\phi(i,\nu-i)$.
	Now we can simplify $\star$ and apply  \eqref{eq:phi on ZI} to see that
	\[\star\equiv\phi(\nu,i)-\phi(i,\nu)-\phi(i,\lambda-\nu)+i\cdot\nu+d_i=2p(\nu)p(i)-\phi(i,\lambda-\nu)+d_i\mod 4,\]
	and thus
	\[\Tw^\natural_\lambda(E_ix)=\bt^{p(\nu)p(i)}\bt_i 
	E_i\Upsilon_{-i}\Tw^\natural_\lambda(x)=\bt^{p(\nu)p(i)}\Twnat(E_i)\Tw^\natural_\lambda(x).\]
	Finally, we note that the kernel of the projection $\ff\rightarrow \hV(\lambda)$ is trivially 
	preserved by $\Tw^\natural_\lambda$, hence it descends to a map on $\hV(\lambda)$.
\end{proof}

The dual twistor $\Twnat$ is what will allow us to define a convenient twistor map on dual modules, as follows.
Recall that $V(-\lambda)$ denotes the $\UU$-module $V(\lambda)^*$. We will adapt this notation to $\hV(\lambda)^\natural$.

\begin{lem}
	For $\lambda\in X^+$, let $\hV(-\lambda)=\hV(\lambda)^\natural$;
	that is, the space $V(\lambda)^*$ with the action of $\hatU$ induced by the antipode $S:\hatU\rightarrow\hatU$.
	Define a map $\Tw$ on $\hV(-\lambda)$ by
	$\Tw(f)(x)=\bt^{2p(f)p(x)}\Tw(f(\Twnat^{-1}(x)))$ for homogeneous $x\in \hV(\lambda)$ and $f\in \hV(-\lambda)$.
	Then $\Tw(uf)=\Tw(u)\Tw(f)$ for all $u\in\hatU$ and $f\in \hV(-\lambda)$.
\end{lem}
\begin{proof}
Let $f\in \hV(-\lambda)$ and $x\in \hV(\lambda)$ be homogeneous. 
First, observe that since $\Twnat$ preserves the $\bigrset$-grading, 
$\Tw(f)(x)=0$ unless $\bigrdeg{x}=\bigrdeg{f}$. Moreover, if $a\in\Qqtt$, 
\[\Tw(f)(ax)=\bt^{2p(f)p(x)}\Tw(f(\Twnat^{-1}(ax)))
=\bt^{2p(f)p(x)}\Tw(\Tw^{-1}(a)f(\Twnat^{-1}(x)))=a\Tw(f)(x),\] so $\Tw(f)$ is indeed
an element of $\hV(-\lambda)$.

Now suppose $u\in\hatU$. We compute that
\[\Tw(uf)(x)=\bt^{2p(uf)p(x)}\Tw((uf)(\Twnat^{-1}(x)))=\bt^{2p(u)p(x)+2p(f)p(x)}\Tw(\pi^{p(u)p(f)}f(S(u)\Twnat^{-1}(x)),\]
\begin{align*}
\Tw(u)\Tw(f)(x)&=\bt^{2p(f)p(ux)}\pi^{p(u)p(f)}\Tw(f(\Twnat^{-1}(S(\Tw(u))x)))\\
&=\bt^{2p(f)p(u)+2p(f)p(x)+2p(u)p(x)}\pi^{p(u)p(f)}\Tw(f(\Twnat^{-1}(S(\Tw(u)))\Tw'^{-1}(x)))\\
&=\bt^{2p(f)p(x)+2p(u)p(x)}\Tw(\pi^{p(u)p(f)}f(S(u)\Twnat^{-1}(x)).
\end{align*}
Therefore, $\Tw(uf)=\Tw(u)\Tw(f)$.
\end{proof}

\subsection{Twistor on tensor products}
Now let us return to the question of relating the
$\osp(1|2)$ and $\fsl(2)$ link invariants. 
Since the invariants arise from maps between tensor products
of simple modules and their duals, 
we shall also need variants of the twistor maps
on the corresponding $\hatU$-modules. In the following, we shall define a number of
versions of $\Tw$ in different settings. However, they will all be compatible
in natural ways, so rather than label these maps differently, we shall
treat them en suite as an operator on $\hatU$ and its modules.

The following proposition takes the first step
in this direction by showing that there is a natural extension
of the twistor maps to tensor powers of $\UU$.

\begin{prop}\label{prop:twandcoprod}
For each positive integer $t$,
there exists a $\Q(\bt)$-algebra automorphism $\Tw$ of
$\hatU^{\otimes t+1}$
which satisfies \[\Tw(x\otimes y)=\Tw(x)\Delta^{s}(\Upsilon_{|y|})\otimes\Delta^{s'}(\tT_{|x|}\Upsilon_{|x|})\Tw(y)\] 
for any positive integers $s,s'$ satisfying $s+s'=t+1$,
$x\in \hatU^{\otimes s}$, and $y\in\hatU^{\otimes s'}$.
Moreover, $\Delta^{t}(\Tw(x))=\Tw(\Delta^{t}(x))$ for any $x\in \UU$.
\end{prop}
\begin{proof}
Define $\Tw':\hatU^{\otimes t+1}\rightarrow \hatU^{\otimes t+1}$ as follows:
for $x=\bigotimes_{s=1}^{t+1} x_s\in\hatU^{\otimes t+1}$, let $\Tw(x)=\bigotimes_{s=1}^{t+1}\Tw(x)_s$
where 
\begin{equation}\label{eq:twistortensorfactors}
\Tw(x)_s=\tT_{|x_1|+\ldots +|x_{s-1}|}\Upsilon_{|x_1|+\ldots +|x_{s-1}|}\Tw(x_s)\Upsilon_{|x_{s+1}|+\ldots +|x_{t+1}|}.
\end{equation}
It is elementary to check that \[\Tw(x\otimes y)=\Tw(x)\Delta^{s}(\Upsilon_{|y|})\otimes\Delta^{s'}\tT_{|x|}\Upsilon_{|x|}\Tw(y)\] 
for any positive integers $s,s'$ satisfying $s+s'=t+1$,
$x\in \hatU^{\otimes s}$, and $y\in\hatU^{\otimes s'}$.
Moreover, since $\Tw$ on $\hatU$ is a bijection,
it is easy to see that so is $\Tw$ on $\hatU^{\otimes t+1}$.

We will prove that $\Tw$ is an isomorphism by induction.
Since $\Tw$ on $\hatU$ is an isomorphism, let us assume $\Tw$ on $\hatU^t$
is an isomorphism.
Then for $x,w\in \hatU^{\otimes t}$ and $y,z\in\hatU$,
\begin{align*}
\Tw(x\otimes y)\Tw(w\otimes z)&=(\Tw(x)\Upsilon_{|y|}\otimes\tT_{|x|}\Upsilon_{|x|}\Tw(y))(\Tw(w)\Upsilon_{|z|}\otimes\tT_{|w|}\Upsilon_{|w|}\Tw(z))\\
&=\pi^{p(y)p(w)}\Tw(x)\Upsilon_{|y|}\Tw(w)\Upsilon_{|z|}\otimes\tT_{|x|}\Upsilon_{|x|}\Tw(y)\tT_{|w|}\Upsilon_{|w|}\Tw(z)\\
&=\pi^{p(y)p(w)}\bt^{\phi(|y|,|w|)-\phi(|w|,|y|)-|w|\cdot|y|}\Tw(xw)\Upsilon_{|yz|}\otimes\tT_{|xw|}\Upsilon_{|xw|}\Tw(yz)\\
&=\pi^{p(y)p(w)}\bt^{2p(y)p(z)}\Tw(xw\otimes yz)=\Tw(\pi^{p(y)p(w)}xw\otimes yz)\\
&=\Tw((x\otimes y)(w\otimes z))
\end{align*}
This completes the induction showing $\Tw$ on $\UU^{t+1}$ is an isomorphism
as claimed. Finally, showing that $\Tw$ commutes with $\Delta^{t}$
is straightforward using \eqref{eq:twistortensorfactors} and checking
on the generators.
\end{proof}
Now that we have a viable twistor map on tensor powers of $\hatU$,
we need an analogue on the tensor powers of modules. 
In particular, suppose we have a collection of $\hatU$ modules which are canonical or anticanonical,
and which carry twistors. We will produce a twistor on the tensor product 
of these modules.

As might be suggested by \eqref{eq:twistortensorfactors}, this
is not as simple as taking the tensor power of the twistors.
A version of such a twistor is produced in \cite[Proposition 6.11]{C}
by rescaling the tensor product of twistors by a power of $\bt$ given
by a function of the weights of the tensor factors.
We will do something similar,
but it turns out that we will need functions which depend not only
on the weights of tensor factors but also their parities, as well as the signature of the tensor product.

\begin{lem}
Let $c=(c_1,c_2)$ where $c_1,c_2\in \set{1,-1}$.
There exists a function $\kappa_c:\bigrset^2\rightarrow \Z$
satisfying $\kappa((0,0),\zeta)\equiv \kappa(\zeta,(0,0))\equiv 0$ modulo 4 and
\[\kappa_c(\zeta+\mu,\zeta'+\nu)-\kappa_c(\zeta,\zeta')\equiv
\ang{\tilde\mu,|\zeta'|}+c_2\phi(\mu,c_2|\zeta'|)+2p(\zeta)p(\nu)+c_1\phi(\nu,c_1|\zeta|)+\mu\cdot\nu+\phi(\mu,\nu)\mod 4\]
for all $\zeta,\zeta'\in \bigrset$ and $\mu,\nu\in\Z[I]$.
\end{lem}
\begin{proof}
Fix $c=(c_1,c_2)$ where $c_1,c_2\in\set{1,-1}$.
Note that it suffices to show such a function $\kappa=\kappa_c$ 
exists on each coset of $\Z[I]\times \Z[I]$ 
(where as in \eqref{eq:ZI in X hat}, we view $\Z[I]$ as a subset of $\bigrset$), 
so fix a set of representatives $C$ of $\bigrset/\Z[I]$.
For $\zeta_0,\zeta_1\in C$, set \[\kappa(\zeta_0+\mu,\zeta_1+\nu)
=\ang{\tilde\mu,|\zeta_1|}+c_2\phi(\mu,c_2|\zeta_1|)+2p(\zeta_0)p(\nu)+c_1\phi(\nu,c_1|\zeta_0|)+\mu\cdot\nu+\phi(\mu,\nu).\]
It is elementary to verify that this has the desired properties.
\end{proof}

We henceforth suppose we have fixed choices of $\kappa_c$ for each $c\in \set{1,-1}^2$.
We can extend $\kappa$ naturally to larger powers of $\bigrset$.
Let $t>1$ be a positive integer and fix a sequence $c=(c_s)\in\set{\pm 1}^t$.
Let $\kappa_c:\bigrset^t\rightarrow \Z$ be the function defined by
\[\kappa_c(\zeta)=\sum_{1\leq r<s\leq t} \kappa_{(c_r,c_s)}(\zeta_r,\zeta_s),
\quad \zeta=(\zeta_s)\in\bigrset^t.\]
Then if
$\zeta=(\zeta_s), \zeta'=(\zeta'_s)\in \bigrset^t$ with $\zeta'_s=\zeta_s+\delta_{r,s}i$
for some $1\leq r\leq t$, then
\[\kappa(\zeta')-\kappa(\zeta)
=\sum_{r<s\leq t}\parens{\ang{\tilde i,|\zeta_{s}|}+c_s\phi(i,c_s|\zeta_{s}|)}+
\sum_{1\leq s'<r}\parens{2p(\zeta_{s'})p(i)+c_{s'}\phi(i,c_{s'}|\zeta_{s'}|)}\mod 4.\]

We can observe some convenient properties of the maps $\kappa_c$.
\begin{lem} \label{lem:kappa props}
Let $c=(c_s)\in\set{\pm 1}^t$ and $\zeta=(\zeta_s),\zeta'=(\zeta'_s)\in \bigrset^t$. 
\begin{enumerate}
\item Let $1\leq r\leq t$, and define $c_{\leq r}= (c_1,\ldots, c_r)$,
$c_{>r}=(c_{r+1},\ldots, c_{t})$. Likewise, define $\zeta''_{\leq r}=(\zeta_1'',\ldots, \zeta_r'')$
and $\zeta''_{>r}=(\zeta''_{r+1},\ldots, \zeta''_t)$ for any $\zeta''=(\zeta''_s)\in \bigrset^t$. Then
\[\kappa_c(\zeta,\zeta')=\kappa_{c_{\leq r}}(\zeta_{\leq r}, \zeta'_{\leq r})+
\kappa_{c_{> r}}(\zeta_{> r}, \zeta'_{>r})
+\sum_{1\leq s<r<s'\leq t} \kappa_{(c_s,c_{s'})}(\zeta_s,\zeta'_{s'})\]
\item Suppose that there exists $1\leq r <t$ such that
$\zeta_{r}=\zeta'_{r}+\nu$, $\zeta_{r+1}=\zeta'_{r+1}-\nu$,
and $\zeta_s=\zeta_s'$ for $s\neq r,r+1$ and some $\nu\in\Z[I]$. 
Then 
\begin{equation}\label{eq:kappa opposing shifts}
\kappa_c(\zeta)-\kappa_c(\zeta')=\ang{\tilde\nu,\zeta_{r+1}}+ c_{r+1}\phi(\nu,c_{r+1}\zeta_{r+1}) +2p(\nu)p(\zeta_r)- c_{r}\phi(\nu,c_{r}\zeta_{r})-\nu\cdot\nu-\phi(\nu,\nu)
\end{equation}
\item For any $\zeta\in \bigrset$ and $c_1=\pm 1$, we have
\[\kappa_{c_1, \pm 1,\mp 1}(\zeta+\bigrelt{\nu},(\pm \lambda,0),(\mp\lambda,0))=\kappa_{c_1,\pm 1,\mp 1}(\zeta,(\pm \lambda,0),(\mp\lambda,0))\]
\[\kappa_{\pm 1,\mp 1,c_1}((\pm \lambda,0),(\mp\lambda,0), \zeta+\bigrelt{\nu})=\kappa_{\pm 1,\mp 1,c_1}((\pm \lambda,0),(\mp\lambda,0), \zeta)\]
\end{enumerate}
\end{lem}
\begin{proof}
We note that (1) is an immediate consequence of the definition of $\kappa_c$.
On the other hand, (2) and (3) both follow from direct computations and the definition.
\end{proof}

The functions $\kappa_c$ allows us to define a twistor on tensor product
modules as follows.

\begin{prop}\label{prop: twistor tensor}
Let $M_1, M_2,\ldots, M_t$ be  canonical or anti-canonical 
$\hatU$-modules carrying twistors and 
let $M=M_1\hotimes  M_2\otimes\ldots\hotimes M_t$
be the $\hatU^{\otimes t}$-module (and hence a mixed $\hatU$-module via $\Delta^{t-1}$) 
with the natural action. Set $c=\sig(M)=(c_1,\ldots, c_t)$. Then the automorphism
\[\Tw(m_1\otimes\ldots\otimes m_t)=\bt^{\kappa_{c}((\bigrdeg m_i))}\Tw(m_1)\otimes\ldots\otimes\Tw(m_t)\] 
satisfies
\[\Tw((x_1\otimes\ldots\otimes x_t)(m_1\otimes\ldots\otimes m_t))=\Tw(x_1\otimes\ldots\otimes x_t)\Tw(m_1\otimes\ldots\otimes m_t).\]
In particular, $\Tw(um)=\Tw(u)\Tw(m)$ for $u\in \hatU$ and $m\in M$.
\end{prop}
\begin{proof}
First, observe it is enough to show
\[\Tw((1^{s-1}\otimes x_s\otimes 1^{t-s})(m_1\otimes\ldots\otimes m_t))=\Tw(1^{s-1}\otimes x_s\otimes 1^{t-s})\Tw(m_1\otimes\ldots\otimes m_t)\]
where $1\leq s\leq t$ and $x_s$ is a generator of $\hatU$. 
This is trivial when $x_s$ is $K_\mu$, $J_\mu$,
$T_\mu$ and $\Upsilon_\mu$ for some $\mu\in \Z[I]$ so it suffices to check
the case $x_s=E_i,F_i$ for $i\in I$. To do this, let us make our equations
more compact with the following notations: for $m_1\otimes\ldots \otimes m_t\in M$,
let
\[m_{<s}=m_1\otimes\ldots\otimes m_{s-1},\quad m_{>s}=m_{s+1}\otimes\ldots\otimes m_t\] \[\Tw(m)_{<s}=\Tw(m_1)\otimes \ldots\otimes \Tw(m_{s-1}),\quad \Tw(m)_{>s}=\Tw(m_{s+1})\otimes \ldots\otimes \Tw(m_{t}),\]
\[\bigrdeg{m}_{<s}=(\bigrdeg{m_1},\ldots,\bigrdeg{m_{s-1}}),\qquad\bigrdeg{m}_{>s}=(\bigrdeg{m_{s+1}},\ldots,\bigrdeg{m_{t}}),\]
\[\phi'(i, m_{<s})=\sum_{1\leq r<s} c_r \phi(i, c_r|m_r|),\qquad 
\phi''(i, m_{>s})=\sum_{s< r\leq t} c_r \phi(i, c_r |m_r|).\]
Using these notations, we compute that
\begin{align*}
\Tw&((1^{s-1}\otimes E_i\otimes 1^{t-s})(m_{<s}\otimes m_s\otimes m_{>s}))=
\Tw(\pi_i^{p(m_{<s})}m_{<s}\otimes E_im_s\otimes m_{>s})\\
&=\bt^{2p(i)p(m_{<s})+\kappa_c(\bigrdeg{m}_{<s},\bigrdeg{m_{s}}+\bigrelt{i},\bigrdeg{m}_{>s})}
\pi^{p(i)p(m_{<s})}\Tw(m)_{<s}\otimes \Tw(E_im_s)\otimes \Tw(m_{>s})\\
&=\bt^{\kappa(\bigrdeg{m_{<s}},\bigrdeg{m_{s}},\bigrdeg{m_{>s}})+\phi'(i,m_{<s})+\phi''(i,m_{>s})+\ang{\tilde i,|m_{>s}|}}\\
&\hspace{4em}\times
\pi^{p(i)p(m_{<s})}\Tw(m)_{<s}\otimes \Tw(E_i)\Tw(m_s)\otimes \Tw(m_{>s})\\
&=\bt^{\kappa(\bigrdeg{m_{<s}},\bigrdeg{m_{s}},\bigrdeg{m_{>s}})}\pi^{p(i)p(m_{<s})}
\parens{\Upsilon_i^{\otimes(s-1)}\Tw(m)_{<s}}\otimes \parens{\Tw(E_i)\Tw(m_s)}\otimes \parens{(\Upsilon_i\tT_i)^{\otimes (t-s)}\Tw(m_{>s})}\\
&=\Tw(1^{s-1}\otimes E_i\otimes 1^{t-s})\Tw(m_{<s}\otimes m_s\otimes m_{>s}).
\end{align*}
The case $x_s=F_i$ proceeds similarly.
\end{proof}

We now have defined a family of compatible twistor maps on (anti-)canonical
modules and their tensor products.
Moreover, the twistor maps on tensor products of modules
are compatible with one another in the following sense.
Let $M_1,\ldots, M_t$, $c_1,\ldots, c_s$ and $M$ be as 
in Proposition \ref{prop: twistor tensor}.
Fix $1\leq r\leq t$ and set $m_{\leq r}=m_1\otimes\ldots\otimes m_r$ and 
$m_{>r}=m_{r+1}\otimes\ldots\otimes m_t$. Then by Lemma \ref{lem:kappa props}(1), 
\begin{equation}\label{eq:twistor tensor assoc}
\Tw(m_{\leq r}\otimes m_{>r})
=\displaystyle\parens{\prod_{1\leq s\leq r<s'\leq t}\bt^{ \kappa_{c_{s},c_{s'}}(\bigrdeg{m_s},\bigrdeg{m_{s'}})}}\Tw(m_{\leq r})\otimes \Tw(m_{>r})
\end{equation}

\subsection{Twisting the crossings, caps, and cups}

We have now lain the groundwork for studying the atomic maps in our
graphical calculus from \S \ref{sec:diagcalc} under the twistor functor. 
Specifically, we will show that the twistor
almost commutes with cups, caps, and crossings up to a 
factor of an integral power of $\bt$,
where the power depends on the map. 
We begin by considering the cups and caps on their
domains of definition.

\begin{prop}\label{prop: cups caps and twistorv1}
Let $\lambda\in X^+$.
Then the map $\ev_\lambda$ (respectively, $\qtr_\lambda$,  $\coev_\lambda$, and $\coqtr_\lambda$)
viewed as a function $\hV(-\lambda)\hotimes\hV(\lambda)\rightarrow \Qqtt$
(resp. $\hV(\lambda)\hotimes\hV(-\lambda)\rightarrow \Qqtt$, 
$\Qqtt\rightarrow \hV(-\lambda)\hotimes\hV(\lambda)$, and $\Qqtt\rightarrow \hV(\lambda)\hotimes\hV(-\lambda)$) is a $\hatU$-module homomorphism.
Moreover, we have
\begin{enumerate}
\item $\ev_\lambda\Tw=\bt^{\kappa_{(-1,1)}((-\lambda,0),(\lambda,0))}\Tw\ev_\lambda$;
\item $\qtr_\lambda\Tw=\bt^{\kappa_{(1,-1)}((\lambda,0),(-\lambda,0))-\ang{\tilde\rho,\lambda}}\Tw\qtr_\lambda$;
\item $\coev_\lambda\Tw=\bt^{-\kappa_{(-1,1)}((-\lambda,0),(\lambda,0))+\ang{\tilde\rho,\lambda}}\Tw\coev_\lambda$;
\item $\coqtr_\lambda\Tw=\bt^{-\kappa_{(1,-1)}((\lambda,0),(-\lambda,0))}\Tw\coqtr_\lambda$.
\end{enumerate}
\end{prop}
\begin{proof}
	First, observe that since these maps are $\UU$-module homomorphisms,
	they preserve weight spaces hence preserve the action of $T_i$ for $i\in I$.
	Therefore, it only remains to check that they commute with the action of $\Upsilon_i$ for $i\in I$,
	As the arguments are all similar, let us show this for $\ev_\lambda$.
	Let $f\in \hV(-\lambda)$ and $x\in V(\lambda)$.
	Then 
\[\Upsilon_i\ev_\lambda(f\otimes x)=\bt^{\phi(i,0)}\ev_{\lambda}(f\otimes x)=f(x).\]
	On the other hand, 
	$\Upsilon_i(f\otimes x)=(\Upsilon_if)\otimes(\Upsilon_ix)=\bt^{-\phi(i,-|f|)+\phi(i,|x|)} f\otimes x$
	hence
	\[\ev_\lambda(\Upsilon_i(f\otimes x))=\bt^{\phi(i,|x|)-\phi(i,-|f|)}f(x).\]
	However, since $f(x)=0$ if $|f|\neq -|x|$, we see that
	$\ev_\lambda(\Upsilon_i(f\otimes x)=\bt^{\phi(i,|x|)-\phi(i,|x|)} f(x)=f(x)=\Upsilon_i\ev_\lambda(f\otimes x)$.
	
	To verify (1)-(4), it suffices to compute the images 
	$\Tw(b^-v_\lambda\otimes (b^-v_\lambda)^*)$ and $\Tw((b^-v_\lambda)^*\otimes b^-v_\lambda)$
	for $b\in\cB_\nu=\cB\cap \ff_\nu$.
	We compute directly that
	\[\Tw(b^-v_\lambda)=\bt^{\ell(b)-\phi(\nu,\lambda)}b^{-}v_\lambda,\]
	\[\Twnat(b^-v_\lambda)=\bt^{\ell(b)+\ang{\nu,\lambda}+\phi(\nu,\lambda-\nu)}b^{-}v_\lambda.\]
	This implies that for any $b,b'\in \cB_\nu$, 
	\[\Tw((b^-v_\lambda)^*)(b'^-v_\lambda)
	=\bt^{2p(\nu)}\Tw((b^-v_\lambda)^*(\Twnat^{-1}(b'^-v_\lambda)))
	=\bt^{2p(\nu)-\ell(b)-\ang{\nu,\lambda}-\phi(\nu,\lambda-\nu)}\delta_{b,b'}
	\]
	and hence $\Tw((b^-v_\lambda)^*)=\bt^{2p(\nu)-\ell(b)-\ang{\nu,\lambda}-\phi(\nu,\lambda-\nu)}(b^-v_\lambda)^*$. In particular, for $c=(1,-1)$
	observe that
	\begin{align*}
	\Tw((b^-v_\lambda)\otimes(b^-v_\lambda)^*)&=\bt^{\kappa_{c}((\lambda,0)-\bigrelt{\nu},(-\lambda,0)+\bigrelt{\nu})
		+\ell(b)-\phi(\nu,\lambda)+2p(\nu)-\ell(b)-\ang{\nu,\lambda}-\phi(\nu,\lambda-\nu)}(b^-v_\lambda)\otimes(b^-v_\lambda)^*\\
	&=\bt^{\kappa_{c}((\lambda,0),(-\lambda,0))+2p(\nu)-\nu\cdot\nu}(b^-v_\lambda)\otimes(b^-v_\lambda)^*
	\end{align*}
	It is easy to verify that $\frac{\nu\cdot\nu}{2}= p(\nu)$ modulo 2 by induction,
	hence we see that
	\[
	\Tw((b^-v_\lambda)\otimes(b^-v_\lambda)^*)=\bt^{\kappa_{(1,-1)}((\lambda,0),(-\lambda,0)
	}(b^-v_\lambda)\otimes(b^-v_\lambda)^*.
	\]
	A similar computation shows that
	\[
	\Tw((b^-v_\lambda)*\otimes(b^-v_\lambda))=\bt^{\kappa_{(-1,1)}((-\lambda,0),(\lambda,0)
	}(b^-v_\lambda)^*\otimes(b^-v_\lambda).
	\]
	Note that in either case, the power of $\bt$ is independent of $b\in \cB$, and applying this to the definition of the maps proves (1) and (4). For (2) and (3), also note that
	$\pi^{p(\nu)}q^{\pm\ang{\tilde{\rho},\lambda-\nu}}=\pi_\nu q_{\nu}^{\mp 2} q^{\pm\ang{\tilde{\rho}, \lambda}}$, and we compute that
	$\Tw(\pi_\nu q_\nu^{\mp 2}q^{\pm\ang{\tilde{\rho}, \lambda}})=\bt^{\mp\ang{\tilde{\rho},\lambda}} \pi_\nu q_\nu^{\mp 2}q^{\pm\ang{\tilde{\rho}, \lambda}}$, the result follows. 
\end{proof}
\begin{example}
	Consider the case $n=1$ and $\lambda=m$. As noted in Example \ref{ex:rk1evcoev}, $\ang{\tilde{\rho},\lambda}=m$ and $\ev_m\circ\coev_m=\pi^m[m+1]$. 
	Then we have $\ev_m\circ\coev_m\circ\Tw(1)=\pi^m[m+1]$, and
	\[\Tw\circ\ev_m\circ\coev_m(1)=\Tw(\pi^m[m+1])=\bt^{-m}\pi^m[m+1]=\bt^{-m}\ev_m\circ\coev_m\circ\Tw(1).\]
	Note that this is consistent with Proposition \ref{prop: cups caps and twistorv1}, as we see that 
	\[\ev_m\circ\coev_m\circ\Tw=\bt^{-\kappa_{(-1,1)}((-\lambda,0),(\lambda,0))+\ang{\tilde\rho,\lambda}}\ev_m\circ\Tw\circ\coev_m=\bt^{m}\Tw\circ\ev_m\circ\coev_m.\]
\end{example}
The last elementary diagram to consider is the crossing, which is to say the automorphism
$R=\Theta\frF\frs$ of a tensor product of two modules. In order to have a concrete comparison
of $R\Tw$ and $\Tw R$ on tensor products carrying twistors, it will be necessary to have
a precise description of $\Tw(\frf(\zeta,\eta))$ for any $\zeta,\eta\in X$.
To that end, let us once and for all fix a transversal $T$ of $X/\Z[I]$ and note that 
$\hat T=\set{(\zeta,0),(\zeta,1)\mid \zeta\in T}$ is a transversal of $\hat X/\Z[I]$.  
Then for $\zeta_0,\zeta_1\in T$, we shall henceforth require that

\begin{equation}\label{eq:f set}
\mathfrak f(|\zeta_0|,|\zeta_1|)=1.
\end{equation}

Then we have the following proposition.
\begin{prop}\label{prop:twistor vs R v1}
	Let $\lambda,\lambda'\in X^+\cup -X^+$.
	Let $\hat\zeta,\hat\zeta'\in \hat T$ be the corresponding coset representatives
	of $(\lambda,0)$ and $(\lambda',0)$ in $\bigrset/\Z[I]$
	and let $(c_1,c_2)=\sig(\hV(\lambda)\hotimes \hV(\lambda'))$.
	Let $\cR:\hV(\lambda)\hotimes \hV(\lambda')\rightarrow\hV(\lambda')\hotimes \hV(\lambda)$
	be the map described in Proposition \ref{prop:Rmat}. Then $\cR$ is a $\hatU$-module homomorphism.
	Moreover, as maps on $\hV(\lambda)\hotimes \hV(\lambda')$, we have 
	\[\Tw\mathcal R=\bt^{\kappa_{(c_2,c_1)}(\hat\zeta',\hat\zeta)-\kappa_{(c_1,c_2)}(\hat\zeta,\hat\zeta')+2p(\hat\zeta)p(\hat\zeta')}\mathcal R\Tw.\]
\end{prop}
\begin{proof}
	Recall that $\cR=\Theta\frf\frs$ by definition.
	It is easy to see that $\cR$ is a $\hatU$-module homomorphism: indeed,
	since $\cR$ preserve weight-spaces, it commutes with the action of the $T_i$ for $i\in I$;
	moreover, $\frf\frs$ obviously commutes with the diagonal action of $\Upsilon_i$, and 
	it is easy to check directly that $\Theta_\nu\Delta(\Upsilon_i)=\Delta(\Upsilon_i)\Theta_\nu$.
	We will prove the remainder of the proposition in two steps.
	
	First we shall show that
	$\Tw(\Theta_\nu)=\Theta_\nu$ for any $\nu\in \Z_{\geq 0}[I]$,
	and thus $\Tw\Theta=\Theta\Tw$ as maps on $V(\lambda)\otimes V(\lambda')$.
	This is straightforward: applying Lemmas \ref{lem:tw+alt}, \ref{lem:twondualbasis} 
	and Proposition \ref{prop:twandcoprod} to the expression for $\Theta_\nu$
	in terms of the canonical basis $\cB$, we compute that
	\begin{align*}
	\Tw(\Theta_\nu)&=
	(-1)^{\height\,\nu} \bt^{2\bp(\nu)}\pi^{\bp(\nu)}\bt_\nu^2\pi_\nu \bt_\nu^{-1}q_\nu
	\sum_{b\in \cB_\nu} \Tw(b^-)\Upsilon_{\nu}\otimes \tT_{-\nu}\Upsilon_{-\nu}\Tw((b^*)^+)\\
	&=
	(-1)^{\height\,\nu+\bp(\nu)}\pi^{\bp(\nu)}\bt_\nu\pi_\nu q_\nu
	\sum_{b\in \cB_\nu} (\Tw(b)^-\Upsilon_{-\nu})\Upsilon_{\nu}\otimes \tT_{-\nu}\Upsilon_{-\nu}(\bt_\nu^{-1} \tT_{\nu}\Upsilon_{\nu}\Tw(b^*)^+)\\
	&=
	(-1)^{\height\,\nu+\bp(\nu)}\pi^{\bp(\nu)}\bt_\nu\pi_\nu q_\nu
	\sum_{b\in \cB_\nu} (\bt^{\ell(b)}b^-)\otimes (\bt_\nu^{-1} \bt^{-\ell(b)}(-1)^{\bp(\nu)}(b^*)^+)\\
	&=
	(-1)^{\height\,\nu}\pi^{\bp(\nu)}\pi_\nu q_\nu
	\sum_{b\in \cB_\nu} b^-\otimes (b^*)^+=\Theta_\nu
	\end{align*}
	
	Now it remains to show that we have 
	$\Tw\mathfrak F\frs=\bt^{\kappa_{(c_2,c_1)}(\hat\zeta',\hat\zeta)-\kappa_{(c_1,c_2)}(\hat\zeta,\hat\zeta')+2p(\hat\zeta)p(\hat\zeta')}\mathfrak F\frs\Tw$ as maps on $V(\lambda)\otimes V(\lambda')$.
	Set $c=(c_1,c_2)$, and $\tilde c=(c_2,c_1)$.
	Let $m\in V(\lambda)$ and $n\in V(\lambda')$. Then we see directly that 
	\[\Tw\mathfrak F\frs(m\otimes n)=t^{2p(m)p(n)+\kappa_{\tilde c}(\bigrdeg{n},\bigrdeg{m})}\Tw(\mathfrak{f}(|n|,|m|))\pi^{p(m)p(n)}\Tw(n)\otimes \Tw(m),\]
	\[\mathfrak F\frs\Tw(m\otimes n)=t^{\kappa_c(\bigrdeg{m},\bigrdeg{n})}\mathfrak{f}(|n|,|m|)\pi^{p(m)p(n)}\Tw(n)\otimes \Tw(m).\]
	The proposition then follows by verifying that
	\[t^{2p(m)p(n)+\kappa_{\tilde c}(\bigrdeg{n},\bigrdeg{m})}\Tw(\mathfrak{f}(|n|,|m|))=t^{\kappa_{\tilde c}(\hat\zeta',\hat\zeta)-\kappa_{c}(\hat\zeta,\hat\zeta')+2p(\hat\zeta)p(\hat\zeta')+\kappa_c(\bigrdeg{m},\bigrdeg{n})}\mathfrak{f}(|n|,|m|),\]
	Note that $\hat\zeta=\bigrdeg{m}+\mu$ and $\hat\zeta'=\bigrdeg{n}+\nu$ for some $\mu,\nu\in\Z[I]$. Let $\zeta=|\hat\zeta|\in X$ and $\zeta'=|\hat\zeta'|\in X$.
	Then in particular, \eqref{eq:f set} implies
	\[\mathfrak{f}(|n|,|m|)=(\pi q)^{\ang{\tilde \nu,\zeta}}q^{\ang{\tilde \mu,\zeta'}-\mu\cdot\nu},\]
	so $\Tw(\mathfrak{f}(|n|,|m|))=\bt^{\ang{\tilde{\nu},\zeta}-\ang{\tilde\mu,\zeta'}+\mu\cdot\nu}\mathfrak{f}(|n|,|m|)$. Therefore, we are reduced to showing that $\ell\equiv r$ modulo $4$, where
	\[\ell=2p(m)p(n) +\ang{\tilde\nu,\zeta}-\ang{\tilde{\mu},\zeta'}+\mu\cdot\nu+\kappa_{\tilde c}(\bigrdeg{n},\bigrdeg{m}),\]
	\[r= 2p(\hat\zeta)p(\hat\zeta')+\kappa_{\tilde{c}}(\hat\zeta',\hat\zeta)-\kappa_{c}(\hat\zeta,\hat\zeta')+\kappa_c(\bigrdeg{m},\bigrdeg{n}) \mod 4\]
	
	We compute directly that
	\begin{align*}
	\kappa_{\tilde c}(\bigrdeg{n},\bigrdeg{m})&-\kappa_{\tilde{c}}(\hat\zeta',\hat\zeta)+\kappa_{c}(\hat\zeta,\hat\zeta')-\kappa_c(\bigrdeg{m},\bigrdeg{n})\\
	&=\kappa_{\tilde c}(\hat\zeta'-\nu,\hat\zeta-\mu)-\kappa_{\tilde{c}}(\hat\zeta',\hat\zeta)+\kappa_{c}(\hat\zeta,\hat\zeta')-\kappa_c(\hat\zeta-\mu,\hat\zeta'-\nu)\\
	&\equiv_4-\ang{\tilde\nu,\zeta}-c_1\phi(\nu,c_1\zeta)+2p(\hat\zeta)p(\mu)-c_2\phi(\mu,c_2\zeta) +\mu\cdot\nu+\phi(\mu,\nu)\\
	&\hspace{2em}+\ang{\tilde\mu,\zeta}+c_1\phi(\mu,c_1\zeta)-2p(\hat\lambda)p(\nu)+\phi(\nu,\lambda) -\mu\cdot\nu-\phi(\nu,\mu)\\
	&\equiv_4 2p(\hat\lambda)p(\nu)+2p(\hat\zeta)p(\mu)+2p(\mu)p(\nu)-\ang{\tilde\nu,\lambda}+\ang{\tilde\mu,\zeta} +\mu\cdot\nu\\
	&\equiv_4 2p(m)p(n)-2 p(\hat\lambda)p(\hat\zeta)-\ang{\tilde\nu,\lambda}+\ang{\tilde\mu,\zeta} +\mu\cdot\nu,
	\end{align*}
	where here $\equiv_4$ denotes equivalence modulo 4.
	This finishes the proof.
\end{proof}

We have seen that the twistor map commutes (up to an integral power of $\bt$) 
with the elementary functions in our graphical calculus.
However, note that in Theorem \ref{thm:knot invariant}, the typical composand of a tangle invariant
is not just one of these maps, but in fact is a tensor product of these maps with various identities.
It is important to note that a consequence of Proposition \ref{prop: twistor tensor}
is that the twistor maps on tensor products are not local, since the power of $\bt$ in the construction
depends on the weight and signature of each tensor factor.
Nevetheless, we can extend Propositions \ref{prop: cups caps and twistorv1} 
and \ref{prop:twistor vs R v1} 
to this more general setting.

\begin{prop}\label{prop:twistor vs cups, caps, R}
Let $M_1,\ldots, M_t$ be $\hatU$-modules such that for each $1\leq s\leq t$, 
$M_s=\hV(\mu_s)$ for some $\mu_s\in X^+\cup-X^+$.
Let $M=M_1\hotimes\ldots \hotimes M_t$ 
and let $c=(c_1,\ldots, c_t)=\sig(M)$. 
For any $\lambda\in X^+$ and $0\leq r\leq t$, we define
$M_{\leq r}=M_1\hotimes\ldots\hotimes M_r$, $M_{>r}=M_{r+1}\hotimes \ldots \hotimes M_t$,
and
\[M(r,\pm \lambda)=M_{\leq r}\hotimes \hV(\pm \lambda)\hotimes \hV(\mp \lambda)\hotimes M_{>r}.\]
\begin{enumerate}
\item Let $\cR_s=1_{M_{\leq s-1}}\otimes \cR\otimes 1_{M_{>s+1}}:M\rightarrow M$ 
for some $1\leq s\leq t-1$.
Then as maps on $M$, $\Tw R_s$ and $R_s\Tw$ are proportional up to an integral power of $\bt$.
\item Let $\ev(M,r,\lambda)=1_{M_{\leq r}}\otimes \ev_\lambda\otimes 1_{M_{>r}}$
for some $1\leq r\leq t$. Then
as maps on $M(r,-\lambda)$, $\Tw\ev(M,r,\lambda)$ and $\ev(M,r,\lambda)\Tw$ are
proportional up to an integral power of $\bt$.
\item Let $\qtr(M,r,\lambda)=1_{M_{\leq r}}\otimes \qtr_\lambda\otimes 1_{M_{>r}}$
for some $1\leq r\leq t$. Then
as maps on $M(r,\lambda)$, $\Tw\qtr(M,r,\lambda)$ and $\qtr(M,r,\lambda)\Tw$ are
proportional up to an integral power of $\bt$.
\item Let $\coev(M,r,\lambda)=1_{M_{\leq r}}\otimes \coev_\lambda\otimes 1_{M_{>r}}$
for some $1\leq r\leq t$. Then
as maps on $M$, $\Tw\coev(M,r,\lambda)$ and $\coev(M,r,\lambda)\Tw$ are
proportional up to an integral power of $\bt$.
\item Let $\coqtr(M,r,\lambda)=1_{M_{\leq r}}\otimes \coqtr_\lambda\otimes 1_{M_{>r}}$
for some $1\leq r\leq t$.
Then as maps on $M$, $\Tw\coqtr(M,r,\lambda)$ and $\coqtr(M,r,\lambda)\Tw$ are
proportional up to an integral power of $\bt$.
\end{enumerate}
\end{prop}

\begin{rmk}
The precise constants of proportionality can be determined directly as in
Propositions \ref{prop: cups caps and twistorv1} and \ref{prop:twistor vs R v1} 
(and can be worked out from the following proof),
but we leave them out of the statement of Proposition \ref{prop:twistor vs cups, caps, R}
because they are not particularly illuminating, and are not necessary
for Theorem \ref{thm:twistor vs knot invariant}
\end{rmk}
\begin{proof}
As the proofs of (2)-(5) are similar, we shall only prove (1) and (2) here in detail.

We will begin with the proof of (1), which is
essentially the same as the proof of Proposition \ref{prop:twistor vs R v1}.
To wit, we first observe that for any $a,b\geq 0$ and $\nu\in\N[I]$,
\[\Tw(1^{\otimes a}\otimes \Theta_\nu\otimes 1^{\otimes b})
=(\Upsilon_{|\Theta_\nu|})^{\otimes{a}} \otimes \Tw(\Theta_\nu)\otimes(\Upsilon_{|\Theta_\nu|}\tT_{|\Theta_\nu|})^{\otimes b}
=,\]
and the result follows from the observation that $|\Theta_\nu|=\nu-\nu=0$.
Then $\Tw R_s=(1^{\otimes s-1}\otimes \Theta\otimes 1^{t-s-1})\Tw \frF_s\frs_s$. 
Then we verify directly that 
$\Tw\frF_s\frs_s=\bt^{\kappa_{(c_{s+1},c_s)}(\hat\zeta',\hat\zeta)-\kappa_{(c_s,c_{s+1})}(\hat\zeta,\hat\zeta')+2p(\hat\zeta)p(\hat\zeta')}\frF_s\frs_s\Tw$
where $\hat\zeta$ (resp. $\hat\zeta'$) is the coset representative
for $(\mu_s,0)$ (resp. $(\mu_{s+1},0)$).

Now, we shall prove (2). Note that an arbitrary element of
$M(r,-\lambda)$ is a linear combination of simple tensors of the form 
$x=m_{\leq r}\otimes (b^-v_\lambda)^*\otimes (b'^-v_\lambda)\otimes m_{>r}$, 
where $b,b'\in\cB$, $m_{\leq r}=m_1\otimes\ldots\otimes m_r \in M_{\leq r}$ and
$m_{>r}=m_{r+1}\otimes\ldots\otimes m_t\in M_{>r}$, hence we need only
prove (1) holds when evaluating both sides at such elements. 
Since $\ev_\lambda((b^-v_\lambda)^*\otimes (b'^-v_\lambda))=\delta_{b,b'}$, 
note that (1) is trivially true when $b\neq b'$, so let's assume $b=b'\in\cB_\nu$.
Then
\[\ev(M,r,\lambda)\Tw(x)=\bt^{\diamondsuit(m_1,\ldots, m_t)+\clubsuit}\Tw(m_{\leq r})\otimes \ev_\lambda\Tw((b^-v_\lambda)^*\otimes (b^-v_\lambda))\otimes\Tw(m_{> r})\]
where we set 
\[\diamondsuit(m_1,\ldots, m_t)=\sum_{s<r<s'}\kappa_{(c_s,c_s')}(\bigrdeg{m_s},\bigrdeg{m_s'})\] 
\begin{align*}
\clubsuit=&\sum_{s<r}\parens{\kappa_{(c_s,-1)}(\bigrdeg{m_s},(-\lambda,0)+\bigrelt{\nu})+\kappa_{(c_s,1)}(\bigrdeg{m_s},(\lambda,0)-\bigrelt{\nu})}
\\&+\sum_{s>r}\parens{\kappa_{(-1,c_s)}((-\lambda,0)+\bigrelt{\nu},\bigrdeg{m_s})+\kappa_{(1,c_s)}((\lambda,0)-\bigrelt{\nu},\bigrdeg{m_s})}
\end{align*}
Now $\clubsuit$ can be simplified. Note that 
\[\kappa_{(c_s,-1)}(\bigrdeg{m_s},(-\lambda,0)+\nu)
=\kappa_{(c_s,-1)}(\bigrdeg{m_s},(-\lambda,0))+2c_s\phi(\nu,|m_s|)+2p(\nu)p(m_s)\]
\[\kappa_{(c_s,1)}(\bigrdeg{m_s},(\lambda,0)-\nu)
=\kappa_{(c_s,1)}(\bigrdeg{m_s},(\lambda,0))-2c_s\phi(\nu,|m_s|)+2p(\nu)p(m_s)\]
hence $\kappa_{(c_s,-1)}(\bigrdeg{m_s},(-\lambda,0)+\nu)+\kappa_{(c_s,1)}(\bigrdeg{m_s},(\lambda,0)-\nu)=\kappa_{(c_s,-1)}(\bigrdeg{m_s},(-\lambda,0))+\kappa_{(c_s,1)}(\bigrdeg{m_s},(\lambda,0))$.
Moreover, note that $\bigrdeg{m_s}=(\mu_s,0)+\nu_s$ 
for some $\nu_s\in\Z[I]$, and so
\[\kappa_{(c_s,-1)}(\bigrdeg{m_s},(-\lambda,0))
=\kappa_{(c_s,-1)}((\mu_s,0),(-\lambda,0))-\ang{\tilde\nu_s,\lambda}-2\phi(\nu_s,\lambda),\]
\[\kappa_{(c_s,1)}(\bigrdeg{m_s},(\lambda,0))
=\kappa_{(c_s,1)}((\mu_s,0),(\lambda,0))+\ang{\tilde\nu_s,\lambda}+2\phi(\nu_s,\lambda),\]
hence 
\[\kappa_{(c_s,-1)}(\bigrdeg{m_s},(-\lambda,0))+\kappa_{(c_s,1)}(\bigrdeg{m_s},(\lambda,0))=\kappa_{(c_s,-1)}((\mu_s,0),(-\lambda,0))+\kappa_{(c_s,1)}((\mu_s,0),(\lambda,0))\]
Similar applies to the sum over $s>r$ in $\clubsuit$, hence
\begin{align*}
\clubsuit=&\sum_{s<r}\parens{\kappa_{(c_s,-1)}((\mu_s,0),(-\lambda,0))+\kappa_{(c_s,1)}((\mu_s,0),(\lambda,0))}
\\&+\sum_{s>r}\parens{\kappa_{(-1,c_s)}((-\lambda,0),(\mu_s,0))+\kappa_{(1,c_s)}((\lambda,0),(\mu_s,0))}
\end{align*}
Note that $\clubsuit$ is independent of $x$.
Then
\begin{align*}
\ev(M,r,\lambda)\Tw(x)&=\bt^{\diamondsuit(m_1,\ldots, m_t)+\clubsuit}\Tw(m_{\leq r})\otimes \ev_\lambda\Tw((b^-v_\lambda)^*\otimes (b^-v_\lambda))\otimes\Tw(m_{> r})\\
&=\bt^{\diamondsuit(m_1,\ldots, m_t)+\clubsuit+\kappa_{(-1,1)}((-\lambda,0),(\lambda,0))}\Tw(m_{\leq r})\otimes\Tw(m_{> r})\\
&=\bt^{\clubsuit+\kappa_{(-1,1)}((-\lambda,0),(\lambda,0))}\Tw(m_{\leq r}\otimes m_{> r})
\end{align*}
Since $\Tw(m_{\leq r}\otimes m_{> r})=\Tw(\ev_\lambda(x))$ and the exponent of $\bt$ is independent of $x$, 
this completes the proof of (2).
\end{proof}

We now arrive at the final result of this paper.
\begin{thm} \label{thm:twistor vs knot invariant}
Let $K$ be any oriented knot, and let $J_K^\lambda(q,\tau)\in\Qqtt$ 
be the $\lambda$-colored
knot invariant defined in Theorem \ref{thm:knot invariant}. Let 
${}_{\so}J_K^\lambda(q)=J_K^\lambda(q,1)$ and 
${}_{\osp}J_K^\lambda(q)=J_K^\lambda(q,\bt)$. 
Then 
\[{}_{\osp}J_K^\lambda(q)=\bt^{\star(K,\lambda)}{}_{\so}J_K^\lambda(\bt^{-1} q),\]
for some
$\star(K,\lambda)\in\Z$.
\end{thm}

\begin{proof}
Let $J=J_K^\lambda(q,\tau)$.
First, observe that $J$ can be thought of as a function
$\Qqtt\rightarrow \Qqtt$, and in that spirit
$\Tw(J)$ is $\Tw\circ J(1)$. On the other hand,
$J=W_K\circ S$, where $W_K=(\frf(\lambda,\lambda)^{-1}\pi^{P(\lambda)}q^{\ang{\rho,\lambda}})^{\wr(K)}$
(interpreted as a function $\Qqtt\rightarrow \Qqtt$)
and $S$ is a slice diagram of $K$ 
interpreted as a composition of morphisms as described in Section
\ref{sec:diagcalc} (with strands colored by $\lambda$).
In particular, observe that by \eqref{eq:f set}
we have $\Tw(\frf(\lambda,\lambda))=\bt^{x}\frf(\lambda,\lambda)$ 
for some $x\in \Z$
depending on the coset representative of $\lambda$ in $X/\Z[I]$,
and that $\Tw(\pi^{P(\lambda)}q^{\ang{\rho,\lambda}})=\pi^{P(\lambda)}q^{\ang{\rho,\lambda})}$.
Then in particular we see that $\Tw W=\bt^{-x\wr(K)} W\Tw$.

Likewise, note that $S$ can be written as 
a composition of maps of the form
$\ev(M,r,\lambda)$, $\coev(M,r,\lambda)$, $\qtr(M,r,\lambda)$, $\coqtr(M,r,\lambda)$, and $R_s:M\rightarrow M$  for various $r,s\in\N$
with all notations being the same as in Proposition \ref{prop:twistor vs cups, caps, R}.
In particular, we see that
$\Tw\circ S=\bt^{y}S\circ \Tw$ for some $y\in \Z$,
and thus 
\[\Tw(J)=\Tw\circ c\circ S(1)=\bt^{-x\wr(K)+y}c\circ S\circ \Tw(1)=\bt^{-x\wr(K)+y}J.\]

On the other hand, observe that $\Tw(J_K^\lambda(q,\tau))=J_K^\lambda(\bt^{-1}q,\bt^{-1}\tau)$, and so
\[\bt^{y-x\wr(K)}J_K^\lambda(\bt^{-1}q,\bt^{-1}\tau)=J_K^\lambda(q,\tau).\]
The theorem follows from specializing $\tau=\bt$.
\end{proof}

\begin{rmk}
Note that since $_{\so}J^\lambda_K(q)\in\Z[q,q^{-1}]$, 
Theorem \ref{thm:twistor vs knot invariant} implies that
(after a renormalization) $_{\osp}J^\lambda_K(q)={}_{\so}J^\lambda_K(v)\in\Z[v,v^{-1}]$
where $v=q\bt^{-1}$. Furthermore, note that when $n$ or
$\ang{{\rf n},\lambda}$ is even, $_{\osp}J^\lambda_K(q)\in \Qq$
(cf. Remark \ref{rmk:red diagrams} (1)),
thus in this case $_{\osp}J^\lambda_K(q)\equiv {}_{\so}J^\lambda_K(q)$
modulo $2$.
\end{rmk} 

\end{document}